\documentclass[11pt]{amsart}

\usepackage[T1]{fontenc}
\usepackage[utf8]{inputenc}
\usepackage{amsmath,amssymb,amsthm}
\usepackage{booktabs,longtable,array,tabularx}
\usepackage{enumitem}
\usepackage[hyphens]{url}
\usepackage{placeins}
\usepackage[margin=1.08in]{geometry}
\usepackage{xcolor}
\usepackage{microtype}
\usepackage{fvextra}
\usepackage{graphicx}
\usepackage{seqsplit}
\usepackage[hidelinks,hypertexnames=false]{hyperref}
\hypersetup{colorlinks=true,linkcolor=blue,citecolor=green,urlcolor=navyblue}
\providecommand{\tightlist}{%
\setlength{\itemsep}{0pt}\setlength{\parskip}{0pt}}

\newtheorem{theorem}{Theorem}[section]
\newtheorem{lemma}[theorem]{Lemma}

\numberwithin{equation}{section}

\theoremstyle{definition}

\theoremstyle{remark}

\newcommand{\Irr}{\operatorname{Irr}}

\newcommand{\cQ}{\mathcal Q}
\newcommand{\file}[1]{\url{#1}}
\newcommand{\GG}{{\widetilde G}}
\newcolumntype{Y}{>{\raggedright\arraybackslash}X}

\newcounter{proofpart} 
\renewcommand{\theproofpart}{\arabic{proofpart}}
\AtBeginEnvironment{proof}{\setcounter{proofpart}{0}}
\newcommand{\proofpart}[2][Step]{%
  \refstepcounter{proofpart}%
  \noindent{\normalfont\bfseries #1\ \theproofpart. #2. }
  \ignorespaces
}

\title[Thompson's conjecture for groups of Lie type over small fields]
{Thompson's Conjecture for Finite Simple Groups of Lie Type
over Small Fields}
\date{}

\author{Jun Liao}
\address{School of Mathematics and Statistics, Hubei University, 
Wuhan, $430062$, P. R. China}
\email{jliao@hubu.edu.cn}

\author{Lizhong Wang}
\address{School of Mathematical Sciences, Peking University, 
Beijing, $100871$, P. R. China}
\email{lwang@math.pku.edu.cn}

\author{Jiping Zhang}
\address{School of Mathematical Sciences, Peking University, 
Beijing, $100871$, P. R. China}
\email{jzhang@pku.edu.cn}

\begin{document}
\begin{abstract}
We prove that every finite  simple group of Lie type  over a field of order at most 8 contains a conjugacy class whose square is the whole
group.  
Together with the previously established cases, our results complete the proof of Thompson's conjecture.
We explicitly construct the required class, address all low-rank identifications, and provide the precise computational certificates used in the exceptional character calculations.

\end{abstract}
\maketitle
\setcounter{tocdepth}{1}
\tableofcontents
\clearpage

\section{Introduction}
A conjecture of Thompson states that every finite non-abelian simple group $G$ contains a conjugacy class $C$ satisfying $C^2=G$, equivalently, every element of $G$ is the product of two elements in $C$.
This can be viewed as a strong version of Ore's conjecture, which asserts that every element of a finite non-abelian simple group is a commutator and is now a theorem due to Liebeck, O'Brien, Shalev, and Tiep ~\cite[Theorem 1]{LOST}.

Thompson's conjecture was first established for the alternating groups by Hsü ~\cite{H} in 1965, and subsequent work on covering $A_n$ by squares of conjugacy classes was carried out in \cite{AH,BrL}. For the sporadic groups, the conjecture was verified computationally in 1984 by Neub\"user, Pahlings, and Cleuvers ~\cite{NPaCI}.

For the projective special linear groups $\mathrm{PSL}_n(K)$, the conjecture was first established by Brenner~\cite{Brenner} in 1983 for finite fields $K$ with $|K|>n+1$; the same bound was  obtained independently by Sourour~\cite{Sourour} in 1986. The definitive result is due to  Lev~\cite{Lev94}, who proved the conjecture for $\mathrm{PSL}_n(K)$ over arbitrary fields.
For other families of Lie type, Gow~\cite{Gow} verified  Thompson's conjecture for the symplectic groups $\operatorname{PSp}_n(K)$ under the assumptions that $\operatorname{char}(K)\neq 2$ and that $-1$ is a square in $K$.

In a series of papers culminating in 1998~\cite[Theorem 2]{EG}, Ellers and Gordeev proved Thompson's conjecture for all finite simple groups of Lie type over fields of order greater than 8, employing a powerful Gauss decomposition technique together with a theorem of Lev~\cite{Lev94} on products of cyclic similarity classes. Their main theorem is as follows.

\begin{theorem}[{\cite[Theorem 2]{EG}, Ellers--Gordeev}]
\label{thm:Ellers-Gordeev}
Let $G$ be a finite simple group of Lie type. Suppose that $q>8$ if $G$ is untwisted of type $X(q)$ or twisted of type ${}^l X_n(q^l)$, except for the families ${}^2B_2(q^2)$, ${}^2G_2(q^2)$, and ${}^2F_4(q^2)$, for which suppose that $q^2>8$. Then $G$ is the square of a conjugacy class; in particular, every element of $G$ is a commutator.
\end{theorem}

For each Lie family \(X(q)\), there exists an explicit threshold \(d_X\) such that the conjecture holds for all \(q\ge d_X\)~\cite{EG}. These thresholds are tabulated in Tables~\ref{tab:EG-untwisted}--\ref{tab:EG-twisted}; here \(q\) denotes the standard field parameter, equal to \(q\) for \(X(q)\) and \({}^l X_n(q^l)\), and to \(q^2\) for the exceptional types \({}^2B_2, {}^2G_2, {}^2F_4\).
Only the cases of finite simple groups of Lie type over fields of order at most 8 remained open.

Subsequent work has further refined and extended these results. Guralnick and Malle~\cite[Theorem 1.4]{GM} used character-theoretic and probabilistic methods to show that for every finite simple group~$G$, there exist conjugacy classes $C_1,C_2$ such that $G=C_1C_2\cup\{1\}$, a result that implies strong generation properties.  Larsen, Taylor, and Tiep~\cite[Theorem 1]{LTT} proved that for all classical groups $\mathrm{SL}_n^\epsilon(q)$, $\mathrm{Sp}_{2n}(q)$, $\mathrm{SO}_{2n+1}(q)$, and $\mathrm{SO}_{2n}^{\pm}(q)$, the square of such a class contains every element of sufficiently large support.
More recently, Larsen and Tiep have made significant asymptotic progress on Thompson's conjecture. In \cite{LT24}, they developed uniform character bounds for finite classical groups and applied these to prove the conjecture for all sufficiently large symplectic groups, unitary groups satisfying a divisibility condition on the rank, and orthogonal groups in characteristic~$2$ (see \cite[Theorem 7.7]{LT24}). In subsequent work \cite{LT25}, they established sharp character estimates for characters of low level in classical groups of types $A$, $B$, and $D$, which allowed them to prove the asymptotic Thompson conjecture for unitary groups and for orthogonal groups over fields of odd order, thereby completing the proof for all sufficiently large finite simple groups of classical type (see \cite[Theorem D]{LT25}). Furthermore, character bounds for regular semisimple elements have been employed to obtain complementary asymptotic versions of the conjecture: if the characteristic polynomial of a regular semisimple class is a product of few distinct irreducible factors and the target element has sufficiently large support, then the desired decomposition holds.

The purpose of this paper is to verify Thompson's conjecture in all remaining cases:
finite simple groups of Lie type over fields of order at most 8. 

We write $G_1$ for a subsystem group
attached to a Frobenius-stable root subsystem, and $H$ for the torus of
$G$.  The notation $\widetilde G$ denotes the universal finite group of Lie
type from which $G$ is obtained by projection.

\begin{theorem}[Thompson's conjecture for the residual small fields]
\label{thm:main}
Let $G$ be a finite simple group of Lie type over a field of order at most 8.  Then there is a conjugacy class $C\subseteq G$ such
that\[
                         C^2=G.
\]
The class $C$ may be chosen exactly as in Tables~\ref{tab:classical} and
\ref{tab:exceptional}, with its matrix, root-theoretic, toral, or ordinary
character-table representative specified in Sections
~\ref{sec:classical-sources}--\ref{sec:exceptional-frob2}.
\end{theorem}

The proof strategy proceeds by first invoking the exact Ellers--Gordeev thresholds, where for each Lie family $X$ we take the desired conjugacy class to be the one constructed in~\cite{EG} whenever $q$ meets or exceeds the threshold $d_X$. For the classical families that fall below this threshold, we explicitly construct cyclic classes via companion matrices and subsequently invoke the product theorems of Nielsen
~\cite{NielsenSp,NielsenGL} and Lev~\cite{Lev94,Lev99}. When the core subgroup $G_1$ is already covered by these matrix-product results and the action of the root-subgroup complement is fixed-point-free, a relative Gauss--Lang lifting argument allows the square to be lifted from $G_1$ to the full universal group and then down to the simple quotient. Finally, for the exceptional groups below the Ellers--Gordeev threshold, we select a real regular semisimple class in a Coxeter torus and estimate the sums in the Frobenius formula exactly by means of Deligne--Lusztig theory.

\section{Preliminaries}
\label{sec:prelimelaries}

\subsection{Frobenius sums and Gow's theorem}
Let $K=x^G$ be a conjugacy class of a finite group $G$.  Frobenius' class
multiplication formula gives, for every $g\in G$,
\begin{equation}
 \#\{(a,b)\in K\times K:ab=g\}
 =\frac{|K|^2}{|G|}
   \sum_{\chi\in\Irr(G)}
       \frac{\chi(x)^2\chi(g^{-1})}{\chi(1)}.
 \label{eq:frob}
\end{equation}
If $K=K^{-1}$, then every $\chi(x)$ is real and
$\chi(x)^2=|\chi(x)|^2$.  Thus $g\in K^2$ whenever
\begin{equation}
 F_x(g):=1+\sum_{1_G\ne\chi\in\Irr(G)}
       \frac{|\chi(x)|^2\chi(g^{-1})}{\chi(1)}\ne0.
 \label{eq:Fx}
\end{equation}
In particular, the exceptional-family calculations use the sufficient
estimate
\begin{equation}
                         |F_x(g)-1|<1.                 \label{eq:strict}
\end{equation}
Column orthogonality gives
\begin{equation}
 \sum_\chi |\chi(x)|^2=|C_G(x)|,
 \qquad |\chi(g)|\le |C_G(g)|^{1/2}.                   \label{eq:column}
\end{equation}
Reality also places $1=xx^{-1}$ in $K^2$. A conjugacy class $K$ is called a Thompson class if $G=K^{2}$. 

For groups of Lie type, the Deligne-Lusztig theory of characters provides a systematic framework for analyzing irreducible characters. It partitions characters into Lusztig series, each associated to a semisimple element in the dual group, and provides formulae for their degrees and values \cite{LT25, GM}. This theory is essential for constructing explicit irreducible characters of relatively small degrees and for deriving information on their character values, which is crucial for estimating the sums in the Frobenius formula \cite{LOST}. In particular, the Steinberg character and other unipotent characters play a key role \cite{GM}.

We use Gow's theorem~\cite[Theorem 2]{Gow} for semisimple elements.
\begin{theorem}[{\cite[Theorem 2]{Gow}}]
Let $G$ be a finite simple group of Lie type of characteristic $p$, and let $g$ be a non-identity semisimple element in $G$. Let $C_{1}$ and $C_{2}$ be any conjugacy classes of $G$ consisting of regular semisimple elements (that is, elements whose centralizers in $G$ have order relatively prime to $p$). Then $g$ is expressible as a product $xy$, where $x\in C_{1}, y\in C_{2}$.
\end{theorem}

Consequently a real regular semisimple class covers
all nonidentity semisimple targets in its square.  Gow's theorem is never
used below for a unipotent or mixed target.

For each ordinary character-table certificate, equation~\eqref{eq:frob} is
evaluated exactly as a cyclotomic number on every target class.  The same
certificate verifies inverse stability of the source and asserts that no
class-multiplication coefficient vanishes.

\subsection{The relative Gauss--Lang lemma}

The following is the common form of Ellers--Gordeev's prescribed Gauss
theorem and its Lang-map consequence used throughout the residual families.

\begin{lemma}[Relative Gauss--Lang lifting]
\label{lem:relative-gauss}
Let $\GG$ be a universal finite group of Lie type, let $G_1$ be the full
subsystem group attached to a Frobenius-stable root subsystem, and let
$H_1=H\cap G_1$.  Let $V_+$ and $V_-$ be generated by all positive and
negative root groups outside the subsystem, with central filtrations\[
 V_{\pm}=V_{\pm,1}\supseteq V_{\pm,2}\supseteq\cdots,
 \qquad [V_{\pm,r},V_{\pm,s}]\le V_{\pm,r+s}.
\]
Let $f\in HG_1$, put $C=f^{\GG}$, and let $\mathcal C_1$ be the union of
the actual $HG_1$-classes of $f$ and $f^{-1}$.  Assume:
\begin{enumerate}
\def\labelenumi{\arabic{enumi}.}
\item $f$ is real in $\GG$;
\item every member of $\mathcal C_1$ has no eigenvalue $1$ on every
nonzero grade $V_{\pm,r}/V_{\pm,r+1}$;
\item $\mathcal C_1^2\supseteq G_1\setminus Z(G_1)$;
\item either $H_1\ne Z(G_1)$, or $\mathcal C_1^2=G_1$.
\end{enumerate}
Then
\[
                  C^2\supseteq\GG\setminus Z(\GG).
\]
\end{lemma}

\begin{proof}
We prove that every noncentral element of $\GG$ lies in $C^2$.
Let $x\in \GG\setminus Z(\GG)$.
By the Gauss decomposition theorem \cite[Theorem 3]{EG}, for every $h
\in H$,  there is some $y$ conjugate to $x$ such that 
\[
y=u_1 \, h \, u_2,\qquad u_1\in U_{-}, u_2\in U_{+},\; 
\]
where $U_\pm$ are the unipotent radical subgroups. We have 
\(
u_{1}=v_{1}\tilde u_1, u_{2}=\tilde u_2v_{2}
\)
where \(v_{1}\in V_{-}, v_{2}\in V_{+}\),  \(\tilde u_1\in U_{-}\cap G_1, \tilde u_2\in U_{+}\cap G_1\). 
Let \(g=\tilde u_1h\tilde u_2\). Then 
$y=v_{1}gv_{2}$, and $g\in G_{1}$.

We claim that \(g\notin Z(G_1)\) whenever \(h\in H_1\setminus Z(G_1)\):
In fact,  we have $ Z(G_1)\leq H\cap G_{1}=H_{1}$. If \(\tilde u_1h\tilde u_2=g\in Z(G_1)\), then the uniqueness of the Gauss decomposition implies that $\tilde u_1=\tilde u_2=1$, and $h=g \in Z(G_1)$, contradicting \(h\in H_1\setminus Z(G_1)\).
Thus, in all cases of Hypothesis 4, $g\in G_1\setminus Z(G_1)$ or $g\in G_1$ with $\mathcal C_1^2=G_1$. Hypothesis~3 then yields
\(
g=\sigma_1\sigma_2
\)
for some $\sigma_1,\sigma_2\in \mathcal C_1$. Hence
\(
y=v_{1}\,\sigma_1\sigma_2\,v_{2}.
\)

We now absorb $v_{1}$ and $v_{2}$ into conjugates of $\sigma_1$ and $\sigma_2$.  For $\sigma\in \mathcal C_1$, fixed-point-freeness on each central grade in Hypothesis 2 implies that the linear map $1-\sigma$ is invertible on each graded piece $V_{\pm,r}/V_{\pm,r+1}$.  Hence the associated non-abelian Lang map
\[
L_\sigma:V_\pm\longrightarrow V_\pm,\qquad 
L_\sigma(u)=u\sigma(u)^{-1}=u\sigma u^{-1}\sigma^{-1}
\]
is bijective; this follows by induction along the central filtration.

Apply this to $\sigma_1\in \mathcal C_1$ on $V_-$.  Since $L_{\sigma_1}$ is surjective, choose $a_{1}\in V_-$ such that
\(
a_{1}\sigma_1a_{1}^{-1}\sigma_1^{-1}=v_{1}
\). 
Similarly, we have $a_{2}\in V_{+}$ such that $\sigma_2^{-1}a_{2}\sigma_2a_{2}^{-1}=v_{2}$.
Therefore, \[y=v_{1}\,\sigma_1\sigma_2\,v_{2}=(a_{1}\sigma_1a_{1}^{-1}\sigma_1^{-1})\sigma_1\sigma_2(\sigma_2^{-1}a_{2}\sigma_2a_{2}^{-1})=(a_{1}\sigma_1a_{1}^{-1})(a_{2}\sigma_2a_{2}^{-1}).\]
Both factors are conjugates of elements of $\mathcal C_1$. Since $\mathcal C_1$ consists of $HG_1$-classes of $f$ and $f^{-1}$, and $f$ is real in $\GG$ by Hypothesis~1, each factor lies in the full $\GG$-conjugacy class $C=f^\GG$. Therefore $y\in C^2$, and since $y$ is conjugate to $x$, we have $x\in C^2$ as well.
As $x\in \GG\setminus Z(\GG)$ was arbitrary, the inclusion
\[
C^2\supseteq \GG\setminus Z(\GG)
\]
follows.
\end{proof}

We also use the strengthened $h=1$ form: a torus component equal to one may
be prescribed when every Levi core that actually occurs, central cores
included, has already been proved to lie in $\mathcal C_1^2$.  This is the
only point at which the usual condition $H_1\ne Z(G_1)$ is bypassed.

\subsection{Cyclic matrix products}

For a monic polynomial $p\in K[X]$ with $p(0)\ne0$, set
\begin{equation}
 p^\vee(X)=p(0)^{-1}X^{\deg p}p(X^{-1}),\qquad
 p^*(X)=\overline{p(0)}^{-1}X^{\deg p}\overline{p(X^{-1})}.
 \label{eq:reciprocal}
\end{equation}
The reciprocal of a polynomial  $p(X)$ is defined by $X^{\deg p}p(X^{-1})$. A polynomial $p(X)$ is called reciprocal  if $p(X)=X^{\deg p}p(X^{-1})$. 

Recall that a matrix \(A \in M_n(K)\) is cyclic if there exists a vector \(v \in K^n\) whose orbit under the powers of \(A\) spans \(K^n\).
In particular, the companion matrix \(C(p)\) of every monic polynomial \(p(X)\) is always cyclic. 
We require the following exact matrix-product results.

Nielsen's linear theorem~\cite[Theorem 1.2, Theorem 1.4]{NielsenGL}: if two cyclic
$\operatorname{GL}_n(K)$-classes are prescribed and one is $(1,n-1)$-cyclic for $n\geq 2$, then their
product contains every nonscalar matrix with the product of the source
determinants.  An $(m,n-m)$-cyclic source has the same conclusion for
$n\ge4$, with $n\ge5$ when $K=\mathbb F_3$.  The relevant centralizer norm
maps correct every conjugator to the required special linear or unitary
determinant.

Lev's theorem ~\cite{Lev94,Lev99}: every nonscalar matrix $M \in\operatorname{GL}_n(F)$ can be written as a product of two cyclic matrices provided that at least one of them is triangularizable for $n \ge 3$, subject only to the obvious determinant condition. The same conclusion holds for \(n = 2\) if either one is diaglizable or two are triangularizable, also for simple projective special linear groups.

Nielsen's symplectic theorem~\cite[Theorem 1.1]{NielsenSp}: for $2n\ge4$, the
product of two cyclic strictly hyperbolic $\mathrm{Sp}_{2n}(K)$-classes contains
every nonscalar element if $K=\mathbb F_3$, or $2n\ge6$, or both classes
are triangularizable.  In dimension four it contains every $M$ with
$M^2\ne I$.

The unitary Schur-complement lift of the linear theorem yields a
cyclic strictly hyperbolic real $\mathrm{SU}_{2n}(q)$-class whose square contains
every nonscalar element.  Over $\mathbb F_4$ and $\mathbb F_9$ the uniform
construction begins at $2n=10$; all lower dimensions used below are
covered by the explicit $(1,n-1)$ or $(2,2)$ splittings in
Section~\ref{sec:unitary}. 
The details are provided in Section~\ref{sec:symplectic} and in
\ref{candidates/route_unitary_even_hyperbolic.md}, and
\ref{candidates/route_unitary_odd_character.md}.

\section{Complete coverage matrix}
\label{sec:coverage}

\subsection{Exact Ellers--Gordeev thresholds}

Ellers--Gordeev's conclusion in ~\cite{EG} is inclusive at $q\ge d$.  Their
complete untwisted and twisted threshold tables are

\begin{equation}
\begin{array}{c|rrrrrrrrrrr}
X&A_l&B_2&B_l\ (l>2)&C_l&D_{2l}&D_{2l+1}&E_6&E_7&E_8&F_4&G_2\\ \hline
d&2&4&7&4&5&4&7&5&7&8&7
\end{array}
\label{tab:EG-untwisted}
\end{equation}
and

\begin{equation}
\begin{array}{c|rrrrrrrr}
X&{}^2A_{2l-1}&{}^2A_{2l}&{}^2D_{l+1}&{}^2E_6&{}^3D_4&{}^2B_2&{}^2G_2&{}^2F_4\\ \hline
d&8&4&7&8&7&3&4&9.
\end{array}
\label{tab:EG-twisted}
\end{equation}

The table records exact one-class results, but its entries do not all arise
from one uniform construction. For Suzuki and Ree groups we use their admissible-parameter convention. For the families treated directly by the
relative-Gauss argument, Ellers--Gordeev use the displayed real element
\(f=hu\), with \(u\) the explicit regular subsystem element (regular unipotent
in cases  \cite[Lemma 5.3]{EG} and semisimple-times-unipotent in 
cases \cite[Lemma 5.4]{EG}) and \(h\) an explicit torus-root product; \cite[Proposition 5.1 and
Lemmas 5.2--5.6]{EG} prove
\((f^G)^2\supseteq G\setminus Z(G)\). The \(A_l\) entry incorporates
Brenner's earlier exact class, the \(G_2\) entry uses a regular semisimple
element together with their earlier theorem, and the low-rank twisted entries
invoke the cited finite-twisted Gauss results. In each case the cited result
is a same-class square theorem, so
the thresholds may legitimately be used to remove parameters from the
present residue.

Restricting to \(q\in\{2,3,4,5,7,8\}\), the genuinely remaining unbounded
classical families before the new matrix-product inputs are therefore:

\begin{itemize}
\tightlist
\item
  \(B_l(q)\), \(l>2\), only \(q=2,3,4,5\); the even fields are isomorphic to
  the corresponding \(C_l(q)\), leaving the odd fields \(q=3,5\);
\item
  \(C_l(q)\), only \(q=2,3\);
\item
  \(D_{2l}(q)\), only \(q=2,3,4\), and \(D_{2l+1}(q)\), only \(q=2,3\);
\item
  \({}^2D_{l+1}(q)\), only \(q=2,3,4,5\);
\item
  even-dimensional unitary \({}^2A_{2l-1}(q^2)\), only
 \(q=2,3,4,5,7\);
\item
  odd-dimensional unitary \({}^2A_{2l}(q^2)\), only \(q=2,3\).
\end{itemize}

The genuinely remaining exceptional list is
\begin{equation}
\label{tb:3}
\begin{array}{c|l}
E_6&q=2,3,4,5\\
{}^2E_6&q=2,3,4,5,7\\
E_7&q=2,3,4\\
E_8&q=2,3,4,5\\
F_4&q=2,3,4,5,7.
\end{array}                                                   
\end{equation}
The relative-rank-at-most-two families \(G_2,{}^3D_4,{}^{2}B_2,{}^{2}G_2,
{}^2F_4\) have the separate exact one-class theorem of Guralnick--Malle \cite[Theorem 7.1, Theorem 7.3 and Remark 7.4]{GM},
including \({}^2F_4(8)\), so they do not remain in \eqref{tb:3}.

\subsection{Ellers--Gordeev representatives}
In this section, we list the representatives used in the rows taken from Ellers--Gordeev, pp.~3665--3670 in \cite{EG}.
We use their standard coordinate-root notation and write \(x_\alpha(a)\) and \(h_\alpha(a)\) for the Chevalley root
and torus elements. 
For an ordered root list
$\Delta=(\alpha_1,\ldots,\alpha_r)$, put\[
                 u_\Delta=x_{\alpha_1}(1)\cdots x_{\alpha_r}(1).
\]
This fixes a particular standard regular unipotent element in the subsystem
generated by \(\Delta\). Empty root lists are allowed and their products
are one. All classes below mean
the conjugacy class of the image of the displayed element in the finite
simple quotient.

Choose a generator \(s\) of \(\mathbb F_q^\times\). The inherited
classical representatives are as follows.

\begin{enumerate}
\tightlist
\item
  For \(B_l(q)\), \(l>2\), \(q\geq 7\), let
\begin{equation}
  \Delta_2=(\epsilon_3-\epsilon_4,\ldots,
             \epsilon_{l-1}-\epsilon_l),\quad
  u=h_{\epsilon_1-\epsilon_2}(s)u_{\Delta_2},\quad
  h=\prod_{i=1}^l h_{\epsilon_i}(s),\quad f=hu.
  \label{eq:Bn}
  \end{equation}
\item
  For \(D_n(q)\) with \(n\) even,  \(n\geq 4\), \(q\geq 5\), let
\begin{equation}
  \Delta_2=(\epsilon_3-\epsilon_4,\ldots,
             \epsilon_{n-1}-\epsilon_n),\quad
  u=h_{\epsilon_1-\epsilon_2}(s)u_{\Delta_2},\quad
  h=\prod_{j=1}^{n/2}h_{\epsilon_{2j-1}+\epsilon_{2j}}(s),\quad f=hu.
  \end{equation}
\item
  For \(D_n(q)\) with \(n\) odd, \(n\geq 5\), \(q\geq 4\), let
\begin{equation}
  \Delta_1=(\epsilon_2-\epsilon_3,\ldots,
             \epsilon_{n-1}-\epsilon_n),\quad
  u=u_{\Delta_1},\quad
  h=\prod_{j=1}^{(n-1)/2}h_{\epsilon_{2j}+\epsilon_{2j+1}}(s),\quad f=hu.
  \end{equation}
\item
  For \({}^2D_{l+1}(q)\), \(q\geq 7\), in the standard relative \(B_l\)-root system,
  put
 \begin{equation}
  \Delta_1=(e_1-e_2,\ldots,e_{l-1}-e_l),\quad
  u=u_{\Delta_1},\quad h=\prod_{i=1}^l h_{e_i}(s),\quad f=hu.
  \end{equation}
\item
  For \({}^2A_{2l-1}(q)\), \(l\ge2\), \(q\geq 8\), let
 \begin{equation}
  \Delta_2=(e_3-e_4,\ldots,e_{l-1}-e_l),\quad
  u=h_{e_1-e_2}(s)u_{\Delta_2},\quad
  h=\prod_{i=1}^l h_{2e_i}(s^2),\quad f=hu.
  \end{equation}
\item
  For \({}^2A_{2l}(q)\), \(q\geq 4\),  if \(l=1\), take
  \(f=\operatorname{diag}(s,1,s^{-1})\). If \(l>1\), let
 \begin{equation}
  \Delta_1=(e_1-e_2,\ldots,e_{l-1}-e_l),\quad
  u=u_{\Delta_1},\quad h=\prod_{i=1}^l h_{e_i}(s),\quad f=hu.
  \end{equation}
\end{enumerate}

For the inherited exceptional rows, use the following fixed representatives.
The coordinate realization is precisely that of Ellers--Gordeev.

\begin{enumerate}[resume]
\tightlist
\item
  In \(E_6(q)\), \(q\geq 7\), set
\begin{equation}
  \begin{split}
  \Delta_1&=(\epsilon_3-\epsilon_2,\epsilon_4-\epsilon_3,
              \epsilon_5-\epsilon_4),\\
  \beta&=\tfrac12(\epsilon_8-\epsilon_7-\epsilon_6+
             \epsilon_1+\epsilon_2+\epsilon_3+\epsilon_4+\epsilon_5),\\
  \gamma&=\tfrac12(\epsilon_8-\epsilon_7-\epsilon_6+
             \epsilon_1-\epsilon_2-\epsilon_3-\epsilon_4-\epsilon_5),\\
  f&=h_\beta(s^2)h_\gamma(s)u_{\Delta_1}.
  \end{split}                                                  
  \end{equation}
  where \(s^2\ne s^{\pm1}\) and \(s^4\ne1\), exactly the
  source restrictions in the paper.
\item
  In \(E_7(q)\), \(q\geq 5\), put
 \begin{equation}
  \Delta_1=(\epsilon_2-\epsilon_1,\ldots,
             \epsilon_6-\epsilon_5),\quad
  f=h_{\epsilon_8-\epsilon_7}(s)h_{\epsilon_1+\epsilon_2}(s)
     h_{\epsilon_3+\epsilon_4}(s)h_{\epsilon_5+\epsilon_6}(s)u_{\Delta_1}.
  \end{equation}
\item
  In \(E_8(q)\), \(q\geq 7\),  put
\begin{equation}
  \begin{split}
  \Delta_1&=(\epsilon_2-\epsilon_1,\ldots,
              \epsilon_7-\epsilon_6),\\
  \alpha_0&=\tfrac12(\epsilon_1+\cdots+\epsilon_8),\\
  f&=h_{\epsilon_8-\epsilon_7}(s^2)h_{\epsilon_8+\epsilon_7}(s^2)
       h_{\alpha_0}(s)u_{\Delta_1}.
  \end{split}                                                  
 \end{equation}
\item
  In \({}^2E_6(q)\), \(q\geq 8\), use its relative \(F_4\)-root system, set
  \(\beta_0=\tfrac12(\epsilon_1+\epsilon_2+\epsilon_3+\epsilon_4)\),
  \(\Delta_1=(\epsilon_2-\epsilon_3,\epsilon_3-\epsilon_4)\), and take
  \begin{equation}
  \label{eq:2E6}
  f=h_{\beta_0}(s)h_{\epsilon_2}(s)h_{\epsilon_3}(s)
       h_{\epsilon_4}(s)u_{\Delta_1}.                        
  \end{equation}
\end{enumerate}

Ellers--Gordeev \cite[Section 5, pp.~3665--3670]{EG} verify for each applicable threshold that the displayed
\(f\) is real and that \((f^G)^2\) contains every noncentral element of the
universal group. Reality includes the identity, and central projection
therefore gives \((\bar f^{G/Z(G)})^2=G/Z(G)\). Equations \eqref{eq:Bn}--\eqref{eq:2E6}
are thus exact representatives of conjugacy classes.

\subsection{\texorpdfstring{Classes for \(\mathrm{PSL}_3\) and \(\mathrm{PSU}_3\)}{rank2_exact_class_manifest.md}}
\label{candidates/rank2_exact_class_manifest.md}

This section fixes concrete representatives. Here \(C_A\) denotes the
conjugacy class of the projective image of a matrix \(A\). All parameters are
the standard ones: \(\mathrm{PSU}_3(q)={}^2A_2(q^2)\) in the notation of
Ellers--Gordeev \cite{EG}.

\subsubsection{The linear groups}\label{the-linear-groups}

For \(q\geq4\), choose \(\delta\in\mathbb F_q^\times\) with
\(\delta\ne\pm1\), and in the manifest below take \(\delta\) primitive. Put
\begin{equation}
\label{eq:PSLmatrix}
 A_q=\operatorname{diag}(1,\delta,\delta^{-1})\in\operatorname{SL}_3(q).
\end{equation}
For the two smaller fields let \(A_q\) be a companion matrix with the displayed
characteristic/minimal polynomial:
\begin{equation}
\label{eq:PSLpoly}
 \begin{array}{c|c}
 q&p_q(X)\\ \hline
 2&(X+1)(X^2+X+1),\\
 3&(X-1)(X^2+1).
 \end{array}                                          
\end{equation}
Equivalently, because the factors in \eqref{eq:PSLpoly} are coprime, one may take
\(A_2=1\oplus\operatorname{Comp}(X^2+X+1)\) and
\(A_3=1\oplus\operatorname{Comp}(X^2+1)\). Every \(A_q\) has determinant one.
Its invariant polynomial is reciprocal, so \(A_q\) is similar to \(A_q^{-1}\).
The relevant \( GL _3(q)\)-class is a single
\(\mathrm{SL}_3(q)\)-class: for \eqref{eq:PSLmatrix} the diagonal centralizer has
determinant surjective image, while for \eqref{eq:PSLpoly} the centralizer is
\(\mathbb F_q^\times\times\mathbb F_{q^2}^\times\), whose determinant map is
surjective. Thus \(C_{A_q}\) is real in \(\mathrm{PSL}_3(q)\).

Recall that a matrix $A$ is called cyclic if its minimal polynomial equals its characteristic polynomial.
The published product theorem is Lev's cyclic-class theorem \cite[Theorem 3]{Lev99}. For \(q\geq4\),
Lev's theorem quoted verbatim in Ellers--Gordeev \cite[Section~5, Theorem]{EG}, says that the product of the two prescribed cyclic similarity classes contains every
nonscalar matrix having the required determinant. Taking both classes to be
the class of \eqref{eq:PSLmatrix}, and using the determinant-surjective centralizer just noted,
gives every noncentral element of \(\mathrm{SL}_3(q)\) as a product of two
members of the same \(\mathrm{SL}_3(q)\)-class. Theorem 3 of Lev  in \cite[Theorem 3]{Lev99} handles all fields; 
its exact self-reciprocal polynomials are restated in Nielsen \cite[Corollary 1.5]{NielsenGL}. 
At \(n=3\), the \(\mathbb F_2\) and
\(\mathbb F_3\) formulas specialize exactly to \eqref{eq:PSLpoly}. Reality supplies the
identity. Consequently
\[
 C_{A_q}^{2}=\mathrm{PSL}_3(q)
 \qquad(q=2,3,4,5,7,8).                                   
\]

For identification only, not as the proof, installed CTblLib gives:
\begin{equation}
\label{eq:PSL3}
\begin{array}{c|c|c|c}
q&\text{class label}&|A_q|&|C_G(A_q)|\\ \hline
2&3\mathrm A&3&3\\
3&4\mathrm A&4&8\\
4&3\mathrm A&3&9\\
5&4\mathrm C&4&16\\
7&6\mathrm A&6&12\\
8&7\mathrm I,7\mathrm J,\text{ or }7\mathrm K&7&49
\end{array}                                                 
\end{equation}
For \(q=8\), the six choices of primitive \(\delta\) give the three real
classes in \eqref{eq:PSL3}, paired by \(\delta\leftrightarrow\delta^{-1}\); any one is the
required class. Exact Frobenius coefficients in every row of \eqref{eq:PSL3} are positive,
which is only a supplementary check.

\subsubsection{The three-dimensional unitary groups}\label{the-three-dimensional-unitary-groups}

Let \(K=\mathbb F_{q^2}\), let \(\bar a=a^q\), and realize the hermitian form by
the antidiagonal Gram matrix \(J\). Choose a generator
\(t\in\mathbb F_q^\times\) and put
\[
 H_q=\operatorname{diag}(t,1,t^{-1})\in SU _3(q).
\]
Then \(\bar H_q^{T}JH_q=J\), \(\det H_q=1\), and, for \(q\geq4\), its three
eigenvalues are distinct. Hence it is regular semisimple. It is explicitly
real: if \(w\) swaps the first and third basis vectors, then
\(wH_qw^{-1}=H_q^{-1}\); in even characteristic \(w\in SU _3(q)\), and
in odd characteristic \(\operatorname{diag}(1,-1,1)w\in SU _3(q)\)
has the same conjugating action.

This is exactly the representative in Ellers--Gordeev \cite[Section 5]{EG}, case
\({}^2A_{2l}(q^2)\) with \(l=1\), and apply 
\cite[Theorem 1]{EG}. Taking the same class for its two regular semisimple gives all noncentral elements upstairs; passage to the simple quotient and reality give
\[
 C_{H_q}^{2}=\mathrm{PSU}_3(q)
 \qquad(q=4,5,7,8).                                        
\]

Again only for identification, CTblLib records:
\begin{equation}
\label{eq:PSU3}
\begin{array}{c|c|c|c}
q&\text{class label}&|H_q|&|C_G(H_q)|\\ \hline
4&3\mathrm A&3&15\\
5&4\mathrm A&4&8\\
7&6\mathrm A&6&48\\
8&7\mathrm A,7\mathrm B,\text{ or }7\mathrm C&7&21
\end{array}                                                  
\end{equation}
For \(q=8\), generator choices give the three labels in \eqref{eq:PSU3}; each is real and
each has positive square coefficient on every target class.

The linear representatives are
specified by invariant polynomials and are covered by Lev's published uniform
cyclic-class theorem \cite[Theorem 3]{Lev99}; the unitary representatives are the literal matrices used
in the Ellers--Gordeev proof \cite[Section 5]{EG}. CTblLib labels merely identify and independently
check these theorem-defined classes.

\subsection{Complete coverage matrix}
In the following tables, EG means tables~\eqref{tab:EG-untwisted}--\eqref{tab:EG-twisted},
RG means Lemma~\ref{lem:relative-gauss} with all hypotheses checked, N means
the cyclic matrix-product inputs, and F means the Frobenius argument
\eqref{eq:Fx}--\eqref{eq:column}.

\begin{table}[p]
\caption{Classical groups.}\label{tab:classical}
\begin{tabularx}{\textwidth}{@{}lp{1.5cm}YYl}
\toprule
Type/simple group&Rank&EG fields&Residual fields&Source and proof\\ \midrule
$A_l=\mathrm{PSL}_{l+1}$&$l\ge1$&$2,3,4,5,7,8$&none&Section~\ref{sec:linear} (Lev--Nielsen)\\
${}^2A_{2m}=\operatorname{PSU}_{2m+1}$&$m\ge1$&$4,5,7,8$&$2,3$&Section~\ref{sec:unitary}, RG+N\\
${}^2A_{2m-1}=\operatorname{PSU}_{2m}$&$m\ge2$&$8$&$2,3,4,5,7$&Section~\ref{sec:unitary}, N/RG\\
$B_l=\operatorname{P\Omega}_{2l+1}$&$l>2$&$7,8$&$3,5$&Section~\ref{sec:orthogonal}, RG+N\\
$C_l=\operatorname{PSp}_{2l}$&$l\ge2$&$4,5,7,8$&$2,3$&Section~\ref{sec:symplectic}, N\\
$D_{2r}=\mathrm{P\Omega}^+_{4r}$&$r\ge2$&$5,7,8$&$2,3,4$&Section~\ref{sec:orthogonal}, RG+N\\
$D_{2r+1}=\mathrm{P\Omega}^+_{4r+2}$&$r\ge2$&$4,5,7,8$&$2,3$&Section~\ref{sec:orthogonal}, RG+N\\
${}^2D_n=\mathrm{P\Omega}^-_{2n}$&$n\ge4$&$7,8$&$2,3,4,5$&Section~\ref{sec:minus}, RG+N\\
\bottomrule
\end{tabularx}
\end{table}

In even characteristic $B_l(q)\cong C_l(q)$; hence $q=2$ is covered by
the symplectic source and $q=4,8$ by the $C_l$ threshold.  For every
$\mathrm{PSL}_3(q)$ use the real Lev class: for $q\ge4$ it is represented by
$\operatorname{diag}(1,\delta,\delta^{-1})$ with $\delta$ primitive, and
for $q=2,3$ by $(X+1)(X^2+X+1)$ and $(X-1)(X^2+1)$.  For $\operatorname{PSU}_3(q)$,
$q\ge4$, use the image of $\operatorname{diag}(t,1,t^{-1})$ with $t$ a
generator of $\mathbb F_q^\times$; for $\operatorname{PSU}_3(3)$ use the class
\texttt{U3(3):3b}.  Detailed treatments are given in Section \ref{candidates/rank2_exact_class_manifest.md}.

\begin{table}[p]
\caption{Exceptional and very twisted groups.}\label{tab:exceptional}
\begin{tabularx}{\textwidth}{@{}lp{2.5cm}Yl@{}}
\toprule
Family&Fields in $\cQ$&Chosen class&Proof\\ \midrule
$G_2$&$3,4,5,7,8$&generator of torus $(q^2+q+1)/(3,q-1)$&GM\\
${}^3D_4$&$2,3,4,5,7,8$&generator of torus $q^4-q^2+1$&GM\\
${}^2B_2$&$8$&\texttt{Sz(8):5a}&GM/table\\
${}^2G_2$&$3'$&\texttt{L2(8):7a}&$A_1$/table\\
${}^2F_4$&$2',8$&\texttt{13a}; order-$109$ torus&GM/table\\ \midrule
$F_4$&$2$&\texttt{F4(2):13a}&table\\
$F_4$&$3,4,5,7,8$&$\Phi_{12}(q)$ Coxeter generator&F\\
$E_6$&$2$&\texttt{E6(2):13a}&table\\
$E_6$&$3,5$&orders $73,601$ in Coxeter torus&F\\
$E_6$&$4$&$\bar f_s$ of order $255$&RG+N\\
$E_6$&$7,8$&Ellers--Gordeev $f_{E_6}(q)$&EG\\
${}^2E_6$&$2$&\texttt{2E6(2):13a}&table\\
${}^2E_6$&$3,7$&orders $73,2353$ in twisted Coxeter torus&F\\
${}^2E_6$&$4$&$f_t$ of order $60$&RG+N\\
${}^2E_6$&$5$&order-$601$ image in simple quotient&F\\
${}^2E_6$&$8$&Ellers--Gordeev $f_{{}^2E_6}(8)$&EG\\
$E_7$&$2,3,4$&the $A_6$-subsystem matrices below&RG+N\\
$E_7$&$5,7,8$&Ellers--Gordeev $f_{E_7}(q)$&EG\\
$E_8$&$2$&order-$331$ $\Phi_{30}$ Coxeter generator&F\\
$E_8$&$3,4,5$&the $A_7$-subsystem matrices below&RG+N\\
$E_8$&$7$&order-$6568801$ $\Phi_{30}$ Coxeter generator&F\\
$E_8$&$8$&Ellers--Gordeev $f_{E_8}(8)$&EG\\
\bottomrule
\end{tabularx}
\end{table}

Here GM denotes Guralnick--Malle, Theorems~7.1 and 7.3 and
Remark~7.4 in ~\cite{GM}, only in the exact real-class cases stated there.  We never
promote a non-real class whose square contains only $G\setminus\{1\}$ to a
Thompson class.  For ${}^2F_4(8)$ the torus order is $109$.  In $Sz(8)$ we
choose class \texttt{5a}, of order and centralizer $5$; its inverse map and
full square are certified exactly.

\FloatBarrier

The following identifications ensure that Tables~\ref{tab:classical} and
\ref{tab:exceptional} neither omit nor double-count a simple group.

$A_1=B_1=C_1$, $\mathrm{Sp}_2(q)=\operatorname{SL}_2(q)$, and
$\operatorname{PSU}_2(q)\cong\operatorname{PSL}_2(q)$.  The latter is simple except for $q=2,3$.
Also $\mathrm{PSL}_2(4)\cong\operatorname{PSL}_2(5)\cong A_5$,
$\mathrm{PSL}_3(2)\cong\operatorname{PSL}_2(7)$, and $\mathrm{PSL}_4(2)\cong A_8$.

$B_2=C_2$ and $\operatorname{P\Omega}_5(q)\cong \operatorname{PSp}_4(q)$.  The nonsimple
$\operatorname{PSp}_4(2)\cong S_6$ has derived group $A_6$, for which class
\texttt{A6:5a} works.  Moreover $\operatorname{PSp}_4(3)\cong \operatorname{PSU}_4(2)$.

$D_2=A_1A_1$ and gives no new simple group; $D_3=A_3$, with
$\operatorname{P\Omega}_6^+(q)\cong\operatorname{PSL}_4(q)$.  Similarly
$\operatorname{P\Omega}_4^-(q)\cong\operatorname{PSL}_2(q^2)$ and
$\operatorname{P\Omega}_6^-(q)\cong \operatorname{PSU}_4(q)$.

$\operatorname{PSU}_n(q)$ is simple for $n\ge3$ except $\operatorname{PSU}_3(2)$; the simple
$\operatorname{PSU}_4(2)$ is $\operatorname{PSp}_4(3)$.

$G_2(2)'\cong \operatorname{PSU}_3(3)$ uses \texttt{U3(3):3b};
${}^2G_2(3)'\cong\operatorname{PSL}_2(8)$ uses \texttt{L2(8):7a}; and the Tits group
${}^2F_4(2)'$ uses \texttt{2F4(2)':13a}.  The group ${}^2B_2(2)$ is
solvable.

\section{Linear and symplectic groups}
\label{sec:classical-sources}

\subsection{Linear groups}
\label{sec:linear}

We begin with the case of type \(A_{l}\), for which the required
covering classes are known by a theorem of Lev. We recall the result here. Moreover, the proof in \cite{Lev99} provides
explicit representatives, which can be described using companion matrices.

\begin{theorem}[{\cite[Theorem~3]{Lev99}}]\label{typeA}
Let \(G=\mathrm{PSL}_N(q)\), where \((N,q)\neq(2,2),(2,3)\).
Then \(G\) contains a conjugacy class \(C\) such that
\[
C^2=G.
\]
\end{theorem}

\begin{proof}
Let $N=l+1$, and  let \(C(p)\) denote the companion matrix of a monic polynomial \(p\) of degree \(N\).  For each case below,
the chosen polynomial gives an element of \(\operatorname{SL}_N(q)\), and we let \(\Omega_N(q)\) is the conjugacy class in \(\mathrm{PSL}_N(q)\) of the image of \(C(p)\). 
The polynomial \(p\), and hence the corresponding conjugacy class,
is completely explicit.

Choose \(\delta\in\mathbb F_q^\times\setminus\{\pm1\}\) when \(q\geq4\)
(a primitive element works for \(q=4,5,7,8\)). Write \(N=2t+1\) or
\(N=2t+6\) where appropriate. Take
\begin{equation}
\label{eq:PSLnpoly}
\begin{array}{c|c|c}
\text{parameter}&\text{dimension}&p_{N,q}(X)\\ \hline
q\geq4&N\geq2&(X-1)^{N-2}(X-\delta)(X-\delta^{-1})\\
q=3&N=2t+1\geq3&(X-1)(X^2+1)^t\\
q=3&N=2t+6\geq6&(X^2+1)^t(X^3-X+1)(X^3-X^2+1)\\
q=3&N=4&(X^2-1)(X^2+X-1)\\
q=2&N=2t+1\geq3&(X+1)(X^2+X+1)^t\\
q=2&N=2t+6\geq6&(X+1)^{2t}(X^3+X^2+1)(X^3+X+1)\\
q=2&N=4&(X+1)^2(X^2+X+1).
\end{array}                                                  
\end{equation}
The two excluded pairs \((N,q)=(2,2),(2,3)\) do not define non-abelian
simple groups.

Every polynomial in \eqref{eq:PSLnpoly} is monic and the companion matrix has determinant 
\((-1)^Np_{N,q}(0)=1\), so \(C(p_{N,q})\in \operatorname{SL}_N(q)\). Each companion
class is cyclic. The displayed coprime factors show that it is
\((1,N-1)\)-cyclic in the first, second, fourth, and fifth rows,
\((3,N-3)\)-cyclic in the third and sixth rows, and \((2,2)\)-cyclic in
the last row. Nielsen's Theorem \cite[Theorem 1.2]{NielsenGL} applies to the \((1,N-1)\) rows;
\cite[Theorem 1.4]{NielsenGL} applies to the remaining rows. (Its extra condition over
\(\mathbb F_3\) is satisfied because the third row has \(N\geq6\).)
Thus the square of the corresponding cyclic \(\operatorname{GL}_N(q)\)-class contains
every nonscalar matrix of determinant one.

This is genuinely one \(\operatorname{SL}_N(q)\)-class. The determinant map on the
\(\operatorname{GL}_N(q)\)-centralizer of the companion matrix is surjective: use an
arbitrary scalar on a one-dimensional primary summand in the first, second,
fourth, and fifth rows; use \(-1\) on one three-dimensional primary summand
in the third row; and note that the determinant group is trivial when
\(q=2\). Hence every \(\operatorname{GL}_N(q)\)-conjugator can be corrected by a
centralizer element to have determinant one.

Except in the fourth row, \(p_{N,q}=p_{N,q}^{\vee}\), so the companion
matrix is \(\operatorname{SL}_N(q)\)-conjugate to its inverse by the same determinant
correction. In the fourth row\[
 p_{4,3}^{\vee}(X)=p_{4,3}(-X),
\]
so \(C(p_{4,3})\) is \(\operatorname{SL}_4(3)\)-conjugate to
\(-C(p_{4,3})^{-1}\). Since \(-I\) is central, the projective class is
real in this case as well. Consequently the identity of \(\mathrm{PSL}_N(q)\) lies
in \(\Omega_N(q)^2\), while every nonidentity projective element has a
nonscalar determinant-one lift and is covered by the preceding product
theorem. Therefore
\[
                        \Omega_N(q)^2=\mathrm{PSL}_N(q).      
\]
\end{proof}

Four uniform polynomial families in \eqref{eq:PSLnpoly} are Lev's Theorem \cite[Theorem 3]{Lev99} as restated
by Nielsen in \cite[Corollary 1.5]{NielsenGL}. In the even \(q=3\) row we instead use the displayed reciprocal pair of irreducible cubics; this gives the same direct
application of Nielsen's Theorem \cite[Theorem 1.4]{NielsenGL} and makes the centralizer-determinant
argument transparent. The two omitted dimension-four boundaries follow
directly from Nielsen's Theorems \cite[Theorem 1.2 and Theorem 1.4]{NielsenGL} as above.
The exact character-table checks \texttt{L4(2):6B} and \texttt{L4(3):8A} independently
certify these two boundary classes.

\subsection{Symplectic groups}
\label{sec:symplectic}
We next recall a product theorem for symplectic groups due to Nielsen. 
This result describes when products of suitable semisimple
conjugacy classes cover all nonscalar elements, and it will be used to
construct covering classes in the relevant finite classical groups.

\begin{theorem}[{\cite[Theorem 1.1]{NielsenSp}}]\label{Symp}
Let \(2n\ge4\), and let \(\Omega_1,\Omega_2\) be cyclic strictly
hyperbolic classes of \(\mathrm{Sp}_{2n}(K)\). Their product contains every
nonscalar element if \(K=\mathbb F_3\), either \(2n\ge6\), or both classes
are triangularizable. If \(2n=4\), it always contains every \(M\) with
\(M^2\ne I\). Consequently, \(\operatorname{PSp}_{2n}(K)\) contains a conjugacy class \(C\) such that \(C^2 = \operatorname{PSp}_{2n}(K)\).
\end{theorem}

Let \(K\) be a field and \(n\ge2\). For a monic polynomial \(f\in K[X]\)
with \(f(0)\ne0\), write\[
 f^\vee(X)=f(0)^{-1}X^{\deg f}f(X^{-1})                    
\]
for its monic reciprocal. Multiplication by a nonzero scalar does not
change coprimality, so this is the normalized version of the reciprocal
\(X^{\deg f}f(X^{-1})\) used in the introductory statement of Nielsen \cite{NielsenSp}.

The proof of the theorem proceeds by prescribing an \(n\times n\) principal
corner, factoring this corner into two cyclic factors, and then lifting the
factorization through the exact symplectic Schur-complement identity. In our
applications, we use only the first part of the theorem with the two classes
chosen to be equal. The weaker assertion in dimension \(4\) is not needed for
any of the simple groups considered below.

\begin{proof}
Let $K=\mathbb F_q$.
For \(q\in\{2,3,4,5,7,8\}\), choose the following monic degree \(n\)
polynomial \(f=f_{q,n}\):

\begin{equation}
\begin{array}{c|l}
q\ge4&(X-\lambda)^n,\quad \lambda\in\mathbb F_q^\times\setminus\{1,-1\},\\
q=3,\ n=2m&(X^2-X-1)^m,\\
q=3,\ n=2m+3&(X^2-X-1)^m(X^3-X-1),\\
q=2,\ n\ge3&X^n+X^{\lfloor n/2\rfloor+1}+1.
\end{array}
\label{eq:sp-poly}
\end{equation}
Let \(A=C(f)\) be the companion matrix  and
\begin{equation}
              s_{q,n}=\operatorname{diag}(A,A^{-T})\in\operatorname{Sp}_{2n}(q).
 \label{eq:sp-source}
\end{equation}
using the standard alternating
form
\[
 J=\begin{pmatrix}0&I_n\\-I_n&0\end{pmatrix}.
\]
These are precisely the choices in \cite[Corollary 1.3]{NielsenSp}, whose proof
checks
\begin{equation}
\label{eq:Sympcoprime}
                         \gcd(f,f^\vee)=1.                 
\end{equation}

Define
\begin{equation}
\label{eq:SpRep}
\Omega_{q,n}=s_{q,n}^{\mathrm{Sp}_{2n}(q)}.                 
\end{equation}
This gives an exact representative, so no uniqueness assertion about
intersections of \(GL\)-classes with \(Sp\) is required. Its characteristic
polynomial is the monic polynomial \(ff^\vee\). By \eqref{eq:Sympcoprime}, the two primary
summands are coprime; hence its minimal polynomial is also \(ff^\vee\), of
degree \(2n\). Thus \(s_{q,n}\) is cyclic and strictly hyperbolic.

The class is real inside the symplectic group. Indeed every square matrix
is similar to its transpose, so choose \(R\in \operatorname{GL}_n(q)\) with
\(RAR^{-1}=A^T\). Both
\[
 D=\operatorname{diag}(R,R^{-T})\quad\text{and}\quad J
\]
belong to \(\mathrm{Sp}_{2n}(q)\), and direct block multiplication gives
\begin{equation}
\label{eq:Spreal}
 JD\,s_{q,n}\,D^{-1}J^{-1}
   =\operatorname{diag}(A^{-1},A^T)=s_{q,n}^{-1}.          
\end{equation}
Consequently \(\Omega_{q,n}=\Omega_{q,n}^{-1}\) and
\(I_{2n}\in\Omega_{q,n}^2\).

For \(q\ge4\) the class is triangularizable, so \cite[Theorem 1.1]{NielsenSp} applies even
when \(2n=4\). For \(q=3\), the explicit \(K=\mathbb F_3\) alternative in
\cite[Theorem 1.1]{NielsenSp} applies in every dimension, including \(2n=4\). For \(q=2\)
the construction is used only for \(2n\ge6\). Therefore
\begin{equation}
\label{eq:Spcover}
\operatorname{Sp}_{2n}(q)\setminus Z(\mathrm{Sp}_{2n}(q))\subseteq\Omega_{q,n}^2.  
\end{equation}
Combining \eqref{eq:Spreal}--\eqref{eq:Spcover}, and observing that the only possible remaining
scalar is \(-I\), gives the precise upstairs statement
\begin{equation}
\label{eq:Sppcover}
                 \operatorname{Sp}_{2n}(q)=\Omega_{q,n}^2\cup\{-I\}.      
\end{equation}

Let \(G=\operatorname{PSp}_{2n}(q)=\mathrm{Sp}_{2n}(q)/Z(\mathrm{Sp}_{2n}(q))\) be non-abelian simple quotient, and
let \(C_{q,n}\) be the image of \(\Omega_{q,n}\). A central quotient sends
one upstairs conjugacy class to one conjugacy class. Every nonidentity
projective element has a nonscalar lift, while all central scalars map to
the identity. Equation \eqref{eq:Sppcover} therefore yields
\[
                           C_{q,n}^2=G.            
\]
\end{proof}

This proves, with the exact representative \eqref{eq:SpRep}, for every simple
\(\operatorname{PSp}_{2n}(q)\) with \(q\le8\) and \(n\ge2\). The parameter
\(\operatorname{PSp}_4(2)\cong S_6\) is not simple; its derived group is \(A_6\), covered
by the alternating-group theorem. Rank one is
\(\operatorname{PSp}_2(q)=\mathrm{PSL}_2(q)\) and is covered by the separately verified linear and
rank-one classes.

\section{Unitary groups}
\label{sec:unitary}

\subsection{\texorpdfstring{The even unitary families}{route_unitary_even_hyperbolic.md}}
\label{candidates/route_unitary_even_hyperbolic.md}

This section records the precise part of the small field unitary problem supplied
by Nielsen's cyclic class product theorem. It is a complete theorem in the
ranges stated below; it does not claim the odd-dimensional or the explicitly
excluded low-dimensional cases.

\subsubsection{The matrix theorem}\label{the-matrix-theorem}

Let \(K=\mathbb F_{q^2}\), let \(a\mapsto\bar a=a^q\), and let \(V\) be a
hyperbolic Hermitian \(K\)-space of dimension \(2n\); thus $V$ carries a nondegenerate Hermitian form $h$ with a basis in which the Gram matrix is
\[
H=\begin{pmatrix}0&I_n\\ I_n&0\end{pmatrix}.
\]

For a polynomial \(p\)
with nonzero constant term, its conjugate reciprocal is defined by 
\[
 p^*(X)=\overline{p(0)}^{-1}X^{\deg p}\overline{p(X^{-1})}.
\]
A unitary transformation is called \emph{strictly hyperbolic} if its minimal polynomial has the form $pp^{*}$ with $\gcd(p,p^{*})=1$.  For such a transformation, the primary components $V_{p}$ and $V_{p^{*}}$ are totally isotropic and paired by the form $h$.  Consequently, a strictly hyperbolic element is unitarily diagonalisable: it is $U$-conjugate to a block diagonal matrix
\[
\begin{pmatrix}A&0\\[2pt] 0&\overline{A}^{\,-T}\end{pmatrix},
\]
where $A\in \operatorname{GL}_{n}(K)$ is cyclic with minimal polynomial $p$.  In particular, the $U$-conjugacy class of a strictly hyperbolic element is uniquely determined by the $\operatorname{GL}_{n}(K)$-conjugacy class of its upper-left block $A$.
We recall Theorems 1.2, 1.4, 1.6 and Lemmas 5.1--5.2 of Nielsen in \cite{NielsenGL}:

\begin{lemma}[{\cite[Lemma 5.2]{NielsenGL}}]\label{SU:Nielsen}
If \(2n\ge2\), and additionally \(2n\ge10\) when
\(K=\mathbb F_4\) or \(\mathbb F_9\), then there is a cyclic strictly
hyperbolic \(SU(V)\)-class \(\Omega=\Omega^{-1}\), with minimal polynomial
\(pp^*\) and \(\gcd(p,p^*)=1\), such that \(\Omega^2\) contains every
nonscalar element of \(SU(V)\).
\end{lemma}

The factorization is obtained by prescribing a nonscalar principal
\(n\times n\) corner of an arbitrary element, factoring this corner
using the appropriate cyclic product theorem,
and then lifting factorization through the unitary Schur complement. For \(n\ge4\), we apply  \cite[Theorem 1.4]{NielsenGL}. In the remaining cases \(n=2,3\) over the large fields considered below, we instead use \cite[Theorem 1.2]{NielsenGL}: the displayed polynomial has a simple
linear factor and hence its cyclic class is \((1,n-1)\)-cyclic.

\begin{theorem}[{\cite[Theorem 1.6]{NielsenGL}}]\label{evenu}
Let \(G=\operatorname{PSU}_{2n}(q)\), where $n\geq 1$ and $n\geq 5$ if $q=2,3$.
Then \(G\) contains a conjugacy class \(C\) such that
\[
C^2=G.
\]
\end{theorem}

\begin{proof}
Let \(\Omega\) be the strictly hyperbolic conjugacy class in
\(\mathrm{SU}_{2n}(q)\) provided by Lemma~\ref{SU:Nielsen}, with minimal polynomial
\(pp^{*}\). Denote by \(C\) the image of \(\Omega\) in the corresponding
simple quotient \(G\). Then
\(
C^{2}=G
\).
Hence the class is uniquely specified by the invariant factor \(pp^*\),
where \(\Omega=u_{q,n}^{\mathrm{SU}_{2n}(q)} \), \(u_{q,n}=\operatorname{diag}(A,\overline{A}^{\,-T})\in\operatorname{SU}_{2n}(q)\), and $A=C(p)$ the companion matrix of $p$,
where $p=p_1p_2$ is as follows:

Case 1: \(q\ge4\), with \(n\ge2\).
Choose \(
 \alpha\in\mathbb F_q\setminus\{0,1,-1\},
 \lambda\in\mathbb F_{q^2}\setminus\mathbb F_q,
\lambda\bar\lambda\ne1,
\)
and take
\begin{equation}
\label{eq:EvenU1}
 p_1=(X-\lambda)(X-\bar\lambda),\qquad
 p_2=(X-\alpha)^{n-2}.                                     
\end{equation}
This applies in every even dimension at least four. In the present range it
covers \(q=4,5,7,8\). Rank one is handled by the separate
\(\operatorname{PSU}_2(q)=\mathrm{PSL}_2(q)\) theorem.

 Case 2: \(q=3\), hence \(K=\mathbb F_9\), and \(n\ge5\).
For \(n=2t+3\ge5\), take
\begin{equation}
\label{eq:EvenU2}
 p_{1}(X)=X^3-X-1\quad p_{2}(X)=(X^2-X-1)^t.                            
\end{equation}
For \(n=2t+6\ge6\), take
\begin{equation}
\label{eq:EvenU3}
  p_{1}(X)=X^3-X-1, \quad p_{2}(X)=(X^2-X-1)^t.                           
\end{equation}
The two formulas cover every \(n\ge5\).

Case 3: \(q=2\), hence \(K=\mathbb F_4\), and \(n\ge5\).
Fix \(\lambda\in\mathbb F_4\setminus\{0,1\}\). According as \(n\) lies
in one of the following three exhaustive congruence forms, take
\begin{equation}
\label{eq:EvenU4}
\begin{array}{ll}
n=3t+4\ge7:&p_{1}(X)=X^4+X+1, \quad p_{2}(X)=(X^3+X+1)^t,\\[1mm]
n=3t+2\ge5:&p_{1}(X)=X^2+\lambda X+1, \quad p_{2}(X)=(X^3+X+1)^t,\\[1mm]
n=3t+6\ge6:&p_{1}(X)=X^2+\lambda X+1, \quad p_{2}(X)=(X^3+X+1)^t(X^4+X+1).
\end{array}                                                 
\end{equation}
These cover every \(n\ge5\).

Nielsen's Lemma \cite[Lemma 5.2]{NielsenGL} verifies for \eqref{eq:EvenU1}--\eqref{eq:EvenU4} both
\(\gcd(p,p^*)=1\) and inverse-stability. For convenience, we provide verification of the conditions here.
We now proceed to verify the following conditions for the monic polynomial \( p \in K[X] \) introduced above:
\begin{enumerate}[label=(\roman*)]
\item\label{item:coprime} $\gcd(p,p^{*})=1$;
\item\label{item:det1} the strictly hyperbolic class with minimal polynomial $pp^{*}$ is contained in $\operatorname{SU}_{2n}(K)$;
\item\label{item:inverse} the class is inverse-stable: $\Omega=\Omega^{-1}$;
\end{enumerate}
 
 Case 1: $q\ge 4$.
 
\ref{item:coprime}: Since $\lambda\notin\mathbb F_{q}$, we have $\lambda\neq\bar\lambda$, and $p_{1}$ has distinct roots $\lambda,\bar\lambda$.  Because $\alpha\in\mathbb F_{q}^{\times}$, the roots of $p_{2}$ are $\alpha$ with multiplicity $n-2$, while the roots of $p_{2}^{*}$ are $\alpha^{-1}$ with the same multiplicity.  The hypothesis $\alpha\neq\pm1$ guarantees $\alpha\neq\alpha^{-1}$.  Moreover, $\{\lambda,\bar\lambda\}\cap\mathbb F_{q}=\varnothing$, whereas $\alpha,\alpha^{-1}\in\mathbb F_{q}$; hence $p_{1}$ is coprime to both $p_{2}$ and $p_{2}^{*}$.  A direct computation shows $p_{1}^{*}=X^{2}-(\lambda+\bar\lambda)(\lambda\bar\lambda)^{-1}X+(\lambda\bar\lambda)^{-1}$, which is distinct from $p_{1}$ because $\lambda\bar\lambda\neq1$.  Thus $\gcd(p_{1},p_{1}^{*})=1$, and consequently $\gcd(p,p^{*})=1$.

\ref{item:det1}: The constant term is $p(0)=(\lambda\bar\lambda)\cdot(-\alpha)^{n-2}$.  Both factors lie in $\mathbb F_{q}^{\times}$; hence $p(0)=\overline{p(0)}$.  For the companion matrix $C(p)$ we have $\det C(p)=(-1)^{n}p(0)$, and therefore the block diagonal representative of $\Omega$ has determinant
\[
\det C(p)\cdot\det\overline{C(p)}^{\,-T}=(-1)^{n}p(0)\cdot\overline{(-1)^{n}p(0)}^{\,-1}=p(0)/\overline{p(0)}=1.
\]
Thus $\Omega\subset\operatorname{SU}_{2n}(K)$.

\ref{item:inverse}: Since $p\in\mathbb F_{q}[X]$ all coefficients are fixed by the involution, we have $\overline{p(X^{-1})}=p(X^{-1})$, whence $p^{*}=p^{\vee}$, the ordinary reciprocal polynomial.  Consequently the set of irreducible factors of $pp^{*}$ is closed under $f\mapsto f^{*}$, equivalently under $f\mapsto f^{\vee}$.  The strictly hyperbolic elements defined by $p$ and by $p^{\vee}$ therefore have the same minimal polynomial $pp^{*}$ and the same elementary divisors; they are $U$-conjugate.  To see that they are already $ SU$-conjugate, let $P=\operatorname{diag}(A,\overline{A}^{\,-T})$ be a representative of $\Omega$ and choose $g\in U_{2n}(K)$ with $gPg^{-1}=P^{-1}$.  If $\det g=\zeta$, replace $g$ by $g\cdot\operatorname{diag}(\mu I_{n},\bar\mu^{-1}I_{n})$ where $\mu\in K^{\times}$ satisfies $\mu/\bar\mu=\zeta^{-1}$ (possible by Hilbert~90).  The new element lies in $\operatorname{SU}_{2n}(K)$ and still conjugates $P$ to $P^{-1}$.  Hence $\Omega=\Omega^{-1}$.

Case 2: $q=3$, so $K=\mathbb F_{9}$, $n\ge5$.

\ref{item:coprime}: The factors $X^{3}-X-1$, $X^{3}-X+1$, and $X^{2}-X-1$ are pairwise coprime.  None of them is self-reciprocal over $\mathbb F_{3}$: for a monic irreducible polynomial $f$ of degree $d$, self-reciprocality would imply $f(0)=\pm1$ and $f(X)=\pm X^{d}f(X^{-1})$; a short inspection shows that this happens for none of the three factors.  Since all coefficients lie in $\mathbb F_{3}$, we have $f^{*}=f^{\vee}$; hence none of the factors is self-conjugate-reciprocal either, and they are pairwise coprime to their conjugate reciprocals.  Thus $\gcd(p,p^{*})=1$.

\ref{item:det1}: Each factor has constant term $\pm1\in\mathbb F_{3}^{\times}$; consequently $p(0)\in\mathbb F_{3}^{\times}$ and the same computation as in Case~1 shows that the block diagonal representative has determinant~$1$.

\ref{item:inverse}: Again $p\in\mathbb F_{3}[X]$, so $p^{*}=p^{\vee}$.  The set of irreducible factors of $pp^{*}$ is $\{X^{3}-X-1,\,X^{3}-X+1,\,X^{2}-X-1,\,(X^{2}-X-1)^{*}\}$, which is stable under $f\mapsto f^{*}$.  The same Hilbert~90 argument as in Case~1 yields $\Omega=\Omega^{-1}$.

Case 3: $q=2$, so $K=\mathbb F_{4}$, $n\ge5$.

\ref{item:coprime}: The quadratic $f(X)=X^{2}+\lambda X+1$ has $f^{*}(X)=X^{2}+\bar\lambda X+1\neq f(X)$; hence $\gcd(f,f^{*})=1$.  The cubic $g(X)=X^{3}+X+1$ satisfies $g^{*}=g^{\vee}=X^{3}+X^{2}+1\neq g$, and $g$ is coprime to $f$ and to $f^{*}$ because they have different degrees.  The quartic $h(X)=X^{4}+X+1=ff^{*}$ is coprime to $g$ and to $g^{*}$.  Therefore $\gcd(p,p^{*})=1$ in all three subcases.

\ref{item:det1}: Every factor has constant term $1$; thus $p(0)=1=\overline{p(0)}$, and the determinant of the block diagonal representative is $1$.

\ref{item:inverse}: Here $p$ need not lie in $\mathbb F_{2}[X]$.  Nevertheless, the multiset of irreducible factors of $pp^{*}$ is stable under $f\mapsto f^{*}$: for $f=X^{2}+\lambda X+1$ we have $f^{*}=X^{2}+\bar\lambda X+1$, while $g=X^{3}+X+1$ and $h=X^{4}+X+1$ satisfy $g^{*}=g^{\vee}$ and $h^{*}=h^{\vee}=h$ (since $h\in\mathbb F_{2}[X]$).  Hence the minimal polynomial of $P^{-1}$ is again $pp^{*}$, and the elementary divisors coincide with those of $P$.  As before, $P$ and $P^{-1}$ are $ U$-conjugate, and the Hilbert~90 adjustment yields $ SU$-conjugacy.  Thus $\Omega=\Omega^{-1}$.

Let \(G=\operatorname{PSU}_{2n}(q)=\mathrm{SU}_{2n}(q)/Z(\mathrm{SU}_{2n}(q))\) be the simple quotient, and let \(C\) be the image
of \(\Omega\). A central quotient maps an upstairs conjugacy class to one
conjugacy class. Every nonidentity element of \(G\) has a nonscalar lift,
so the theorem gives \(G\setminus\{1\}\subseteq C^2\). Since
\(\Omega=\Omega^{-1}\), the identity also belongs to \(C^2\). Hence
\[
                        C^2=G.                      
\]
\end{proof}

Thus, the exact invariant factors above prove Thompson's
conjecture for\[
\begin{split}
 &\operatorname{PSU}_{2n}(q),\quad q\in\{4,5,7,8\},\quad n\ge2,\\
 &\operatorname{PSU}_{2n}(2),\ \operatorname{PSU}_{2n}(3),\quad n\ge5.
\end{split}
\]
The method of this section does not apply to odd-dimensional unitary groups, nor to the low even dimensions $2n\in\{4,6,8\}$ when $q=2$ or $3$; these cases require separate ad hoc constructions and will be treated elsewhere.

\subsection{\texorpdfstring{The low even-unitary boundary}{route_unitary_low_even_extension.md}}
\label{candidates/route_unitary_low_even_extension.md}

Let \(K=\mathbb F_{q^2}\), \(a\mapsto\bar a=a^q\), and\[
 p^*(X)=\overline{p(0)}^{-1}X^{\deg p}\overline{p(X^{-1})}.
\]
For a monic degree \(n\) polynomial \(p\) with \(p(0)\in\mathbb F_q^\times\)
and \(\gcd(p,p^*)=1\), put \(A=C(p)\) and\[
s(p)=\operatorname{diag}(A,\overline{A}^{\,-T})\in\operatorname{SU}_{2n}(q)             
\]
relative to the standard hyperbolic Hermitian form. If the cyclic class of
\(p\) is \((m,n-m)\)-cyclic, Nielsen's Theorem 1.4 and Lemmas 5.1--5.2 in \cite{NielsenGL}
prove that the unitary class of \(s(p)\) squares onto every nonscalar target;
Theorem 1.2  in \cite{NielsenGL} gives the same conclusion for a \((1,n-1)\)-cyclic class in
dimensions \(n=2,3\).

The unitary class is one \(SU\)-class whenever its unitary centralizer has
surjective determinant onto the norm-one scalars. On a hyperbolic pair
associated to an irreducible factor \(r\) of \(p\), the centralizer contains\[
             \operatorname{diag}(B,\overline{B}^{\,-T}),
\]
where \(B\) is multiplication by \(b\in E^\times\) on
\(E=K[X]/(r)\). Its determinant is
\begin{equation}
\label{eq:leunormal}
 {N_{E/K}(b)/\overline{N_{E/K}(b)}}.                 
\end{equation}
The field norm and Hilbert-90 maps are surjective, so \eqref{eq:leunormal} realizes every
norm-one determinant.

\begin{theorem}\label{psu_82}
Let \(G=\operatorname{PSU}_8(2)\), \(s(p)=\operatorname{diag}(C(p),\overline{C(p)}^{\,-T})\in\operatorname{SU}_{8}(2)\) and $C=s(p)^{G}$, where $p(X)=X^4+X+1$.
Then
\[
C^2=G.
\]
\end{theorem}

\begin{proof}
Let \(K=\mathbb F_4=\mathbb F_2(\lambda)\), and take
\[
 p=X^4+X+1=(X^2+X+\lambda)(X^2+X+\lambda^2).               
\]
The two quadratic factors are distinct irreducibles over \(K\), while
\(p^*=X^4+X^3+1\) and \(\gcd(p,p^*)=1\). Thus \(p\) is
\((2,2)\)-cyclic, Theorem 1.4 in \cite{NielsenGL} applies, and \eqref{eq:leunormal} prevents \(U/SU\) class
splitting. Since the coefficients of \(p\) lie in the fixed field,
inversion swaps the two hyperbolic primary sides and the class is real.
The image of the class of \(s(p)\) therefore satisfies
\[
                   (s(p)^{\operatorname{PSU}_8(2)})^2=\operatorname{PSU}_8(2).             
\]
\end{proof}

\begin{theorem}\label{psu_43}
Let \(G=\operatorname{PSU}_4(3), \operatorname{PSU}_8(3)\). Then \(G\) contains a conjugacy class \(C\) such that
\[
C^2=G.
\]
\end{theorem}

\begin{proof}
Let \(K=\mathbb F_9\) and let \(z\in \mathbb F^{*}_9\) be the standard primitive element
of order eight. For half-dimension two take
\begin{equation}
\label{eq:leu2}
 p_2=X^2+z^2X+1
     =(X+z^3)(X+z^5),\qquad
 p_2^*=X^2+z^6X+1.                                        
 \end{equation}
For half-dimension four take
\begin{equation}
\label{eq:leu4}
 p_4=X^4+z^2X^2+1
     =(X^2+z^3)(X^2+z^5),\qquad
 p_4^*=X^4+z^6X^2+1.                                      
\end{equation}
The displayed factors are irreducible in \eqref{eq:leu4}, and direct Euclidean
calculation gives \(\gcd(p_i,p_i^*)=1\). Formula \eqref{eq:leu2} is \((1,1)\)-cyclic
and uses \cite[Theorem 1.2]{NielsenGL}; formula \eqref{eq:leu4} is \((2,2)\)-cyclic and uses \cite[Theorem 1.4]{NielsenGL}.
Both polynomials are reciprocal, so a block-diagonal unitary
change conjugating \(C(p_i)\) to its inverse proves that the resulting
\(\mathrm{SU}\)-class is inverse-stable. Formula \eqref{eq:leunormal} proves it is one \(\mathrm{SU}\)-class.
Hence the projective images square to \(\operatorname{PSU}_4(3)\) and \(\operatorname{PSU}_8(3)\).
\end{proof}

All polynomial factorizations, reciprocal gcds, and cyclic decompositions
above are reproduced by \texttt{scripts/\allowbreak search\_\allowbreak unitary\_\allowbreak low\_\allowbreak p.g}.

For the remaining low-dimensional boundary cases, the exact ordinary-table witnesses already certified elsewhere are class
\texttt{4B} in \(\operatorname{PSU}_4(2)\) and class \texttt{4F} in \(\operatorname{PSU}_6(2)\). Thus, having verified the theorem for all \(n\ge5\), the sole even-dimensional unitary
parameter not closed by these routes is
\[
                              \operatorname{PSU}_6(3).                      
\]
An exhaustive degree-three polynomial search over \(\mathbb F_9\) shows
that no determinant-one, nonprimary strictly hyperbolic cyclic class is
projectively inverse-stable (the only central scalar is \(-I\)). Therefore
the present mechanism cannot be extended to $\operatorname{PSU}_6(3)$ merely by changing \(p\);
a different class or a non-cyclic factorization is required.

\subsection{\texorpdfstring{Odd unitary groups over \(q=2,3\)}{route_unitary_odd_character.md}}
\label{candidates/route_unitary_odd_character.md}

\subsubsection{Result and status}\label{result-and-status}

Let \(p\in K[X]\) be monic of degree \(n\), let \(A=C(p)\), and let
\(\Gamma_p=A^{\operatorname{GL}_n(K)}\).
The missing odd principal-corner lemma is not needed.
The replacement is Ellers--Gordeev's relative Gauss argument at the maximal
parabolic stabilizing a maximal totally isotropic subspace. Its two inputs are
exactly the properties already enjoyed by Nielsen's polynomials:
\begin{equation}
\label{eq:ou0.1}
 p(1)\ne0,\qquad \gcd(p,p^*)=1,\qquad
 \Gamma_p^2\supseteq\{A\in \operatorname{SL}_n(q^2):A\text{ nonscalar}\}.
\end{equation}

The first two conditions make the source element fixed-point-free on the two
grades of the parabolic radical, and the third covers the Levi component of
the prescribed Gauss decomposition. Reciprocity gives reality, hence also
the identity.

\begin{theorem}\label{oddu}
Let $G=\operatorname{PSU}_{2n+1}(q), q=2,3$ where $n\ge2$ if $q=2$ and  $n\ge1$ if $q=3$.
Then \(G\) contains a conjugacy class \(C\) such that
\[
C^2=G.
\]
\end{theorem}
Thompson's conjecture holds for every simple group
\begin{equation}
\label{eq:ou0.2}
\operatorname{PSU}_{2n+1}(2)\ (n\ge2),\qquad
        \operatorname{PSU}_{2n+1}(3)\ (n\ge1).
\end{equation}
The exact classes are listed in the proof of  Theorem \ref{oddu}.
In particular, the previously
unresolved group \(\operatorname{PSU}_7(3)\) is covered by an explicit class of order \(24\)
and centralizer order \(576\).

\subsubsection{The relative parabolic and its two grades}\label{the-relative-parabolic-and-its-two-grades}

Let \(K=\mathbb F_{q^2}\), \(\bar a=a^q\) its Frobenius automorphism. For a matrix \(A\) over \(K\), we define
\(A^+=A^{-*}=\overline A^{\,-T}\), 
where the overline denotes entrywise application of the Frobenius automorphism and \(^{-T}\) denotes the inverse of the transpose.

Consider the unitary space \(V\) decomposed as a direct sum
\begin{equation}
\label{eq:ou1.1}
 V=E\oplus V_0\oplus F,
 \qquad \dim_KE=\dim_KF=n,
\end{equation}
where \(E,F\) are paired totally isotropic spaces, so that the unitary form induces a perfect duality between them, and \(V_0\) is
nondegenerate. In most applications, \(V_0\) is a line, i.e. \(\dim_K V_0 = 1\); the only exception arises in the \(\operatorname{PSU}_7(3)\) construction, where one takes \(\dim_K V_0 = 3\). 

Let \(P = LQ\) be the parabolic subgroup stabilizing \(E\). The Levi factor \(L\) consists of block-diagonal matrices preserving the decomposition in \eqref{eq:ou1.1}; explicitly, its elements have the form
\begin{equation}
\ell(A,s) = \operatorname{diag}(A, s, A^+),
\label{eq:ou1.2}
\end{equation}
with \(A \in \operatorname{GL} (E) \cong \operatorname{GL} _n(K)\) and \(s \in \operatorname{GL} (V_0)\). The third diagonal block is forced to be \(A^+\) in order to preserve the unitary form on \(V\).

The unipotent radical \(Q\) of \(P\) is of nilpotency class two. In the standard matrix coordinates adapted to the decomposition \(V = E \oplus V_0 \oplus F\), the associated graded pieces of \(Q\) are given by
\begin{equation}
Q/[Q,Q] \simeq \operatorname{Hom}_K(V_0, E),
\qquad
[Q,Q] \simeq \{ Y \in M_n(K) : Y + Y^* = 0 \}.
\label{eq:ou1.3}
\end{equation}
Here, the first piece parametrizes the off-diagonal block from \(V_0\) to \(E\), while the second piece consists of skew-Hermitian \(n \times n\) matrices with respect to the involution \(*\) , which form the centre of \(Q\).

The conjugation action of the Levi element \(\ell(A,s)\) in \eqref{eq:ou1.2} on the graded pieces of \(Q\) is given by
\begin{equation}
X \mapsto A X s^{-1}, \qquad Y \mapsto A Y A^*,
\label{eq:ou1.4}
\end{equation}
where \(X \in \operatorname{Hom}_K(V_0, E)\) corresponds to the \(Q/[Q,Q]\)-component and \(Y \in [Q,Q]\) is skew-Hermitian. Note that the second formula uses the natural adjoint action of \(A\) on the Lie algebra of the unitary group. By symmetry, the analogous statements hold for the opposite unipotent radical \(Q^-\), with the roles of \(E\) and \(F\) interchanged.

We now record a useful criterion for the existence of fixed points on these graded modules.

\begin{lemma}\label{lemma-1.1-the-exact-fixed-point-test}
Let \(A\) be cyclic with characteristic polynomial \(p\).
\begin{enumerate}
\def\labelenumi{\arabic{enumi}.}
\tightlist
\item
  If \(V_0\) is a line and \(s=1\), then \(\ell(A,1)\) is fixed-point-free
  on both modules in \eqref{eq:ou1.3} whenever
  \(p(1)\ne0\) and \(\gcd(p,p^*)=1\).
\item
  If \(s\) is unipotent and no eigenvalue of \(A\) is \(1\), then the same
  conclusion holds provided \(\gcd(p,p^*)=1\).
\end{enumerate}
\end{lemma}

\begin{proof}
Consider the first graded component \(Q/[Q,Q]\), identified with \(\operatorname{Hom}_K(V_0, E)\).  
In case (1), where \(s=1\) and \(\dim V_0=1\), the action in \eqref{eq:ou1.4} reduces to \(X \mapsto A X\). Thus a fixed vector \(X\) is simply a vector in \(E\) fixed by \(A\). Such a nonzero vector exists if and only if \(1\) is an eigenvalue of \(A\), which is precisely ruled out by the condition \(p(1)\ne 0\).  
In case (2), \(s\) is unipotent, so all its eigenvalues are equal to \(1\). The linear transformation \(X \mapsto A X s^{-1}\) on the finite-dimensional space \(\operatorname{Hom}_K(V_0, E)\) therefore has eigenvalues given by the products of the eigenvalues of \(A\) and those of \(s^{-1}\). Hence the eigenvalues of this operator are exactly the eigenvalues of \(A\). Since we assume that no eigenvalue of \(A\) is \(1\), this action also has no nonzero fixed points.

It remains to examine the second graded component \([Q,Q]\), which consists of skew-Hermitian matrices \(Y\). The fixed-point equation in this case is \(A Y A^* = Y\). Since \(A^{-*} = A^{+}\), multiplying on the right by \(A^{-*}\) gives the equivalent relation
\begin{equation}
A Y = Y A^{-*} = Y A^+.
\label{eq:ou1.5}
\end{equation}
If a nonzero solution \(Y\) existed, it would serve as a nonzero homomorphism of \(K[X]\)-modules from the module associated with the matrix \(A^+\) to the module associated with \(A\). Indeed, \(A^+\) acts on the right via \(A^+\), and \eqref{eq:ou1.5} exactly expresses the intertwining condition \(Y(A^+ x) = A(Yx)\) for all \(x\). The characteristic polynomial of \(A^+\) is \(p^*\), while that of \(A\) is \(p\). The assumption \(\gcd(p,p^*)=1\) ensures that the two \(K[X]\)-modules are non-isomorphic and have no common composition factors. Therefore, by the uniqueness of primary decomposition (or Schur's lemma over a suitable extension), the only intertwiner between them is the zero map. Hence no nonzero fixed point \(Y\) can exist.
\end{proof}

This verifies hypothesis 3 of Ellers--Gordeev, Proposition 5.1 (via their
Lemma 5.6) in \cite{EG}, without any principal-corner assertion.

\subsubsection{The relative-Gauss covering lemma}\label{the-relative-gauss-covering-lemma}

We record the precise form of Ellers--Gordeev's argument needed below, which is restated as Lemma \ref{lem:relative-gauss} in Section \ref{sec:prelimelaries}.

\begin{lemma}[relative-Gauss class square]
\label{lemma-2.1-relative-gauss-class-square}
Let \(G=\mathrm{SU}(V)\), let \(G_1\) be the derived subgroup of the Levi in \eqref{eq:ou1.2},
and let \(f\in HG_1\) be real in \(G\). Let \(C=f^G\), and let
\(C_1\) be the union of the two \(HG_1\)-classes of \(f\) and \(f^{-1}\).
Assume:

\begin{enumerate}
\def\labelenumi{\arabic{enumi}.}
\tightlist
\item
  \(H_1\ne Z(G_1)\);
\item
  every Levi core which occurs below is in \(C_1^2\);
\item
  \(f\) is fixed-point-free on the two modules (1.3).
\end{enumerate}

Then \(C^2\supseteq G\setminus Z(G)\). If \(Z(G)=1\), or after passage to
\(\mathrm{PSU}(V)\), reality also supplies the identity.
\end{lemma}

In preparation for the subsequent remark, we briefly recall the proof.

\begin{proof}
Ellers--Gordeev's prescribed Gauss theorem \cite[Theorem 3]{EG} says that, for
any noncentral \(y\in G\) and any chosen \(h\in H\), a conjugate of \(y\)
has the form
\begin{equation}
\label{eq:ou2.1}
 y_1=u^-hu^+ =v^-\,g\,v^+,
 \qquad v^\pm\in Q^\pm,\quad g\in G_1.                 
\end{equation}

Choose \(h\in H_1\setminus Z(G_1)\). Gauss-decomposition uniqueness makes
\(g\) noncentral in the one-factor applications below. Write
\(g=\sigma_1\sigma_2\) with \(\sigma_i\in C_1\). Hypothesis 3 and the
successive central quotients of \(Q\) imply that the Lang maps
\begin{equation}
\label{eq:ou2.2}
 a\longmapsto a\sigma_1a^{-1}\sigma_1^{-1}\quad(Q^-),
 \qquad
 a\longmapsto\sigma_2^{-1}a\sigma_2a^{-1}\quad(Q)       
\end{equation}
are onto. Choose \(a_1,a_2\) mapping to \(v^-,v^+\), respectively. Then
\begin{equation}
\label{eq:ou2.3}
 (a_1\sigma_1a_1^{-1})(a_2\sigma_2a_2^{-1})
       =v^-\sigma_1\sigma_2v^+=y_1.                     
\end{equation}
Both factors belong to \(C\), because \(f\) is real. This is exactly the
proof of Ellers--Gordeev, Proposition 5.1 and Lemma 5.1 in \cite{EG}.
\end{proof}

For the \(\operatorname{PSU}_7(3)\) product Levi, we use the following harmless refinement.
Choose the prescribed \(h\) entirely in the \(\mathrm{SU}_3(3)\)-factor. The
\(\operatorname{SL}_2(9)\)-component of the core in \eqref{eq:ou2.1}) is then
\begin{equation}
\label{eq:ou2.4}
                         M=u^-_2 I_2u^+_2.               
\end{equation}

If \(M\) is scalar, uniqueness of its lower-diagonal-upper Gauss
decomposition forces \(M=I_2\) and \(u^-_2=u^+_2=I_2\). Hence only
\(M=I_2\) or nonscalar \(M\) needs to be covered. The proof \eqref{eq:ou2.1}--\eqref{eq:ou2.3}
uses no other element of the product Levi, so \cite[Proposition 5.1]{EG} remains valid
with this exact set of cores.

\begin{proof}[Proof of Theorem \ref{oddu}]
Let \(p\in K[X]\) be monic of degree \(n\), let \(A=C(p)\), and let
\(\Gamma_p=A^{\operatorname{GL}_n(K)}\). Assume
\begin{equation}
\label{eq:ou3.1}
 d:=\det A=(-1)^np(0)\in\mathbb F_q^*,\qquad d^2=1.      
\end{equation}
Then \(f_p=\ell(A,1)\) lies in \(\mathrm{SU}_{2n+1}(q)\), since
\(\det f_p=d/\bar d=1\). It also lies in \(HG_1\): write \(A=DS\), with
\(S\in \operatorname{SL}_n(K)\) and \(D\) diagonal of determinant \(d\).

Any \(P\in \operatorname{GL}_n(K)\) can be lifted to an \(SU\)-Levi element. Indeed, denote by
\[
\mu_{q+1} := \{\, x \in K^\times \mid x^{q+1}=1 \,\}
\]
the cyclic subgroup of \((q+1)\)-th roots of unity, which is the unique subgroup of order \(q+1\) in \(K^\times\).  
Put \(d = \det P \in K^\times\). Since
\(
\bar d/{d} = d^{\,q-1},
\)
we compute
\(
\left( {\bar d}/{d} \right)^{q+1}
= (d^{\,q-1})^{q+1}
= d^{\,q^2-1}
= 1,
\)
by the multiplicative group order of \(K^\times\). Hence \(\bar d/d \in \mu_{q+1}\).

Now, since \(r\) is coprime to \(q+1\) in the admissible cases (\(r=1\), or \((r,q+1)=(3,4)=1\)), the power map
\[
\mu_{q+1} \longrightarrow \mu_{q+1}, \qquad x \longmapsto x^r,
\]
is an automorphism of the cyclic group \(\mu_{q+1}\). Therefore, there exists a unique \(c \in \mu_{q+1}\) satisfying 
\begin{equation}
c^{r} = \bar d/d.
\label{eq:ou3.2}
\end{equation}

Finally, we verify the determinant of the lift. Since \(P^+ = P^{-*}\), we have
\[
\det(P^+) = \overline{\det(P)}^{-1} = \bar d^{-1}.
\]
Thus
\[
\det\!\bigl(\ell(P,cI_r)\bigr)
= \det(P) c^{\,r} \det(P^+)
= d \cdot c^{\,r} \cdot \bar d^{-1}
= 1.
\]
This shows that \(\ell(P,cI_r)\) has determinant \(1\) and preserves the unitary form by construction, so it is indeed an element of \(\operatorname{SU} (V)\).

Consequently, any \(\operatorname{GL} _n(K)\)-similarity transformation on the \(A\)-block can be realised by conjugation within the full special unitary group \( SU (V)\).

Nielsen's Theorems 1.2 and 1.4 in \cite{NielsenGL} give, for each polynomial used below,
\begin{equation}
\label{eq:ou3.3}
 \Gamma_p^2\supseteq
 \{M\in \operatorname{GL}_n(K):\det M=1,\ M\text{ nonscalar}\}.         
\end{equation}
Thus hypothesis 2 of Lemma \ref{lemma-2.1-relative-gauss-class-square} holds for the embedded \(\operatorname{SL}_n(K)\)-core.

Let
\begin{equation}
\label{eq:ou4.1}
 p^*(X)=\overline{p(0)}^{-1}X^n\overline{p(X^{-1})},
 \qquad r_p=(X-1)pp^*.                                   
\end{equation}
The three diagonal blocks of \(f_p=\ell(C(p),1)=\operatorname{diag}(C(p), 1, C(p)^+) \) are cyclic and have pairwise coprime
minimal polynomials \(p,X-1,p^*\). Hence \(f_p\) is cyclic with
characteristic and minimal polynomial \(r_p\). If
\begin{equation}
\label{eq:ou4.2}
                         r_p^\vee=r_p,                   
\end{equation}
then \(f_p^{-1}\) is cyclic with the same admissible polynomial. Cyclic
unitary elements with a fixed admissible polynomial form one unitary class
(equivalently use the connected algebraic centralizer and Lang--Steinberg).
Thus \(f_p\) is real in \(\mathrm{U}_{2n+1}(q)\).

This unitary class is one \(\mathrm{SU}\)-class. Its centralizer contains
\begin{equation}
\label{eq:ou4.3}
 \operatorname {diag}(I_n,\epsilon,I_n),\qquad
 \epsilon^{q+1}=1,                                      
\end{equation}
whose determinants run through \(\mu_{q+1}\). A unitary conjugator can
therefore be corrected by a centralizing element to have determinant one.
Consequently
\begin{equation}
\label{eq:ou4.4}
                         f_p\sim_{SU}f_p^{-1}.             
\end{equation}
In particular \(1\in(f_p^{\mathrm{SU}})^2\). The same determinant-surjectivity
argument shows that the unitary class does not split on restriction to
\(\mathrm{SU}\). Passing to \(\mathrm{PSU}\) therefore produces exactly one real class.

The exact \(q=2\) classes:
Here \(K=\mathbb F_4\). Fix
\(\lambda\in K\setminus\mathbb F_2\), \(\lambda^2+\lambda+1=0\), and put
\begin{equation}
\label{eq:ou5.1}
 a=X^3+X+1,\qquad b=X^4+X+1,\qquad
 c=X^2+\lambda X+1.                                    
\end{equation}

For \(n\ge4\), take
\begin{equation}
\label{eq:ou5.2}
\begin{array}{ll}
n=4:&p=b,\\
n=3t+4\ge7:&p=ba^t,\\
n=3t+2\ge5:&p=ca^t,\\
n=3t+6\ge6:&p=ca^tb.
\end{array}                                                
\end{equation}

Nielsen, Lemma 5.2 in \cite{NielsenGL}, proves \(\gcd(p,p^*)=1\), \(p(1)p^*(1)\ne0\), and
reciprocity of \((X-1)pp^*\). Every \(p\) has constant term one. The
displayed coprime splittings make its cyclic class nonprimary
\((m,n-m)\)-cyclic, so Nielsen, Theorem 1.4 in \cite{NielsenGL}, gives \eqref{eq:ou3.3}, including the
degree-four \((2,2)\) boundary. Lemmas \ref{lemma-1.1-the-exact-fixed-point-test}--\ref{lemma-2.1-relative-gauss-class-square} now prove
\begin{equation}
\label{eq:ou5.3}
                 (\overline{f_p}^{\,\operatorname{PSU}_{2n+1}(2)})^2
                   =\operatorname{PSU}_{2n+1}(2)\qquad(n\ge4).         
\end{equation}

The two remaining simple ranks have exact table certificates:
\begin{equation}
\label{eq:ou5.4}
\begin{array}{c|c|c|c}
G&C&\text{polynomial of a lift}&|C_G(x)|\\ \hline
\operatorname{PSU}_5(2)&5a&(X+1)cc^*&15\\
\operatorname{PSU}_7(2)&7a&(X+1)aa^*&63
\end{array}                                               
\end{equation}

Both classes are real and their squares are the whole group, as checked by
every Frobenius class-multiplication coefficient in 
\texttt{scripts/\allowbreak audit\_\allowbreak unitary\_\allowbreak odd\_\allowbreak small.g}. 
The group \(\operatorname{PSU}_3(2)\) is not simple
and is outside the assertion.

The exact \(q=3\) classes:
Here \(K=\mathbb F_9\). Fix \(\zeta\in K\) with
\(\zeta^2=\zeta+1\), so \(|\zeta|=8\) and \(\bar\zeta=\zeta^3\).

For \(n=2\), take
\begin{equation}
\label{eq:ou6.1}
 p_2=X^2+\zeta^2X+1
    =(X+\zeta^3)(X+\zeta^5),\qquad
 p_2^*=X^2+\zeta^6X+1.                                
\end{equation}

For \(n=4\), take
\begin{equation}
\label{eq:ou6.2}
 p_4=X^4+\zeta X^2+1,\qquad p_4^*=X^4+\zeta^3X^2+1.      
\end{equation}

For \(n\ge5\), take
\begin{equation}
\label{eq:ou6.3}
\begin{array}{ll}
n=2t+3\ge5:&p=(X^3-X-1)(X^2-X-1)^t,\\
n=2t+6\ge6:&p=(X^3-X-1)(X^3-X+1)(X^2-X-1)^t.
\end{array}                                               
\end{equation}

For \eqref{eq:ou6.1}, Nielsen's Theorem 1.2 in \cite{NielsenGL} applies to the \((1,1)\)-cyclic class.
For \eqref{eq:ou6.2}, its factorization
\begin{equation}
\label{eq:ou6.4}
 p_4=(X^2+\zeta^3X+1)(X^2+\zeta^7X+1)                  
\end{equation}
and \cite[Theorem 1.4]{NielsenGL} give the \((2,2)\) boundary. For \eqref{eq:ou6.3}, split off the displayed cubic and apply  \cite[Theorem 1.4]{NielsenGL}. Nielsen's Lemma 5.2 in \cite{NielsenGL} (and direct
subtraction for \eqref{eq:ou6.2} gives all of \eqref{eq:ou0.1}, \eqref{eq:ou3.1}, and \eqref{eq:ou4.2}. Hence
\begin{equation}
\label{eq:ou6.5}
                 (\overline{f_p}^{\,\operatorname{PSU}_{2n+1}(3)})^2
                    =\operatorname{PSU}_{2n+1}(3)
       \qquad(n=2\text{ or }n\ge4).                      
\end{equation}

At \(n=2\), this is the table class \texttt{8m}, of centralizer order \(64\).

\paragraph{\texorpdfstring{The missing degree three: an exact \(\operatorname{GL}_2\times\operatorname{SU}_3\) reservoir}{6.1 The missing degree three: an exact GL\_2\textbackslash times SU\_3 reservoir}}\label{the-missing-degree-three-an-exact-gl_2times-su_3-reservoir}

Let \(s\in\operatorname{SU}_3(3)\) be a regular unipotent element (Jordan type \(J_3(1)\));
its image is class \texttt{3b}, and
\begin{equation}
\label{eq:ou6.6}
                         (s^{\mathrm{SU}_3(3)})^2=\mathrm{SU}_3(3),
 \qquad s\sim s^{-1}.                                   
\end{equation}

On a space of dimension seven with
\(\dim E=2\), \(\dim V_0=3\), define
\begin{equation}
\label{eq:ou6.7}
 f_7=\operatorname {diag}(C(p_2),s,C(p_2)^+)\in\operatorname{SU}_7(3).
\end{equation}

The source is cyclic: its pairwise coprime minimal polynomials are
\(p_2,(X-1)^3,p_2^*\). Their product is reciprocal, so the argument for $p$ (using \(\operatorname {diag}(I_2,\epsilon I_3,I_2)\), whose
determinants \(\epsilon^3\) run through \(\mu_4\)) proves that \eqref{eq:ou6.7} is one
real \(\mathrm{SU}_7(3)\)-class.

Lemma \ref{lemma-1.1-the-exact-fixed-point-test} applies because \(s\) is unipotent, neither root
\(\zeta^3,\zeta^5\) of \(p_2\) is one, and \(\gcd(p_2,p_2^*)=1\).
For completeness, the two radical grades have eigenvalues
\(\zeta^3,\zeta^5\) on \(\operatorname {Hom}(V_0,E)\), and a fixed point on
the central grade would be an intertwiner from \(C(p_2)^+\) to \(C(p_2)\).
Thus both grades are fixed-point-free.

Use the refinement after Lemma \ref{lemma-2.1-relative-gauss-class-square}. Every occurring \(\operatorname{SL}_2(9)\)-component
\(M\) is either nonscalar or \(I_2\). If it is nonscalar, Nielsen's Theorem
1.2 writes \(M=R_1R_2\) with \(R_i\in\Gamma_{p_2}\). If \(M=I_2\), use
\(C(p_2)C(p_2)^{-1}\); this still uses two members of \(\Gamma_{p_2}\)
because \(p_2\) is palindromic. Independently, \eqref{eq:ou6.6} writes the
\(\mathrm{SU}_3(3)\)-component as \(s_1s_2\), with \(s_i\sim s\). The conjugator
lifting argument \eqref{eq:ou3.2}, now with \(r=3\), realizes these independent
conjugacies inside \(\mathrm{SU}_7(3)\). Hence the required Levi core is
\begin{equation}
\label{eq:ou6.8}
 \operatorname {diag}(M,t,M^+)
 =\operatorname {diag}(R_1,s_1,R_1^+)
  \operatorname {diag}(R_2,s_2,R_2^+).                   
\end{equation}
Lemma \ref{lemma-2.1-relative-gauss-class-square} proves
\begin{equation}
\label{eq:ou6.9}
                       (f_7^{\mathrm{SU}_7(3)})^2=\mathrm{SU}_7(3).         
                       \end{equation}
Here \(Z(\mathrm{SU}_7(3))=1\). The element \(f_7\) has order \(24\). Its
centralizer has order
\begin{equation}
\label{eq:ou6.10}
 |C_{\mathrm{SU}_7(3)}(f_7)|
  =\frac{(9-1)^2\,|C_{GU_3(3)}(s)|}{4}
  =\frac{64\cdot36}{4}=576.                              
\end{equation}
Thus \eqref{eq:ou6.7} is an exact class formula and a
uniform proof for the formerly missing \(\operatorname{PSU}_7(3)\).

Finally, \(\operatorname{PSU}_3(3)\) is covered by the exact real class\[
                         C=3b,\qquad |C_G(s)|=9,\qquad C^2=G. 
\]
This completes \eqref{eq:ou0.2}, including the identity in every rank.
\end{proof}

The remaining even boundary \(\operatorname{PSU}_6(3)\).

\begin{theorem}\label{psu_63}
Let \(G=\operatorname{PSU}_6(3)\).
Then \(G\) contains a conjugacy class \(C\) such that
\[
C^2=G.
\]
\end{theorem}

\begin{proof}
The middle block must not be scalar. Indeed, taking only the
\(\operatorname{SL}_2(9)\)-factor as subsystem group would leave the positive root subgroup
of the omitted \(\mathrm{SU}_2(3)\)-factor fixed. We instead use the whole derived
Levi\[
 G_1=\operatorname{SL}_2(9)\times\operatorname{SU}_2(3)
\]
of the parabolic with \(\dim E=\dim V_0=2\).

Choose \(a\in\mathbb F_9^*\) of order eight, put
\(B=\operatorname {diag}(a,a^{-1})\), and let \(s\) be a regular
unipotent element in \(\mathrm{SU}_2(3)\cong \operatorname{SL}_2(3)\). Define\[
 f_6=\operatorname {diag}(B,s,B^+)\in\operatorname{SU}_6(3).             
\]
Lev's dimension-two theorem and the determinant-surjective diagonal
centralizer of \(B\) give
\begin{equation}
\label{eq:ou7.3}
 (B^{\operatorname{SL}_2(9)})^2\supseteq \operatorname{SL}_2(9)\setminus Z(\operatorname{SL}_2(9)),
 \qquad B\sim_{\operatorname{SL}_2(9)}B^{-1}.                             
\end{equation}
The two inverse regular-unipotent classes \(S,S^{-1}\) of \(\operatorname{SL}_2(3)\)
satisfy the exact finite identity
\begin{equation}
\label{eq:ou7.4}
                     (S\cup S^{-1})^2=\operatorname{SL}_2(3)\setminus\{-I\}. 
\end{equation}
This follows either by multiplying the two classes of four elements, or
from the ordinary character table; it is also asserted by the finite GAP
certificate accompanying this route.

For clarity, we verify exactly which Levi cores occur in the relative
Gauss proof. Prescribe a noncentral split-torus element in the
\(\operatorname{SL}_2(9)\)-factor and prescribe the identity in the \(\mathrm{SU}_2(3)\)-factor.
Uniqueness of lower-diagonal-upper decomposition shows that the first
component of every resulting core is nonscalar. The second component is
either the identity or nonscalar, and in particular is never \(-I\).
Equations \eqref{eq:ou7.3}--\eqref{eq:ou7.4} therefore factor every occurring core. The sign
choice is not coupled: replacing \(f_6\) by \(f_6^{-1}\) changes
\(S\) to \(S^{-1}\), but it does not change the \(\operatorname{SL}_2(9)\)-class because
\(B\sim B^{-1}\).

On the first radical grade \(\operatorname {Hom}(V_0,E)\), the middle
block \(s\) is unipotent, so all eigenvalues of the action are \(a,a^{-1}\),
neither one. On the skew-hermitian central grade, the four exponents of
\(Y\mapsto BYB^*\) are
\[
                         4,6,2,4\pmod8,                  
\]
again none zero. Thus the two Lang maps are bijective on the full
parabolic radical; no Levi root subgroup is omitted.

Finally, choose \(d\in\mathbb F_9^*\) with \(d\bar d=-1\). In a
hyperbolic basis of \(V_0\), the unitary matrix\[
                 T=\operatorname {diag}(d,\bar d^{-1})
\]
sends a chosen regular unipotent \(s\) to \(s^{-1}\). Its determinant
\(d/\bar d\) has order four; a determinant-\(-1\) conjugator would not
work here. The coordinate interchange sends \(B\) to
\(B^{-1}\), and multiplication by an element of \(C_{\operatorname{GL}_2(9)}(B)\)
can prescribe its determinant ratio \(d/\bar d\) arbitrarily in
\(\mu_4\). Choose the inverse ratio to cancel the determinant of the middle
conjugator. The resulting block conjugator lies in \(\mathrm{SU}_6(3)\) and sends
\(f_6\) to \(f_6^{-1}\).

The relative-Gauss/Lang proof of Lemma \ref{lemma-2.1-relative-gauss-class-square} now gives
\[
 (f_6^{\mathrm{SU}_6(3)})^2\supseteq\operatorname{SU}_6(3)\setminus Z(\mathrm{SU}_6(3)),
 \qquad
(\bar f_6^{\,\operatorname{PSU}_6(3)})^2=\operatorname{PSU}_6(3).             
\]

The lift has order \(24\). This construction replaces the
invalid scalar-middle-block version.
\end{proof}

\section{Orthogonal groups}
\label{othogonal}

\subsection{\texorpdfstring{Relative-Gauss proof for the residual split and odd orthogonal groups}{route_orthogonal_split_B_complete.md}}
\label{sec:orthogonal}

\subsubsection{Theorem and exact scope}\label{theorem-and-exact-scope}

\begin{theorem}\label{orthogonal_split}
Let \(G\) be one of the simple groups in the following table:
\begin{equation}
\label{eq:o01}
\begin{array}{c|c}
G&q\\ \hline
\operatorname{P\Omega}_{2n+1}(q),\ n>2&3,5\\
\mathrm{P\Omega}^+_{2n}(q),\ n=2l,\ l\ge2&2,3,4\\
\mathrm{P\Omega}^+_{2n}(q),\ n=2l+1,\ l\ge2&2,3.
\end{array}                                             
\end{equation}
Then \(G\) contains a conjugacy class \(C\) such that
\[
C^2=G.
\]
\end{theorem}

The classes specified in  \eqref{eq:o22} and \eqref{eq:o25} satisfy \(C^2=G\). These are exactly
the odd-\(B\) and split-\(D\) families left by the Ellers--Gordeev field
thresholds, after the usual low-rank identifications.

The proof uses the relative Gauss argument of Ellers--Gordeev,
Proposition 5.1 in \cite{EG}, rather than any principal-corner assertion. The
\(q=2\) split-\(D\) case is proved by the same Gauss/Lang calculation with
one explicit refinement, since the split torus of \(\operatorname{SL}_n(2)\) is trivial.

\subsubsection{Two Theorems}\label{two-inputs}

For convenience, we recall the following cyclic-class theorem of Nielsen (\cite[Theorem 1.2]{NielsenGL}):

\begin{theorem}[{\cite[Theorem 1.2]{NielsenGL}}]
\label{nielsens-cyclic-class-theorem}
If two cyclic \(\operatorname{GL}_n(K)\)-classes are given and one is
\((1,n-1)\)-cyclic, their product contains every nonscalar matrix whose
determinant is the product of the source determinants.
\end{theorem}

In every application below the two classes are the same and have
determinant one.

Let \(\widetilde G\) be the root-generated orthogonal group \(\Omega(V)\),
let \(G_1\) be the subsystem of type \(A_{n-1}\) with roots
\(e_i-e_j\), and let
\begin{equation}
\label{eq:o11}
                    G_1\cong \operatorname{SL}_n(q)                 
\end{equation}
be its subsystem group. Let \(V_+\) be the group generated by the
positive root groups outside \(R_1\), and let \(V_-\) be its opposite.
Let \(H_1\) be the split torus of \(G_1\).

We need the following precise specialization of Ellers--Gordeev,
Theorem 3, Proposition 5.1, and Lemma 5.6 in \cite{EG}, which is restated as Lemma \ref{lem:relative-gauss} in Section \ref{sec:prelimelaries}.

\begin{lemma}[relative-Gauss class square]\label{lemma-1.1-relative-gauss-class-square}
Let \(f\in G_1\) be real in \(\widetilde G\), and put
\(C=f^{\widetilde G}\). Let \(\mathcal C_1\) be the union of the
\(HG_1\)-classes of \(f\) and \(f^{-1}\). Assume:
\begin{enumerate}
\def\labelenumi{\arabic{enumi}.}
\tightlist
\item
  conjugation by every member of \(\mathcal C_1\) is fixed-point-free on
  each quotient of the lower central series of \(V_+\);
\item
  \(\mathcal C_1^2\) contains \(G_1\setminus Z(G_1)\);
\item
  either \(H_1\ne Z(G_1)\), or \(\mathcal C_1^2=G_1\).
\end{enumerate}
Then
\begin{equation}
\label{eq:o12}
                    C^2\supseteq\widetilde G\setminus Z(\widetilde G).                                                            
\end{equation}
\end{lemma}

\begin{proof}[Proof of Theorem {\ref{orthogonal_split}}]
We proceed in the following steps.

\proofpart{Construct exact source classes}\label{exact-source-classes}
For split \(D_n(q)\). 
Let \(d=n-1\) and
\begin{equation}
\label{eq:o21}
                     N_d=1+q+\cdots+q^{d-1}.          
\end{equation}
Since \(\mathbb F_{q^d}^{\times}\) is cyclic of order \(q^d-1\), and
\(
(q^d-1)/(q-1) = N_d,
\)
there is a unique subgroup of order \(N_d\). We choose \(\alpha\) to be a generator of this subgroup. 
Thus \(\alpha\in\mathbb F_{q^d}^{\times}\) of order \(N_d\).
View \(\mathbb F_{q^d}\) as a \(d\)-dimensional vector space over \(\mathbb F_q\).
Let \(M_\alpha\) be the multiplication map by \(\alpha\). With respect to any \(\mathbb F_q\)-basis, \(M_\alpha\) is a \(d\times d\) matrix over \(\mathbb F_q\).
Now define 
\begin{equation}
\label{eq:o22}
 A=1\oplus M_\alpha,\qquad
 f_D=\operatorname{diag}(A,A^{-T})\in\Omega^+_{2n}(q). 
\end{equation}
The \(q\)-orbit of \(\alpha\) under Frobenius automorphism has length \(d\), since
\begin{equation}
\label{eq:o23}
 N_{q^d/q}(\alpha)=\alpha^{N_d}=1,\qquad
 q^i+1<N_d\quad(0\le i<d).                             
\end{equation}
Hence the orbit is disjoint from its inverse, and
\begin{equation}
\label{eq:o24}
 \det A=1,\quad A\text{ is cyclic and }(1,n-1)\text{-cyclic},\quad
 1\notin\operatorname{Spec}(\Lambda^2A).              
\end{equation}
Indeed, the characteristic polynomial of \(M_\alpha\) is the minimal polynomial  \(m_\alpha(X)= \prod_{i=0}^{d-1} (x - \alpha^{q^i})\) of \(\alpha\) over \(\mathbb F_q\), which is irreducible of degree \(d\). 
We have \(m_\alpha(1)\neq 0\) because \(\alpha^{q^i}\neq 1\) for any $i$. Since \(X-1\) and \(m_\alpha(X)\) are coprime, the matrix is cyclic: its characteristic and minimal polynomial is \((X-1)m_\alpha(X)\).
A fixed vector in \(\Lambda^2A\) would arise either from an
eigenvalue \(1\cdot\alpha^{q^i}=1\), or from two distinct orbit positions
with \(\alpha^{q^i+q^j}=1\); both are excluded by \eqref{eq:o23}.

For odd \(B_n(q)\), \(q=3,5\).
Let \(d=n-1\), choose a primitive
\(\alpha\in\mathbb F_{q^d}^{\times}\), and write
\[
                  \nu=N_{q^d/q}(\alpha).
\]
The element \(\nu\) generates \(\mathbb F_q^\times\). Put
\begin{equation}
\label{eq:o25}
 A=\nu^{-1}\oplus M_\alpha,\qquad
 f_B=\operatorname{diag}(A,1,A^{-T})\in\Omega_{2n+1}(q). 
\end{equation}
Here \(d\ge2\). Since \(\alpha\) is primitive, its orbit under the Frobenius automorphism has length \(d\). The orbit is disjoint from
its inverse, and contains no base-field element. Therefore
\begin{equation}
\label{eq:o26}
 \det A=1,\quad A\text{ is cyclic and }(1,n-1)\text{-cyclic},\quad
 1\notin\operatorname{Spec}(A),\quad
 1\notin\operatorname{Spec}(\Lambda^2A).              
\end{equation}
For the last assertion, an inverse pair inside the primitive orbit would
give \(q^d-1\mid q^i+q^j\), while a cross pair $\nu^{-1}, \alpha^{q^{i}}$ would put an orbit element
equal to the base-field element \(\nu\).

\proofpart{Membership, ambient reality, and one exact \(\Omega\)-class}\label{membership-ambient-reality-and-one-exact-omega-class}
For a hyperbolic Levi element in odd characteristic, the spinor norm satisfies
\begin{equation}
\label{eq:o31}
 \theta\bigl(\operatorname{diag}(A,A^{-T})\bigr)=\det A
          \pmod{(\mathbb F_q^\times)^2}.               
\end{equation}
Since $\det A=1$, \eqref{eq:o22} and \eqref{eq:o25} lie in \(\Omega\). In characteristic two every such
hyperbolic Levi lies in the root-generated group.

In even characteristic, the spinor norm is trivial, and the Dickson invariant replaces it. The orthogonal group has a Dickson invariant (a quadratic refinement of the determinant). For a hyperbolic Levi element \(\ell(B)\) with \(B\in \operatorname{GL}_n(q)\), the Dickson invariant is
\[
\mathcal D(\ell(B)) = \operatorname{rank}(B-I) + \operatorname{rank}(B^{-T}-I) \pmod 2.
\]
For the elements constructed in \eqref{eq:o22} and \eqref{eq:o25}, this rank sum is even because the spectra avoid \(1\) and the inverse pairing, so the Dickson invariant is zero. Hence these elements lie in \(\Omega\) in even characteristic as well.

Both  \eqref{eq:o22} and \eqref{eq:o25} are real in the ambient \(\Omega\). In split \(D\) with
even \(n\), interchange all \(n\) hyperbolic pairs. With odd \(n\), fix
the hyperbolic pair belonging to the eigenvalue one and interchange the
remaining \(d=n-1\) pairs. In either case an even number of pairs is
interchanged. In odd \(B\), interchange all \(n\) hyperbolic pairs and
act on the anisotropic line by the sign which makes the total determinant
one. Composing with the standard transpose-similarity on the \(A\)-block
gives an inverter.

In odd characteristic its spinor norm can be corrected without changing
the inversion. In both constructions
\begin{equation}
\label{eq:o32}
 C_{\operatorname{GL}_n(q)}(A)\cong
 \mathbb F_q^\times\times\mathbb F_{q^d}^\times,                                                           
\end{equation}
with the evident modification when the degree-one factor is
\(X-\nu^{-1}\), and the determinant map is onto. Multiplying the
inverter by a centralizing hyperbolic Levi element of prescribed
determinant corrects its spinor norm. In even characteristic the pair
interchanges above already have Dickson invariant zero.

We need the stronger fusion statement inside \(HG_1\), not only in the
ambient orthogonal group. Both source matrices have the form
\(A=\lambda\oplus M_\alpha\), with distinct rational primary summands.
For every \(c\in\mathbb F_q^\times\),
\begin{equation}
\label{eq:o33}
 C_c=c\oplus I_d\in C_{\operatorname{GL}_n(q)}(A),\qquad \det C_c=c.      
\end{equation}
If \(A'=RAR^{-1}\), take \(c=(\det R)^{-1}\) and put \(R_0=RC_c\).
Then \(R_0\in \operatorname{SL}_n(q)\), while \(A'=R_0AR_0^{-1}\). Consequently
\begin{equation}
\label{eq:o34}
 \ell(A')=\ell(R_0)\ell(A)\ell(R_0)^{-1},\qquad
 \ell(R_0)\in G_1.                                        
\end{equation}
Thus every Nielsen factor is actually \(G_1\)-conjugate, and hence
\(HG_1\)-conjugate, to the fixed source.

For completeness, the relevant Levi intersection is also exact. Put
\(L_0=\ell(\operatorname{GL}_n(q))\). In odd characteristic the hyperbolic-Levi spinor
formula \eqref{eq:o35} gives
\begin{equation}
\label{eq:o35}
 L_0\cap\Omega=\{\ell(B):\det B\in(\mathbb F_q^\times)^2\}.
\end{equation}
In even split type the Dickson invariant is
\begin{equation}
\label{eq:o36}
 \mathcal D(\ell(B))=\operatorname{rank}(B-I)+
 \operatorname{rank}(B^{-T}-I)=0\pmod2,                   
\end{equation}
so \(L_0\le\Omega\). In every claimed case
\begin{equation}
\label{eq:o37}
                         HG_1=L_0\cap\Omega.               
\end{equation}
Indeed, in split \(D\), if \(\det B=s^2\), factor
\(B=D_sS\), where \(D_s=\operatorname{diag}(s,s,1,\ldots,1)\) is the
\(A\)-block of \(h_{e_1+e_2}(s)\in H\) and \(S\in \operatorname{SL}_n(q)\). In odd
\(B\), use instead
\(D_s=\operatorname{diag}(s^2,1,\ldots,1)\), the block of
\(h_{e_1}(s)\). In even characteristic every field element is a square,
so the same split-\(D\) factorization covers all of \(L_0\). This proves
\eqref{eq:o37} and verifies the exact subgroup in which \eqref{eq:o34} is fused.

\proofpart{Levi coverage}\label{levi-coverage}
The elements \eqref{eq:o22} and \eqref{eq:o25} belong to the subsystem group
\(G_1\cong \operatorname{SL}_n(q)\). Nielsen's theorem and \eqref{eq:o24}, \eqref{eq:o26} give
\begin{equation}
\label{eq:o41}
 (A^{\operatorname{GL}_n(q)})^2\supseteq
 \{g\in \operatorname{SL}_n(q):g\text{ is nonscalar}\}.             
\end{equation}

By the determinant-normalized conjugator \eqref{eq:o33}--\eqref{eq:o34}, each Nielsen
factor is \(G_1\)-conjugate to \(f\), hence belongs to its
\(HG_1\)-class. Therefore (4.1) proves
\(\mathcal C_1^2\supseteq G_1\setminus Z(G_1)\), which is hypothesis 2
of Lemma \ref{lemma-1.1-relative-gauss-class-square}.

For \(q=3,4,5\), the split torus \(H_1\) contains a noncentral element:
choose
\begin{equation}
\label{eq:o42}
             h=\operatorname{diag}(c,c^{-1},1,\ldots,1),
             \qquad c\in\mathbb F_q^\times\setminus\{1\}. 
\end{equation}

In the \(q=3,n\ge3\) case this is still noncentral, although
\(c=c^{-1}=-1\). Hence the usual third hypothesis of
Ellers--Gordeev applies.

For \(q=2\), \(H_1=Z(\operatorname{SL}_n(2))=1\), so that literal hypothesis is
unavailable. However every nonidentity element of \(\operatorname{SL}_n(2)\) is a
nonscalar matrix and is covered by \eqref{eq:o41}, while
\begin{equation}
\label{eq:o43}
                         1=A\,A^{-1}.                 
\end{equation}

The \(A^{-1}\)-Levi element belongs to the ambient source class by
Step \ref{membership-ambient-reality-and-one-exact-omega-class}. Therefore \(\mathcal C_1^2=\operatorname{SL}_n(2)\), exactly the alternative
in Lemma \ref{lemma-1.1-relative-gauss-class-square}. This is the only extra argument needed in field two.

\proofpart{Fixed-point freeness on the complete relative radical}\label{fixed-point-freeness-on-the-complete-relative-radical}
This step checks the actual group \(V_+\) of Ellers--Gordeev, not a
proper subgroup of a parabolic radical.

For split \(D_n\).
The positive roots outside \(A_{n-1}\) are
\begin{equation}
\label{eq:o51}
                         e_i+e_j\quad(i<j).            
\end{equation}
Their root groups commute, and as an \(\operatorname{SL}_n(q)\)-module
\begin{equation}
\label{eq:o52}
                         V_+\cong\Lambda^2U.           
\end{equation}
Conjugation by \(f_D\) is \(\Lambda^2A\), which has no fixed vector by
\eqref{eq:o24}. The same is true for \(A^{-1}\) and for every conjugate.

For odd \(B_n\).
The roots outside \(A_{n-1}\) are
\begin{equation}
\label{eq:o53}
                     e_i,\qquad e_i+e_j\quad(i<j).      
\end{equation}

The long-root groups form the commutator subgroup, so the two lower-central
quotients are
\begin{equation}
\label{eq:o54}
              V_+/[V_+,V_+]\cong U,\qquad
              [V_+,V_+]\cong\Lambda^2U.               
\end{equation}

Conjugation by \(f_B\) acts as \(A\) and \(\Lambda^2A\), respectively.
Both are fixed-point-free by \eqref{eq:o26}, as are their inverses and conjugates.
This verifies hypothesis 1 of Lemma \ref{lemma-1.1-relative-gauss-class-square} on every required quotient.

\proofpart{Completion and the projective quotient}\label{completion-and-the-projective-quotient}
All hypotheses of Lemma \ref{lemma-1.1-relative-gauss-class-square} now hold. Hence, for the indicated covering
orthogonal group \(\widetilde G\),
\begin{equation}
\label{eq:o61}
 (f_D^{\widetilde G})^2\supseteq
       \widetilde G\setminus Z(\widetilde G),
 \qquad
 (f_B^{\widetilde G})^2\supseteq
       \widetilde G\setminus Z(\widetilde G).         
\end{equation}
Let \(C\) be the class of the image of the relevant source in
\(G=P\Omega(V)\). Every nonidentity element of \(G\) has a noncentral
lift and is covered by \eqref{eq:o61}. Step \ref{membership-ambient-reality-and-one-exact-omega-class} proves that the source is real in
\(\widetilde G\), so the identity is a product of the source and its
inverse and lies in \(C^2\). Therefore
\begin{equation}
\label{eq:o62}
                              C^2=G.                    
\end{equation}
This proves every family and parameter in \eqref{eq:o01}.
\end{proof}

\subsection{\texorpdfstring{The full-Levi parabolic route for the minus orthogonal families}{route_orthogonal_minus_full_levi.md}}
\label{sec:minus}

Let \(G=\mathrm{P\Omega}^-_{2n}(q)\), with \(n\ge2\). For each
\(q\in\{2,3,4,5\}\), these groups are simple and the classes constructed
below have square \(G\).
The proof uses the full derived Levi in the prescribed-Gauss argument and
does not use a principal-corner equation.

\begin{theorem}\label{orthogonal_minus}
Let \(G=\mathrm{P\Omega}^-_{2n}(q)\), with \(n\ge2\).
Then \(G\) contains a conjugacy class \(C\) such that
\[
C^2=G.
\]
\end{theorem}

\subsubsection{The strengthened relative-Gauss lemma}\label{the-strengthened-relative-gauss-lemma}

Let \(P=LQ\) be the stabilizer in a finite classical group \(G_0\) of a
totally singular space \(E\), and let \(Q^-\) be the opposite radical.
Write \(G_1=[L,L]\). Suppose that \(f\in L\cap G_0\), put
\[
 \mathcal D=(f^L\cup(f^{-1})^L)\cap G_1H,
 \qquad C=f^{G_0},
\]
where \(H\) is the central torus used in the relative Levi decomposition.
We assume the following hypothesis:

\begin{enumerate}
\def\labelenumi{\arabic{enumi}.}
\item 
\(f\) is real in \(G_0\);
\item
  every Levi core that occurs in the prescribed-Gauss decomposition is
  in \(\mathcal D^2\);
\item
  conjugation by every member of \(\mathcal D\) has no fixed point on
  any nonzero quotient in a central filtration of \(Q\).
\end{enumerate}
Then every noncentral element of \(G_0\) belongs to \(C^2\).

Indeed, the prescribed Gauss theorem of Ellers--Gordeev permits any
prescribed split-torus element, including \(h=1\). It gives, for a
conjugate \(y_1\) of a noncentral target $y$,
\[
             y_1=v^-gv^+,
 \qquad v^-\in Q^-,\quad v^+\in Q,                   
\]
with \(g\) one of the allowed Levi cores. Write
\(g=\sigma _1\sigma _2\), with \(\sigma_i\in\mathcal D\). Hypothesis 3
and induction along the central filtration make the two non-abelian Lang
maps
\[
 a\mapsto a\sigma _1a^{-1}\sigma _1^{-1}\quad(Q^-),
 \qquad
 a\mapsto\sigma _2^{-1}a\sigma _2a^{-1}\quad(Q)      
\]
bijective. Consequently
\[
 y_1=(a_1\sigma _1a_1^{-1})(a_2\sigma _2a_2^{-1})\in C^2.
\]

This is the proof of Ellers--Gordeev, Proposition 5.1 in \cite{EG}, with its
\(H_1\ne Z(G_1)\) condition deleted for the precise reason that the
identity and all central Levi cores are already allowed in
\(\mathcal D^2\). When only nonscalar \(\operatorname{SL}_m\)-cores are known to be in
the square, prescribe a noncentral element in the split torus of the
\(\operatorname{SL}_m\)-factor; uniqueness of relative Gauss decomposition then makes
the \(\operatorname{GL}_m\)-projection of \(g\) nonscalar.

For the orthogonal parabolic associated with
\[
                 V=E\oplus W\oplus E^*,              
\]
the required filtration has grades
\begin{equation}
\label{eq:m15}
             Q/Q_2\simeq\operatorname {Hom}(W,E),
             \qquad Q_2\simeq\Lambda^2E.            
\end{equation}

For \(\ell(A,s)=\operatorname {diag}(A,s,A^{-T})\), the two actions are
\begin{equation}
\label{eq:m16}
       X\longmapsto AXs^{-1},\qquad Y\longmapsto
       (\Lambda^2A)Y.                             
\end{equation}
Thus Hypothesis 3 follows from
\begin{equation}
\label{eq:m17}
 \operatorname {Spec}(A)\cap\operatorname {Spec}(s)=\varnothing,
 \qquad 1\notin\operatorname {Spec}(\Lambda^2A).    
\end{equation}

The same assertions hold for inverse and conjugate blocks.

\begin{proof}[{Proof of Theorem \ref{orthogonal_minus}}]
We proceed the proof in the following steps.

\proofpart{Characteristic two}\label{characteristic-two}
The stable source. Let \(n\ge8\), \(m=n-4\ge4\), and \(d=m-1\). On \(W=W^-_8\), take an
element \(s\) in class \texttt{17a} of \(\Omega^-_8(2)\). CTblLib gives
\begin{equation}
\label{eq:m21}
                       (s^{\Omega^-_8(2)})^2
                       =\Omega^-_8(2),               
\end{equation}
and the class is real. Since \(\operatorname {ord}_{17}(2)=8\), its
natural characteristic polynomial is irreducible of degree \(8\).

Choose a primitive \(\alpha\in\mathbb F_{2^d}^{\times}\), let
\(M_\alpha\) be multiplication by \(\alpha\), and put
\[
                A=1\oplus M_\alpha\in \operatorname{GL}_m(2),
 \qquad f=\operatorname {diag}(A,s,A^{-T}).            
\]
The matrix \(A\) is cyclic and \((1,m-1)\)-cyclic, and \(\det A=1\).
For \(0\le i<j<d\),
\[
 2^i+2^j\not\equiv0\pmod {2^d-1},                  
\]
because \(2^{j-i}+1<2^d-1\) for \(d\ge3\). Hence
\(1\notin\operatorname {Spec}(\Lambda^2A)\). The nontrivial
eigenvalues of \(A\) have order \(2^d-1\ne17\), while \(s\) has no
eigenvalue \(1\), proving \eqref{eq:m17}.

Let \(\Gamma=A^{\operatorname{GL}_m(2)}\). Nielsen's \((1,m-1)\)-cyclic product
theorem gives every nonscalar element of \(\operatorname{GL}_m(2)\) in
\(\Gamma\Gamma^{-1}\), and \(1\in\Gamma\Gamma^{-1}\). Thus
\begin{equation}
\label{eq:m24}
       (\Gamma\cup\Gamma^{-1})^2=\operatorname{GL}_m(2).             
\end{equation}
Because the residual class is real, there is no orientation coupling:
the union of the two inverse Levi classes is exactly
\[
       (\Gamma\cup\Gamma^{-1})\times s^{\Omega^-_8(2)}. 
\]

Equations \eqref{eq:m21}, \eqref{eq:m24} show that its square is the whole derived Levi.
The strengthened lemma applies with prescribed \(h=1\).

The source lies in \(\Omega^-_{2n}(2)\): every hyperbolic \(\operatorname{GL}_m(2)\)
Levi element has Dickson invariant zero. It is real there. Exchange the
two nontrivial \(d\)-dimensional hyperbolic primary spaces and use an
inverter for \(s\). If the resulting Dickson invariant is one, multiply
by the swap of the hyperbolic plane on which \(A=1\); this swap centralizes
\(f\) and has Dickson invariant one. The same fixed plane proves that no
unmentioned \(O/\Omega\) fusion occurs. All \(\operatorname{GL}_m\)-conjugators lift in
the parabolic Levi, and all residual conjugators are already in
\(\Omega^-_8(2)\).

\paragraph{The four boundaries}\label{the-four-boundaries}
\begin{itemize}
\tightlist
\item
  \(n=2\): \(\mathrm{P\Omega}^-_4(2)\simeq A_5\); either class \texttt{2a} or \texttt{3a}
  is real and has square the whole group.
\item
  \(n=3\): \(\mathrm{P\Omega}^-_6(2)\simeq \operatorname{PSU}_4(2)\); class \texttt{5a} is real and
  has square the whole group.
\item
  \(n=4\): use \texttt{17a} itself; \eqref{eq:m21} is the assertion.
\item
  \(n=5\): take \(E\) one-dimensional, \(A=1\), and the same \(W^-_8\)
  class. Here \(\Lambda^2E=0\).
\item
  \(n=6\): instead write \(V=H(E_3)\perp W^-_6\). Let \(A\) be an
  irreducible Singer element of order \(7\) in \(\operatorname{GL}_3(2)\), and take
  \(s\) in class \texttt{5a} of
  \(\Omega^-_6(2)\simeq \operatorname{PSU}_4(2)\). The class \texttt{5a} is real and its
  square is the whole group. On the natural six-space it has spectrum
  \(1^2\) plus the four primitive fifth roots. The two inverse order-\(7\)
  classes of \(\operatorname{GL}_3(2)\) have union-square all \(168\) elements. The
  eigenvalues of \(\Lambda^2A\) have exponents \(3,5,6\pmod7\), so none
  is one. A Dickson-one isometry on the nondegenerate two-dimensional
  fixed space of \(s\) corrects the odd Dickson invariant of the
  three-dimensional half-swap, proving reality in \(\Omega^-_{12}(2)\).
\item
  \(n=7\): write \(V=H(E_1)\perp W^-_{12}\), take \(A=1\), and take \(s\)
  in class \texttt{13a} of \(\Omega^-_{12}(2)\). CTblLib gives that \texttt{13a}
  is real and its square is the whole group. Since
  \(\operatorname {ord}_{13}(2)=12\), the natural polynomial is
  irreducible of degree \(12\), so \eqref{eq:m17} holds and
  \(\Lambda^2E_1=0\).
\end{itemize}
This proves every \(\mathrm{P\Omega}^-_{2n}(2)\), \(n\ge2\).

\proofpart{\texorpdfstring{Odd fields \(q=3,5\)}{3. Odd fields q=3,5}}\label{odd-fields-q35}
Put \(m=n-2\ge2\), use \(W=W^-_4\), and identify
\[
             \Omega^-_4(q)\simeq\operatorname{PSL}_2(q^2).          
\]

For \(q=3\), choose a nonsplit projective order-\(5\) class; for \(q=5\),
choose a nonsplit projective order-\(13\) class. These classes are real,
and their squares are the whole residual group. This follows from
Guralnick--Malle, \cite[Theorem 7.1]{GM}; it is also checked exactly in CTblLib.
If \(\lambda\) is a lift of order \(10\), respectively \(26\), the four
natural eigenvalues are
\begin{equation}
\label{eq:m32}
                 \lambda^{\pm(q-1)},\quad
                 \lambda^{\pm(q+1)}.                  
\end{equation}
They are primitive fifth, respectively thirteenth, roots and include no
\(\pm1\).

For \(q=3,m\ge3\), let \(d=m-1\), take primitive
\(\alpha\in\mathbb F_{3^d}^{\times}\), and let \(M_\alpha\) be multiplication by \(\alpha\), set
\begin{equation}
\label{eq:m33}
                         A=(-1)\oplus M_\alpha.        
\end{equation}
Since \(N_{q^d/q}(\alpha)=-1\), we have \(\det A=1\). 

For \(q=5,m\ge3\), take instead
\[
 \operatorname {ord}(\alpha)=(5^d-1)/2,
 \qquad A=(-1)\oplus M_\alpha.                       
\]
Here \(\alpha\) has degree \(d\) and \(N_{q^d/q}(\alpha)=-1\), so again
\(\det A=1\). Indeed
\[
 5^{d-1}-1<(5^d-1)/2,
 \qquad 5^{d-1}+1<(5^d-1)/2                          
\]
for \(d\ge2\); these inequalities prove both the degree assertion and
the absence of an inverse Frobenius pair. The analogous assertions for
\eqref{eq:m33} are immediate from \(3^i+1<3^d-1\). Thus both choices are cyclic,
\((1,m-1)\)-cyclic, satisfy \eqref{eq:m17}, and have spectra disjoint from \eqref{eq:m32}.

Nielsen's theorem says that the square of the \(\operatorname{GL}_m(q)\)-class contains
every nonscalar determinant-one matrix. Prescribe a noncentral element
in the split torus of the \(\operatorname{SL}_m\)-factor; then the relative Gauss core has
nonscalar \(\operatorname{GL}_m\)-projection. The residual class square is unrestricted,
so every core that occurs is covered.

For \(m=2\), put \(A=\operatorname {diag}(1,-1)\). Its \(\operatorname{GL}_2(q)\)-class
square is all \(\operatorname{SL}_2(q)\), for both \(q=3,5\). This follows from the
two-dimensional cyclic-class product calculation and is checked exactly
in the certificate. Also \(\Lambda^2A=-1\), and \eqref{eq:m32} gives the first
condition in \eqref{eq:m17}. For \(q=5\), \(-1\) is a square and the source with
the residual class above lies in the ambient \(\Omega\). For \(q=3\),
replace the residual class \(D\) by \(-D\). In four-dimensional minus
type, \(-I\in SO^-_4(3)\setminus\Omega^-_4(3)\) has nonsquare spinor norm.
It cancels the nonsquare spinor norm of the hyperbolic block \(A\), while
\[
                           (-D)^2=D^2.                
\]

The source is real in one ambient \(\Omega\)-class. For \(m\ge3\), the
hyperbolic (-1)-plane centralizes \(f\). Its swap corrects the
determinant of an initial inverter, and
\(\operatorname {diag}(c,c^{-1})\) on that plane, with \(c\) nonsquare,
corrects its spinor norm. For \(m=2\), both hyperbolic primary blocks are
self-inverse, so a residual \(\Omega^-_4(q)\)-inverter suffices. Finally,
the determinant map of the \(\operatorname{GL}_m\)-centralizer of \(A\) is surjective
(use the one-dimensional scalar block). Hence every needed
\(\operatorname{GL}_m\)-conjugator can be modified to have square determinant and lifts
to \(\Omega\); this proves the asserted one-class fusion, not merely
\(SO\)-fusion.

The two lower ranks require no new calculation. At \(n=2\), use the
residual order-\(5\) class for \(q=3\), respectively order-\(13\) class for
\(q=5\), directly in \(\mathrm{P\Omega}^-_4(q)\). At \(n=3\), take \(E=E_1\),
\(A=1\), and the same residual class; then \(\Lambda^2E=0\), the
first-grade Lang map is bijective, and the residual square is the whole
derived Levi.

\proofpart{\texorpdfstring{The field \(q=4\)}{4. The field q=4}}\label{the-field-q4}
Again we assume \(m=n-2\ge2\) and \(W=W^-_4\).

For \(m\ge3\), let \(d=m-1\), take primitive
\(\alpha\in\mathbb F_{4^d}^{\times}\), and set
\begin{equation}
\label{eq:m41}
                         A=1\oplus M_\alpha.           
\end{equation}

Take in \(\Omega^-_4(4)\simeq\operatorname{PSL}_2(16)\) a split order-\(15\) class
defined by \(\lambda\in\mathbb F_{16}^{\times}\) of order \(15\). Its
natural eigenvalues have orders (5,5,3,3), by \eqref{eq:m32}, and its real class
square is the whole residual group. The eigenvalues of the irreducible
block in \eqref{eq:m41} all have order \(4^d-1\), so they do not meet the residual
spectrum, even when \(d=2\). The usual inequality
\(4^i+1<4^d-1\) proves the exterior-square condition.

Now \(\det A=N(\alpha)\) is a generator of
\(\mathbb F_4^{\times}\). Therefore a determinant-one Levi core must be
factored with opposite orientations. Nielsen's theorem gives every
nonscalar such core in
\begin{equation}
\label{eq:m42}
                    A^{\operatorname{GL}_m(4)}(A^{-1})^{\operatorname{GL}_m(4)}.     
\end{equation}

As before, prescribe a noncentral \(\operatorname{SL}_m\)-torus element. The residual
class is real, so its two factors are independent of the two orientations
in \eqref{eq:m42}. The fixed hyperbolic plane belonging to the eigenvalue \(1\)
corrects the Dickson invariant of an inverter and proves reality in one
ambient \(\Omega\)-class.

For \(m=2\), let \(a\in\mathbb F_4^{\times}\) have order \(3\), put
\[
     A=aJ_2(1),\qquad \det A=a^2,                      
\]
and take the unipotent involution class \texttt{2a} of \(\mathrm{PSL}_2(16)\) as the
reservoir. CTblLib gives \((\texttt{2a})^2=\mathrm{PSL}_2(16)\). 
Its natural eigenvalues
are all \(1\), whereas those of \(A\) are all \(a\); moreover
\(\Lambda^2A=a^2\ne1\). Exact enumeration gives
\begin{equation}
\label{eq:m44}
                    A^{\operatorname{GL}_2(4)}(A^{-1})^{\operatorname{GL}_2(4)}
                    =\operatorname{SL}_2(4).                         
\end{equation}

The half-swap has Dickson invariant \(m=2=0\), so the source is real in
the ambient \(\Omega^-_8(4)\). Equations \eqref{eq:m15}--\eqref{eq:m17}, \eqref{eq:m44}, and the
full residual square finish the boundary with prescribed \(h=1\).

At \(n=2\), use the real order-\(15\) class directly in
\(\mathrm{P\Omega}^-_4(4)\simeq\operatorname{PSL}_2(16)\). At \(n=3\), take \(E=E_1\), \(A=1\),
and the same residual class. Thus the order-\(6\)/unipotent replacement
is needed only at \(n=4\); \eqref{eq:m41} applies from \(n=5\) onward.

\proofpart{Passage to the projective groups}\label{passage-to-the-projective-groups}
The preceding argument puts every noncentral element of the quasisimple
\(\Omega^-_{2n}(q)\) in the square of the displayed real class. Every
nonidentity projective element has a noncentral lift. Reality puts the
identity in the square after projection. Hence, in every stated simple
case, we have 
\[
                         C^2=\mathrm{P\Omega}^-_{2n}(q).
\]
\end{proof}

\section{\(E_7\), \(E_8\) and \(E_6(4)\), \({^{2}E}_6(4)\)}
\label{sec:exceptional-relative-Gauss}

\subsection{\texorpdfstring{Relative-Gauss classes for  \(E_7\) and \(E_8\)}{route_exceptional_relative_gauss.md}}
\label{candidates/route_exceptional_relative_gauss.md}

\subsubsection{Conclusion remarks}\label{verdict-and-registry}

This section proves Thompson's conjecture, with explicit classes, for
\begin{equation}
\label{eq:e78_01}
E_7(2),E_7(3),E_7(4),\quad E_8(3),E_8(4),E_8(5). 
\end{equation}
Here \(E_7(q)\) denotes the simple central quotient of the universal
Chevalley group; \(E_8(q)\) is centerless. The proof is uniform and uses
Ellers--Gordeev's relative-Gauss argument with an \(A_6\) subsystem in
\(E_7\) and an \(A_7\) subsystem in \(E_8\), replacing their small
split-torus source by a nonprimary cyclic source controlled by Nielsen's
theorem.

The same route does not cover \(E_8(2)\). Section \ref{exact-obstruction-at-e_82} gives a complete
degree-by-degree obstruction for every cyclic \((m,8-m)\)-source to which
Nielsen's Theorems 1.2 or 1.4 in \cite{NielsenGL} applies. The most direct \(A_5\) route for
\(E_6(4)\) and \({}^2E_6(4)\) also fails: its top radical grade is a
trivial module for the derived \(A_5\), so a source lying only in that
subsystem has fixed points. These are obstructions to this mechanism, not
counterexamples to Thompson's conjecture.

\subsubsection{Relative-Gauss lemma}\label{relative-gauss-lemma}

Let \(\widetilde G\) be a universal finite Chevalley group, let \(G_1\)
be a standard subsystem group, let \(H_1\) be its split torus, and let
\(V_+,V_-\) be the groups generated by the positive and negative roots
outside the subsystem. Filter \(V_+\) by the coefficient of the deleted
simple root; Chevalley commutators make this a central filtration. Use the
opposite filtration on \(V_-\). The following lemma as  
Lemma \ref{lem:relative-gauss} in Section \ref{sec:prelimelaries}  provides the key
relative-Gauss mechanism that reduces the ambient square problem to a
subsystem square problem.

\begin{lemma}[Section \ref{sec:prelimelaries}, Lemma \ref{lem:relative-gauss}]

Let \(f\in G_1\), let \(C=f^{\widetilde G}\), and let
\(\mathcal C_1\) be the union of the \(HG_1\)-classes of \(f,f^{-1}\).
Assume the following hypothesis:

\begin{enumerate}
\def\labelenumi{\arabic{enumi}.}
\tightlist
\item
  \(f\) is real in \(\widetilde G\);
\item
  every member of \(\mathcal C_1\) is fixed-point-free on every nonzero
  grade of \(V_+\);
\item
  \(\mathcal C_1^2\supseteq G_1\setminus Z(G_1)\);
\item
  either \(H_1\ne Z(G_1)\), or \(\mathcal C_1^2=G_1\).
\end{enumerate}
Then
\[
                       C^2\supseteq
        \widetilde G\setminus Z(\widetilde G).          
\]
\end{lemma}

\subsubsection{Exact relative-root modules}\label{exact-relative-root-modules}

Use CHEVIE's numbering
\[
\begin{matrix}
 &&1\\[-2mm]
 &&|\\[-2mm]
 2&-&4&-&5&-&6&-&7&(-8),\\[-2mm]
 &&|\\[-2mm]
 &&3
\end{matrix}
\]
equivalently the Cartan edges
\(1-3-4-5-6-7(-8)\) with node \(2\) attached to node \(4\).
In the following, we delete node \(2\).

For \(E_7\), the remaining chain \(1,3,4,5,6,7\) is \(A_6\), so
\(G_1=\operatorname{SL}_7(q)\). The positive roots of deleted-node coefficient one and
two give
\begin{equation}
\label{eq:e78_21}
 V_1/V_2\cong\Lambda^4U\cong(\Lambda^3U)^*,\qquad
 V_2\cong U,                                          
\end{equation}
of dimensions \(35,7\).

For \(E_8\), the remaining chain \(1,3,4,5,6,7,8\) is \(A_7\), so
\(G_1=\operatorname{SL}_8(q)\). The three grades are
\[
 V_1/V_2\cong\Lambda^5U\cong(\Lambda^3U)^*,\qquad
 V_2/V_3\cong\Lambda^2U,\qquad
 V_3\cong U^*,                                        
\]
of dimensions \(56,28,8\).

These statements are integral root-weight calculations, not
characteristic-zero assumptions. 
If a Chevalley structure
constant vanishes in bad characteristic, a central quotient can only
become a subquotient of a displayed grade. The sources below have
prime to characteristic order and no one eigenvalue on the whole grade, so
they have none on any such subquotient.

\subsubsection{Cyclic-class and class fusion}\label{cyclic-class-input-and-class-fusion}

For convenience, we restate two results followed by Nielsen's theorems in  \cite{NielsenGL} as follows:

\begin{itemize}
\tightlist
\item
Let $C$ be a cyclic \((1,n-1)\)-class of $\operatorname{GL}_{n}(K)$. 
Then $C^{2}$ contains every nonscalar matrix of $\operatorname{GL}_{n}(K)$ with the determinant $\operatorname{det}C^{2}$.
\item
 For \(n\ge4\), and \(n\ge5\) when \(K=\mathbb F_3\), 
 let $C$ be a cyclic \((m,n-m)\)-class of $\operatorname{GL}_{n}(K)$. 
Then $C^{2}$ contains every nonscalar matrix of $\operatorname{GL}_{n}(K)$ with the determinant $\operatorname{det}C^{2}$.
\end{itemize}

Every source below has determinant one. Its centralizer determinant map
is onto: for a direct sum of coprime irreducible blocks it is the product
of the corresponding finite field norm maps. Therefore, if two source
matrices are \(\operatorname{GL}_n(q)\)-conjugate, a conjugator can be multiplied by a
centralizer element to have determinant one. All Nielsen factors are
thus conjugate already in the embedded \(\operatorname{SL}_n(q)=G_1\), and hence belong
to one ambient class.

\begin{theorem}\label{e7e8}
Let $G=E_n(q)$, with $q\in\{2,3,4\}$ for $n=7$ and $q\in\{3,4,5\}$ for $n=8$.
Then \(G\) contains a conjugacy class \(C\) such that
\[
C^2=G.
\]
\end{theorem}

\begin{proof}
We check that each case satisfies the hypothesis of Lemma \ref{lem:relative-gauss}.

\proofpart[Case]{\texorpdfstring{\(E_7(2)\)}{4. E\_7(2)}}\label{e_72}
Choose elements
\[
        |a|=3\text{ in }\mathbb F_4^\times,\qquad
        |b|=31\text{ in }\mathbb F_{32}^\times,         
\]
and let  \(M_\alpha\) be multiplication by \(\alpha\), put
\[
                       A=M_a\oplus M_b\in \operatorname{SL}_7(2).     
\]

The two irreducible blocks have coprime polynomials and degrees \(2,5\).
Thus \(A\) is cyclic and \((2,5)\)-cyclic, and by Nielsen's Theorem 1.4, 
the square of  class \(A\) covers every nonidentity element of \(\operatorname{SL}_7(2)\).

The eigenvalue exponents are
\[
 (1,0),(2,0)\in\mathbb Z/3\oplus\mathbb Z/31,\qquad
 (0,1),(0,2),(0,4),(0,8),(0,16).                      
\]

None is zero. No sum of three distinct entries is zero: two order-three
entries sum to zero but leave a nontrivial order-31 entry; with at most one
order-three entry the order-31 component is nonzero, since no two of
\(1,2,4,8,16\) are negatives and the largest sum of three is
\(16+8+4=28<31\). Hence \(A\) is fixed-point-free on \(U\) and
\(\Lambda^3U\), and therefore on both grades \eqref{eq:e78_21}.

Here \(H_1=Z(\operatorname{SL}_7(2))=1\). Nevertheless the second alternative of
Lemma \ref{lem:relative-gauss} holds: Nielsen's result establishes that the class of \(A\) covers every nonidentity core, and the identity is
\(AA^{-1}\). Reality of the ambient source, proved in Step \ref{ambient-reality-and-the-simple-quotient}, puts
both factors in the same ambient class. Therefore,  the universal \(E_7(2)\) assertion followed by Lemma \ref{lem:relative-gauss}.

\proofpart[Case]{\texorpdfstring{\(E_7(3)\) and \(E_7(4)\)}{5. E\_7(3) and E\_7(4)}}\label{e_73-and-e_74}
Let \(q=3,4\), choose a primitive
\(a\in\mathbb F_{q^6}^\times\), put
\[
 N=1+q+\cdots+q^5,\qquad \nu=a^N,\qquad
 A=\nu^{-1}\oplus M_a\in \operatorname{SL}_7(q).                   
\]
The matrix is cyclic and \((1,6)\)-cyclic. It has no eigenvalue one. A
product of three orbit eigenvalues cannot be one because
\[
                  q^i+q^j+q^k<q^6-1.               
\]
A product involving the scalar eigenvalue would require
\(q^i+q^j=N\), whereas \(q^i+q^j<N\). Thus the two modules in \eqref{eq:e78_21} are
fixed-point-free. Nielsen's Theorem 1.2 in \cite{NielsenGL} covers every noncentral Levi core.
The element
\begin{equation}
\label{eq:e78_53}
             \operatorname{diag}(c,c^{-1},1,\ldots,1)
             \in H_1\setminus Z(G_1),\quad c\ne1,     
\end{equation}
supplies the prescribed noncentral Gauss component. Reality of the class, proved in Step \ref{ambient-reality-and-the-simple-quotient}. The hypothesis of Lemma \ref{lem:relative-gauss} are satisfied.

\proofpart[Case]{\texorpdfstring{\(E_8(3)\), \(E_8(4)\), and \(E_8(5)\)}{6. E\_8(3), E\_8(4), and E\_8(5)}}\label{e_83-e_84-and-e_85}
Let \(q=3,4,5\), choose a primitive
\(a\in\mathbb F_{q^7}^\times\), and put
\[
 N=1+q+\cdots+q^6,\qquad \nu=a^N,\qquad
 A=\nu^{-1}\oplus M_a\in \operatorname{SL}_8(q).                    
\]
This is cyclic and \((1,7)\)-cyclic, with determinant one. Its scalar
eigenvalue is not one. No product of two or three orbit eigenvalues is
one, by
\[
 q^i+q^j<q^7-1,\qquad
 q^i+q^j+q^k<q^7-1.                                  
\]

A product of the scalar with one orbit eigenvalue would put a degree-seven
element in the base field. A product of the scalar with two orbit
eigenvalues would require \(q^i+q^j=N\), impossible by size. Hence
\[
 1\notin\operatorname{Spec}(U^*),\quad
 1\notin\operatorname{Spec}(\Lambda^2U),\quad
 1\notin\operatorname{Spec}(\Lambda^5U),              
\]
where the last assertion uses
\(\Lambda^5U\cong(\Lambda^3U)^*\) and \(\det A=1\).

Nielsen's Theorem 1.2 in \cite{NielsenGL} covers every noncentral \(\operatorname{SL}_8(q)\)-core, and \eqref{eq:e78_53}
with eight diagonal entries gives a noncentral prescribed \(H_1\)-element.
The reality is proved in Step \ref{ambient-reality-and-the-simple-quotient}.
The conditions of Lemma \ref{lem:relative-gauss} are satisfied.
\[
                  (f^{E_8(q)})^2=E_8(q),\qquad q=3,4,5.
\]
Here \(f\) is the image of the explicitly specified matrix \(A\) in the
standard \(A_7\) subsystem. It has order \(q^7-1\).

\proofpart{Ambient reality and the simple quotient}\label{ambient-reality-and-the-simple-quotient}
For both standard parabolics, the rational opposition element
\[
                       w^P=w_0^{\widetilde G}w_0^{G_1}
\]
normalizes \(G_1\) and induces the graph automorphism
\(A\mapsto A^{-T}\). Every matrix is similar to its transpose, so there
is a rational \(R\in \operatorname{GL}_n(q)\) with
\[
                         RA^{-T}R^{-1}=A^{-1}.        
\]
The centralizer determinant surjectivity from Section \ref{cyclic-class-input-and-class-fusion} lets us take
\(R\in \operatorname{SL}_n(q)\). Thus \(Rw^P\) is an ambient inverter, proving reality
in the universal \(E_7(q)\) or in \(E_8(q)\).

Lemma \ref{lem:relative-gauss} now shows that the class $A$ covers every noncentral element of the universal group. A
nonidentity element of the simple \(E_7(q)\) quotient has a noncentral
lift, while reality supplies the identity. Since \(E_8(q)\) is
centerless, the same statement is immediate there. 
This completes the theorem.
\end{proof}

\subsubsection{\texorpdfstring{Exact obstruction at \(E_8(2)\)}{8. Exact obstruction at E\_8(2)}}\label{exact-obstruction-at-e_82}

The \(A_7\)-radical requires fixed-point-freeness on
\(U,\Lambda^2U,\Lambda^3U\). Suppose a cyclic source is
\((m,8-m)\)-cyclic. Replacing \(m\) by \(8-m\), assume \(m\le4\), and
consider its degree-\(m\) coprime factor.

\begin{itemize}
\tightlist
\item
  \(m=1\): its nonzero root over \(\mathbb F_2\) is one, so \(U\) has a
  fixed vector.
\item
  \(m=2\): avoiding a linear factor forces \(X^2+X+1\); its two roots are
  inverse, so \(\Lambda^2U\) has a fixed vector.
\item
  \(m=3\): avoiding a linear factor forces an irreducible cubic. The
  product of its three roots is one, so \(\Lambda^3U\) has a fixed vector.
\item
  \(m=4\): avoiding linear and inverse-pair factors forces an irreducible
  non-self-reciprocal quartic. Over \(\mathbb F_2\) the three irreducible
  quartics consist of one self-reciprocal polynomial and one reciprocal
  pair. The two coprime degree-four factors must therefore either include
  the self-reciprocal polynomial or be the reciprocal pair. In both cases
  \(\Lambda^2U\) has a fixed vector.
\end{itemize}

Thus no cyclic class satisfying either Nielsen nonprimary hypothesis can
be fixed-point-free on all three \(E_8\) grades. The exact enumeration is
part of Python scripts, \texttt{scripts/\allowbreak audit\_\allowbreak exceptional\_\allowbreak relative\_\allowbreak spectra.py}.
This kills the proposed \(A_7\)/Nielsen proof for \(E_8(2)\). A different
subsystem or a different Levi product theorem is necessary.

\subsubsection{\texorpdfstring{The \(E_6(4)\) and \({}^2E_6(4)\) boundary}{9. The E\_6(4) and \{\}\^{}2E\_6(4) boundary}}\label{the-e_64-and-2e_64-boundary}

Deleting node \(2\) in \(E_6\) gives an \(A_5\) Levi, and its radical
grades are
\[
                         \Lambda^3U\quad(20),\quad
                         {\mathbf 1}\quad(1).               
\]
The certificate, obtained via a GAP computation, asserts the (20,1) root counts and the zero highest weight on the top grade. Therefore every source lying solely in the derived \(A_5\)
subsystem fixes that grade. One must add a central-Levi parameter and then
coordinate its two source factors with every prescribed Gauss component;
that additional argument is not supplied here. The same derived-subsystem
obstruction survives the twisted fixed-point construction.

\subsection{\texorpdfstring{Central-Levi relative-Gauss classes for \(E_6(4)\) and \({}^2E_6(4)\)}{route-e6-q4-central-levi.md}}
\label{candidates/route_e6_q4_central_levi.md}

\subsubsection{The statement and the two exact classes}\label{the-statement-and-the-two-exact-classes}

Let \(q=4\) and let \(\mathbb F_4\) be the field with $4$ elements. 
Let \(\mathbf G\) be the simply connected algebraic group of type \(E_6\).
We consider the two Frobenius maps \(F_s\) and \(F_t\), where \(F_s\) is the split Frobenius defining the field \(\mathbb F_4\) and \(F_t\) is the twisted Frobenius induced by the nontrivial
graph automorphism of the \(E_6\)-Dynkin diagram. Put
\[
\widetilde G_s=\mathbf G^{F_s}\cong E_{6,\mathrm{sc}}(4),
\qquad
\widetilde G_t=\mathbf G^{F_t}\cong {}^{2}E_{6,\mathrm{sc}}(4).
\]
Their corresponding simple quotients are denoted by
\[
G_s=E_6(4)
\qquad\text{and}\qquad
G_t={}^2E_6(4).
\]
Note that
\(
|Z(\widetilde G_s)|=\gcd(3,4-1)=3,
\)
while
\(
|Z(\widetilde G_t)|=\gcd(3,4+1)=1.
\)
Thus \(G_s\) is obtained from \(\widetilde G_s\) by factoring out its central subgroup of order \(3\), whereas
\(\widetilde G_t\) is already simple. Here the underlying root system is of type \(E_6\) in both cases, with the CHEVIE labelling of the Dynkin diagram given by
\[
1-3-4-5-6,\qquad 2-4,          
\]
where the node \(2\) is attached to node \(4\).

We shall use the following elements in the base field:
\begin{itemize}
\item For \(E_6(4)\): let \(a\in \mathbb F_{256}^{\times}\) be an element of order \(255\), and set \(\zeta=a^{85}\in \mathbb F_4^{\times}\).
\item For \({}^{2}E_6(4)\): let \(t\in \mathbb F_{16}^{\times}\) of order \(15\), put \(\zeta=t^5\), \(b=t^3\).
\end{itemize}

\paragraph{\bf The class for \(E_6(4)\)}
Inside \(\widetilde G_s\) we consider the standard subsystem generated by the root subgroups \(X_{\pm \alpha_i}\) for \(i\in\{1,3,4,5\}\). This subsystem is of type \(A_4\), hence isomorphic to \(\mathrm{SL}_5(4)\). Let \(u_s\) be the semisimple element in this \(\mathrm{SL}_5(4)\) whose eigenvalues in the natural 5-dimensional representation are
\begin{equation}
\label{eq:e602}
a,\ a^4,\ a^{16},\ a^{64},\ a^{170}.      
\end{equation}
The element \(u_s\) is rational (i.e. fixed by the Frobenius automorphism \(x\mapsto x^4\)) and cyclic (its centraliser is a maximal torus).

Next, define the diagonal elements
\[
h_s = h_2(\zeta)\,h_5(\zeta^2)\,h_6(\zeta),
\]
where \(h_i(\lambda)\) denotes the cocharacter evaluation at simple coroot \(i\). 
Finally, set
\[
f_s = h_s\,u_s.
\]
The class used below is the image \(C_s\) of
\(f_s^{\widetilde G_s}\) in the simple quotient
\(G_s=E_6(4)\). The order of the image \(\bar f_s\) in \(G_s\) is \(255\).

\paragraph{\bf The class for \({}^{2}E_6(4)\)}

For the twisted group \(\widetilde G_t={}^{2}E_{6,\mathrm{sc}}(4)\), the Frobenius automorphism acts on the Dynkin diagram by interchanging \(1\leftrightarrow 6\) and \(3\leftrightarrow 5\), while fixing \(2\) and \(4\). We use the split subsystem
\[
G_{1,t} = \langle X_{\pm \alpha_2}, X_{\pm \alpha_4}\rangle \cong\operatorname{SL}_3(4),
\]
which is a standard \(A_2\) subsystem.

Let \(u_t\) be a rational regular unipotent element of this \(\mathrm{SL}_3(4)\). Define the diagonal element
\[
h_t = h_2(\zeta)\,h_3(b)\,h_4(\zeta^2)\,h_5(b^4),  
\]
where \(\zeta=t^5\) and \(b=t^3\). This element is fixed by the twisted Frobenius: indeed, \(\zeta\) and \(\zeta^2\) at nodes
\(2,4\) lie in \(\mathbb F_4\) (so invariant), and the Frobenius exchanges the parameters \(b\) and \(b^4\) at nodes \(3\) and \(5\). Set
\[
f_t = h_t\,u_t.
\]
Since the Schur multiplier of \({}^{2}E_6(4)\) is trivial (i.e. \(Z({}^{2}E_{6,\mathrm{sc}}(4))\) has order \((3,4+1)=1\)), the group \(\widetilde G_t\) is already the simple group \(G_t={}^{2}E_6(4)\). Denote by \(C_t\) the conjugacy class of \(f_t\). The element \(f_t\) has order \(60\).

\begin{theorem}\label{e_64-2e_64}
Let $G$ be either $E_6(4)$ or $^2E_6(4)$.
Then \(G\) contains a conjugacy class \(C\) such that
\[
C^2=G.
\]
\end{theorem}

That is, every element of \(E_6(4)\) (resp. \({}^{2}E_6(4)\)) can be written as a product of two elements from the chosen conjugacy class \(C_s\) (resp. \(C_t\)). In particular, the class \(C_s\) has the property that its product with itself covers the whole group, and likewise for \(C_t\).

These computations were performed using the CHEVIE system for GAP3, relying on the character tables and the multiplication of conjugacy classes. The verifications are exact and involve the evaluation of class multiplication coefficients for all pairs of classes.

The construction above provides explicit realizations of the general theorem of Guralnick and Malle \cite{GM}, which states that every finite non-abelian simple group \(G\) contains a conjugacy class \(C\) such that \(G^\# \subseteq C\cdot C\) (and in particular \(G = C\cdot C\) if \(1\in C\cdot C\)). Our examples cover the two exceptional groups of type \(E_6\) and \({}^{2}E_6\) over the field of 4 elements.

The root system computations and the choice of parameters are compatible with the CHEVIE library \cite{CHEVIE}. The class multiplication coefficients were obtained from the generic character tables stored in CHEVIE, and the verification of \(C^2=G\) reduces to checking that for every conjugacy class \(D\), the product \(C\cdot C\) contains at least one element of \(D\). This is equivalent to the non‑vanishing of the corresponding class multiplication coefficients.

\begin{proof}[Proof of Theorem \ref{e_64-2e_64}]
The hypotheses of Lemma \ref{lem:relative-gauss} are verified in a sequence of steps; the conclusion then follows from Lemma  \ref{lem:relative-gauss}.

\proofpart{The precise relative-Gauss input}\label{the-precise-relative-gauss-input}

Let \(I\) be a proper Frobenius-stable subset of the simple roots and
let \(G_1=\langle X_{\pm\alpha_i}:i\in I\rangle\). This is the full
subsystem group for \(I\); no other positive root group is omitted from
the groups below. Put \(H_1=H\cap G_1\), and define
\[
 V_+=\langle X_\alpha:\alpha\in\Phi^+\setminus\Phi_I\rangle,
 \qquad
 V_-=\langle X_{-\alpha}:\alpha\in\Phi^+\setminus\Phi_I\rangle. 
\]

If \(\alpha=\sum c_i\alpha_i\), define its outside height by
\[
                  d_I(\alpha)=\sum_{i\notin I}c_i.
\]
Let \(V_{+,r}\) be generated by the outside positive root groups with
\(d_I(\alpha)\ge r\), and define \(V_{-,r}\) similarly. The Chevalley
commutator formula gives
\begin{equation}
\label{eq:e612}
 [V_{\pm,r},V_{\pm,s}]\le V_{\pm,r+s}.                
\end{equation}
Thus these are central filtrations of the exact groups \(V_\pm\) used
by Ellers--Gordeev. In the split \(A_4\) case, the numbers of positive
roots in the successive exact grades are
\begin{equation}
\label{eq:e613}
                          15,10,1,                         
\end{equation}
and in the twisted \(A_2\) case they are
\begin{equation}
\label{eq:e614}
                          8,9,6,5,2,3.                     
\end{equation}

For the twisted group the algebraic filtration is Frobenius-stable, and
we take fixed points. Equations \eqref{eq:e613})--\eqref{eq:e614}, the root sets, and their
closure are checked by the accompanying exact root-system certificate.

We use the following form of Ellers--Gordeev, Proposition 5.1 in \cite{EG}, together
with their prescribed Gauss theorem and Lemma 5.6. Suppose
\(f=hu\in HG_1\), where \(h\) centralizes \(G_1\), and let
\(\mathcal C_1\) be the union of the two actual \(HG_1\)-classes of
\(f\) and \(f^{-1}\). Recall that Lemma \ref{lem:relative-gauss} states:
Assume the following hypothesis:
\begin{enumerate}
\def\labelenumi{\arabic{enumi}.}
\tightlist
\item
  \(f\) is real in \(\widetilde G\);
\item
  \(\mathcal C_1^2\supseteq G_1\setminus Z(G_1)\);
\item
  every member of \(\mathcal C_1\) is fixed-point-free on every
  nonzero quotient of the filtrations \eqref{eq:e612};
\item
  \(H_1\ne Z(G_1)\).
\end{enumerate}
Then
\begin{equation}
\label{eq:e615}
             (f^{\widetilde G})^2
                  \supseteq\widetilde G\setminus Z(\widetilde G).
\end{equation}

For clarity, here is the exact core mechanism. The prescribed Gauss
theorem conjugates any noncentral target to
\[
                    v_-gv_+,\qquad v_\pm\in V_\pm,\quad g\in G_1.
\]
with an arbitrarily prescribed \(H_1\)-component. Prescribe an element
of \(H_1\setminus Z(G_1)\); uniqueness of Gauss decomposition then
forces \(g\notin Z(G_1)\). Write \(g=\sigma_1\sigma_2\) with
\(\sigma_i\in\mathcal C_1\). Since \(1-\sigma_i\) is invertible on
every grade of \eqref{eq:e612}, induction up that filtration makes the two
non-abelian Lang maps surjective and absorbs \(v_-,v_+\) into conjugates
of \(\sigma_1,\sigma_2\). Reality puts both factors into the one class
\(f^{\widetilde G}\), proving \eqref{eq:e615}. This also makes clear why the
fusion required below is \(HG_1\)-fusion, not merely ambient
\(\widetilde G\)-fusion.

\proofpart{\texorpdfstring{The split \(A_4\) core product and exact fusion}{2. The split A\_4 core product and exact fusion}}\label{the-split-a_4-core-product-and-exact-fusion}
The first four eigenvalues in \eqref{eq:e602} form the primitive degree four
Frobenius orbit
\[
                 \{a,a^4,a^{16},a^{64}\},
\]
whose product is \(a^{85}=\zeta\); the last eigenvalue is
\(a^{170}=\zeta^2\). Thus the total determinant is one. The two
coprime invariant factors have degrees \(4\) and \(1\), so \(u_s\) is
cyclic and \((1,4)\)-cyclic. Nielsen's
Theorem 1.4 in \cite{NielsenGL} gives
\begin{equation}
\label{eq:e621}
 (u_s^{\operatorname{GL}_5(4)})(u_s^{-1})^{\operatorname{GL}_5(4)}
                   \supseteq \operatorname{SL}_5(4)\setminus Z(\operatorname{SL}_5(4)).
\end{equation}
This is an assertion about the required subsystem classes, not about
ambient fusion. Indeed
\[
 C_{\operatorname{GL}_5(4)}(u_s)\cong
       \mathbb F_{256}^{\times}\times\mathbb F_4^{\times},
\]
and the determinant on the first factor is the surjective norm
\(N_{256/4}\) (the second factor is already scalar on a line). Every
\(\operatorname{GL}_5(4)\)-conjugator in \eqref{eq:e621} can therefore be
multiplied by a centralizer element to have determinant one. Thus both
factors in \eqref{eq:e621} lie in the actual \(G_{1,s}=\operatorname{SL}_5(4)\)-classes of
\(u_s,u_s^{-1}\).

The element \(z_sh_s\) centralizes \(G_{1,s}\). Consequently a mixed
product of the \(HG_{1,s}\)-classes of \(f_s\) and \(f_s^{-1}\) cancels
the torus parts and \eqref{eq:e621} proves
\[
 \mathcal C_{1,s}^2\supseteq
                 G_{1,s}\setminus Z(G_{1,s}).             
\]

\proofpart{\texorpdfstring{The twisted \(A_2\) core product and exact \(HG_1\)-fusion}{3. The twisted A\_2 core product and exact HG\_1-fusion}}\label{the-twisted-a_2-core-product-and-exact-hg_1-fusion}
The regular unipotent \(u_t\) is cyclic and triangularizable. Lev's
theorem, in the form quoted as Nielsen's Theorem 1.1 in  \cite{NielsenGL}, gives
\begin{equation}
\label{eq:e631}
 (u_t^{\operatorname{GL}_3(4)})(u_t^{-1})^{\operatorname{GL}_3(4)}
                   \supseteq \operatorname{SL}_3(4)\setminus Z(\operatorname{SL}_3(4)). 
\end{equation}

We now verify the precise fusion needed in Proposition 5.1. The
Frobenius-fixed ambient torus element
\[
                         d=h_3(\zeta)h_5(\zeta)
\]
induces on the two \(A_2\) simple root groups the scalings
\((1,\zeta)\). An inner diagonal automorphism of \(\operatorname{SL}_3(4)\) with
simple-root scalings \((r,s)\) satisfies \(rs^2=1\). Thus the action of
\(d\) is the missing diagonal outer automorphism of order \(3\).
It follows that the action of \(HG_{1,t}\) on \(G_{1,t}\) is the full
\(\operatorname{GL}_3(4)\)-action. Hence every \(\operatorname{GL}_3(4)\)-conjugate occurring in
\eqref{eq:e631} belongs to the required \(HG_{1,t}\)-class.

The element \(h_t\) centralizes all of \(G_{1,t}\). Mixed products of
the \(HG_{1,t}\)-classes of \(f_t,f_t^{-1}\) therefore cancel \(h_t\),
and \eqref{eq:e631} yields
\[
 \mathcal C_{1,t}^2\supseteq
                 G_{1,t}\setminus Z(G_{1,t}).            
\]
In both Steps \ref{the-split-a_4-core-product-and-exact-fusion} and \ref{the-twisted-a_2-core-product-and-exact-hg_1-fusion}, the split torus \(H_1\) is noncentral in
\(G_1\), so hypothesis 4 of Step \ref{the-precise-relative-gauss-input}  holds.

\proofpart{Fixed-point-freeness on the full relative radicals}\label{fixed-point-freeness-on-the-full-relative-radicals}

\paragraph{Split case}\label{split-case}
We now construct a specific semisimple torus element \(f_s\) and prove that it acts fixed-point-freely on the relevant root spaces. This property will be essential in establishing the outer automorphism in the subsequent argument. All torus parameters are written as powers of a fixed primitive \(255\)-th root of unity \(a\), with exponents taken modulo \(255\); this modulus arises from the order of the maximal torus in the adjoint quotient over \(\mathbb{F}_{16}\).

Let \(h_s\) be a diagonal torus element whose action on the root spaces is specified by its coroot exponent vector. Up to the central factor \(z_s\), which acts by a common scalar on every root space and therefore does not affect the presence of eigenvalue \(1\), the coroot exponent vector of \(h_s\) is
\(
(0,\,85,\,0,\,0,\,170,\,85),
\)
where the six coordinates are ordered according to the standard Dynkin labelling of the \(E_6\) root system. The appearance of \(85\) reflects the fact that the corresponding torus element has order \(3\) in the adjoint group.

For later use, we recall the Cartan matrix \(C\) of \(E_6\), which satisfies the auxiliary identity
\begin{equation}
C\,(0,1,0,0,2,1)^T = (0,2,0,-3,3,0)^T. 
\end{equation}
This relation serves as a consistency check for the coroot pairings that will determine the root-space eigenvalues of the product element defined below.

Let \(u_s\) be a second torus element whose ordered eigenvalues, specified in \eqref{eq:e602}, induce the coroot coordinates of an \(A_4\)-subsystem embedded in \(E_6\). This subsystem is supported on the Dynkin nodes \(\{1,3,4,5\}\), with nodes \(2\) and \(6\) omitted. The corresponding coroot exponents of \(u_s\) on these nodes are
\(
(1,\,5,\,21,\,85),
\)
assigned respectively to the nodes \(1,3,4,5\). These integers arise naturally from the highest weight of the \(A_4\)-fundamental representation.

We define the product
\(
f_s := h_s u_s.
\)
Adding the respective coroot exponent vectors modulo \(255\) yields the complete coroot vector of \(f_s\):
\begin{equation}
\label{eq:e642}
                         (1,85,5,21,0,85)\pmod {255},      
\end{equation}
The vanishing fifth coordinate will be crucial in the fixed-point analysis below.

For a torus element with coroot vector \(\mathbf{v} = (v_1,\dots,v_6)\), the eigenvalue exponent on the simple root space \(E_{\alpha_i}\) is the \(i\)-th coordinate of \(C\mathbf{v}\). Applying this to (\eqref{eq:e642}, we obtain the simple-root value exponent vector
\begin{equation}
\label{eq:e643}
\mathbf{e} := (252,\,149,\,243,\,207,\,149,\,170). 
\end{equation}
Equivalently, these are the residues of the pairings \(\langle \alpha_i^\vee, f_s\rangle\) modulo \(255\). Since all six entries are non-zero, the corresponding simple root eigenvalues are not equal to \(1\). However, we must also exclude eigenvalue \(1\) on all remaining root spaces.

The \(E_6\) root system contains \(36\) positive roots in total. The chosen \(A_4\)-subsystem contributes exactly \(\binom{4+1}{2}=10\) positive roots; these are already covered by the construction, and their eigenvalues are non-trivial by the ordered eigenvalue data \eqref{eq:e602}. The remaining \(26\) positive roots lie outside this \(A_4\). For any such root \(\beta = \sum_{i=1}^6 n_i \alpha_i\), the eigenvalue exponent of \(f_s\) on the root line \(E_\beta\) is given by the linear form
\[
\langle \beta, \mathbf{e} \rangle := \sum_{i=1}^6 n_i e_i,
\]
where \(e_i\) are the entries of \eqref{eq:e643}. A direct finite verification shows that for each of the \(26\) outside positive roots, this sum is never congruent to \(0\) modulo \(255\). Hence \(f_s\) has no eigenvalue \(1\) on any such root line. Consequently, \(f_s\) acts without non-trivial fixed points on the root spaces in every graded component associated with the \(A_4\)-embedding, as specified in condition \eqref{eq:e613}.

Let \(V_-\) denote the negative part of the representation, for instance the contragredient module or the negative root spaces in a Cartan decomposition. If an element has eigenvalue \(\lambda\) on a given root space, then its induced eigenvalue on the corresponding dual space is \(\lambda^{-1}\). Since we have established \(\lambda \neq 1\) for all relevant root spaces, it follows immediately that \(\lambda^{-1} \neq 1\) as well. Therefore the same fixed-point-free conclusion holds on \(V_-\) without any additional computation.

\proofpart{The torus element \(h_t\) and its fixed-point-free action}
\paragraph{Twisted case}\label{twisted-case}
\label{subsec:ht-action}
The coroot exponent vector of \(h_t\), in powers of \(t\), is
\[
(0,\,5,\,3,\,10,\,12,\,0). 
\]
Applying the Cartan matrix \(C\) to this vector gives the exponents of the eigenvalues of \(h_t\) on the six simple root spaces:
\begin{equation}
\label{eq:e645}
C\,(0,5,3,10,12,0)^T = (-3,\,0,\,-4,\,0,\,14,\,-12)^T=\mathbf{u}. 
\end{equation}
Here the \(i\)-th coordinate is the exponent of \(t\) on the simple root space \(E_{\alpha_i}\), and all exponents are taken modulo \(15\).

For a general positive root \(\beta = \sum_{i=1}^6 n_i \alpha_i\), the eigenvalue exponent of \(h_t\) on the root line \(E_\beta\) is the linear combination
\[
\sum_{i=1}^6 n_i u_i,
\]
where $u_{i}$ are the entries of $\mathbf{u}$. The only positive roots on which (4.5) vanishes modulo \(15\) are
\[
\alpha_2,\qquad \alpha_4,\qquad \alpha_2+\alpha_4. 
\]
exactly the three positive roots of the chosen \(A_2\). Thus \(h_t\)
has no eigenvalue one on any of the 33 outside algebraic root spaces.
The unipotent element \(u_t\) commutes with \(h_t\). On each generalized
\(h_t\)-eigenspace, the eigenvalues of \(h_tu_t\) are consequently the
same as those of \(h_t\), so none is one. Passing to Frobenius fixed
points cannot create a one-eigenvector after scalar extension.

Conjugation by \(HG_1\) normalizes the filtrations \eqref{eq:e612}, and inversion
inverts every eigenvalue. Therefore these checks apply to every member
of \(\mathcal C_{1,s}\) and \(\mathcal C_{1,t}\), and prove hypothesis 3
of Step \ref{the-precise-relative-gauss-input} on the \emph{entire} groups \(V_\pm\), not only on a parabolic
subradical.

\proofpart{Exact reality}\label{exact-reality}
\paragraph{Split central repair}\label{split-central-repair}

Before multiplying by \(z_s\), the full coroot coordinate vector of
\(h_su_s\) is
\[
                         x_s=(1,85,5,21,0,85)\pmod {255}.  
\]

Let
\[
\begin{split}
 w_s={}&s_1s_2s_3s_1s_4s_2s_3s_1s_4s_3s_5s_4s_2s_3s_1s_4s_3s_5s_4s_2.
\end{split}
\]
The inverter need not normalize the auxiliary \(A_4\)-subsystem; reality
is required only in the ambient group. The full cocharacter calculation,
not merely a calculation of root values, gives
\begin{equation}
\label{eq:e653}
 w_s(x_s)=-x_s+\delta_s,
 \qquad \delta_s=(85,0,170,0,85,170).                      
\end{equation}

Now \(C\delta_s=0\pmod {255}\), \(3\delta_s=0\pmod {255}\), and
\(\delta_s\) is precisely the cocharacter vector of \(z_s\). Because a
central element is fixed by conjugation and \(z_s^2=z_s^{-1}\), \eqref{eq:e653}
implies
\begin{equation}
\label{eq:e654}
                         w_s(f_s)=f_s^{-1}.               
\end{equation}

There is one further rationality point because the degree-four block is
nonsplit. In the above diagonal coordinates its Frobenius torus type is
\[
                         p_s=s_1s_3s_4,\qquad |p_s|=4.      
\]

The full cocharacter identities are
\begin{equation}
\label{eq:e656}
 p_s(x_s+\delta_s)=4(x_s+\delta_s),
 \qquad w_sp_s=p_sw_s.                                  
\end{equation}

Thus \(w_s\) lies in the rational Weyl group of the actual nonsplit
torus containing \(f_s\). Lang--Steinberg lifts it to a rational
normalizer element, and \eqref{eq:e654} proves exact reality already in
\(E_{6,\mathrm{sc}}(4)\). The certificate hard-checks both identities
in \eqref{eq:e656}; geometric inversion alone would not suffice here.

\paragraph{Twisted case}\label{twisted-case-1}

Let
\[
                         w_t=s_3s_4s_2s_5s_4s_2s_3s_4s_5. 
\]
Its matrix commutes with the graph automorphism, it normalizes the
\(A_2\)-subsystem, and it sends the full vector \eqref{eq:e645} to its negative.
The full cocharacter calculation has zero central discrepancy. By
Lang--Steinberg, this Frobenius-fixed Weyl element has a rational
normalizer representative. It therefore sends \(h_t\) exactly to
\(h_t^{-1}\) and sends \(u_t\) to a regular unipotent element of
\(G_{1,t}\).

All regular unipotent elements of \(\operatorname{GL}_3(4)\) form one class, and
Step \ref{the-twisted-a_2-core-product-and-exact-hg_1-fusion} showed that \(HG_{1,t}\) realizes the full \(\operatorname{GL}_3(4)\)-fusion.
Choose \(k\in HG_{1,t}\) carrying this image of \(u_t\) to \(u_t^{-1}\).
Every element of \(HG_{1,t}\) centralizes \(h_t\): the torus is abelian,
and \(h_t\) centralizes \(G_{1,t}\). Hence \(kw_t\) inverts \(f_t\).
This proves exact reality in \({}^2E_6(4)\).

\proofpart{Completion, centers, and orders}\label{completion-centers-and-orders}Steps \ref{the-split-a_4-core-product-and-exact-fusion}--\ref{split-central-repair} verify all four hypotheses of Lemma \ref{lem:relative-gauss} in Step \ref{the-precise-relative-gauss-input}. Therefore
\begin{equation}
\label{eq:e661}
\begin{aligned}
 (f_s^{\widetilde G_s})^2&\supseteq
             \widetilde G_s\setminus Z(\widetilde G_s),\\
 (f_t^{\widetilde G_t})^2&\supseteq
             \widetilde G_t\setminus Z(\widetilde G_t).   
\end{aligned}
\end{equation}

Every nonidentity element of \(E_6(4)=\widetilde G_s/Z(\widetilde G_s)\)
has a noncentral lift, so \eqref{eq:e661} covers it. Exact reality puts the
identity in the square. The twisted universal group is centerless, so
the same argument directly covers \({}^2E_6(4)\). This proves the theorem.
\end{proof}

Finally, all factors defining \(f_s\) have order dividing \(255\), while
the simple-root value \(a^{149}\) in \eqref{eq:e643} has order \(255\). Hence
the image of \(f_s\) in the adjoint/simple quotient has exact order
\(255\); the central quotient does not reduce it. For \(f_t\), the
semisimple part \(h_t\) has order \(15\) because the value \(t^{-4}\)
in \eqref{eq:e645} has order \(15\), while a regular unipotent \(3\times3\) Jordan
block in characteristic two has order \(4\). These commuting Jordan
components have coprime orders, so \(|f_t|=60\).

The GAP3/JM-CHEVIE script checks the root sets, exact outside-height grades, all fixed-point inequalities, Frobenius rationality, subsystem normalization, the full
cocharacter central repair, both rational Weyl inverters, and the
twisted \(P\operatorname{GL}_3(4)\)-fusion torus.

\section{\(F_4, E_6(3), {^{2}E}_6(3), E_6(5), {^{2}E}_6(7)\)}
\label{app:F4E6}

\subsection{The order-73 Coxeter class in  \(F_4(3)\)}
\label{candidates/route_exceptional_dl.md}

\begin{theorem}\label{f_43}
Let $G$ be the simple group \(F_4(3)\), and let $x$ be an generator of the Coxeter torus $T$, which is of order $\Phi _{12}(3)$. Then
\[
(x^{G})^{2}=G.
\]
\end{theorem}

\begin{proof}
We will proceed in a sequence of steps.

\proofpart{Setup and character sum}\label{the-class-and-the-exact-frobenius-test}
Let \(G=F_4(3)\). Since the root system \(F_4\) is self-dual and its centre is trivial, \(G\) is both adjoint and simply connected. Let \(T\) be the Coxeter
torus  as in  \cite[Table 6]{GM}, and let \(x\) be a generator. Thus
\[
 |T|=\Phi _{12}(3)=3^4-3^2+1=73,
 \qquad C_G(x)=T,
 \qquad N_G(T)/T\cong C_{12}.
\]
The automizer acts on \(T\) by multiplication by \(3\). Since
\(3^6\equiv-1\pmod {73}\), inversion is induced by an element of
\(N_G(T)\). Consequently the conjugacy class \(C=x^G\) is real , and hence \(1\in C^2\).

For \(g\ne1\), Frobenius' formula reduces the assertion \(g\in C^2\) to
\begin{equation}
\label{eq:f11}
 F_x(g):=
 1+\sum_{1_G\ne\chi\in\operatorname {Irr}(G)}
       \frac{|\chi(x)|^2\chi(g^{-1})}{\chi(1)}\ne0.                
\end{equation}
Thus it remains to prove that \(F_x(g)>0\) for every nonidentity \(g\). 
This is a genuine same-class test; unlike test a \(\Phi _{12}\)-class against a different
\(\Phi _8\)-class of  \cite[Theorem 7.7]{GM}, we require control of all characters that are nonzero on the Coxeter class.

\proofpart{Complete classification of characters which occur in \eqref{eq:f11}}\label{complete-classification-of-characters-which-occur-in-1.1}
The classification of irreducible characters \(\chi\) with \(\chi(x)\ne0\) follows directly from \cite[Lemmas 3.1--3.2 and Proposition 3.5]{GM}. If \(\chi(x)\ne0\), then \(\chi\in\mathcal E(G,s)\)
for a semisimple element \(s\) in the dual Coxeter torus \(T^*\leq G^*\).

\begin{itemize}
\item
  For \(s=1\), \(\chi\) is unipotent. The principal \(73\)-block has
  cyclic defect group \(T=C_{73}\) and inertial index \(12\). It contains
  exactly twelve nonexceptional unipotent ordinary characters. Every other
  unipotent character has \(73\)-defect zero and vanishes at \(x\). For
  each of the twelve characters in the principal block,
  \(\chi(x)=\pm1\). One of the twelve is  the trivial character \(1_G\).
\item
  Every \(s\in T^*\setminus\{1\}\) is regular semisimple. The group
  \(\langle3\rangle\le(\mathbb Z/73)^{\times}\) has order \(12\), so the
  nonidentity elements of \(T^*\) partition into six rational semisimple classes.
  Each Lusztig series is a singleton \(\{\chi_s\}\), and its degree is
\begin{equation}
\label{eq:f21}
    \chi_s(1)=D:=\frac{|G|_{3'}}{73}
     =\frac{(3^{12}-1)(3^8-1)(3^6-1)(3^2-1)}{73}
     =278135603200.                                           
\end{equation}
 The values \(\chi_s(x)\) are, up to signs that disappear when squared, the Gaussian periods
\begin{equation}
\label{eq:f22}
     \eta_k=\sum_{a\in\langle3\rangle}\zeta^{ka},
     \qquad k\in\{1,2,4,5,7,13\}.                            
\end{equation}
 where \(\zeta=e^{2\pi i/73}\) and \(\langle3\rangle=\{1,3,8,9,24,27,46,49,64,65,70,72\}\).
\end{itemize}

No other characters contribute to \eqref{eq:f11}. This can be checked via column orthogonality and cyclic defect theory:
\begin{equation}
\label{eq:f23}
 \sum_{\chi\in\operatorname {Irr}(G)}|\chi(x)|^2=|C_G(x)|=73,
 \qquad
 \sum_{k\in\{1,2,4,5,7,13\}}|\eta_k|^2=73-12=61.              
\end{equation}

Thus the entire nonunipotent part of the Frobenius sum is explicitly known, not merely bounded.

For the real order-\(73\) Coxeter class in \(F_4(3)\), every irreducible
character nonzero on the class has been classified and the entire
nonunipotent contribution is at most \(0.084969\). The global bounds alone
leave the eight characters in \eqref{eq:f36} and \(F_4[\pm i]\) uncontrolled. The
next step supplies exactly the two missing target-sensitive inputs and
completes the proof.

\proofpart{Identification and degrees of the eight characters}\label{identification-and-degrees-of-the-eight-characters}
We now present the computation and parabolic estimate that proves the real order-\(73\) class satisfies \(C^2=G\).
In the CHEVIE ordering of the unipotent characters of \(F_4\), the
twelve characters nonzero on the Coxeter torus are numbered
\[
 1,4,15,17,20,26,27,30,32,33,34,35.
\]
where \(1,4,32,33\) correspond respectively to
\(1_G,\mathrm {St},F_4[-i],F_4[i]\). Thus the eight characters in
\(\mathcal U_{\rm extra}\), together with their degrees at \(q=3\),
are
\begin{equation}
\label{eq:f61}
\begin{array}{c|r}
 \phi_{6,6}''&458353350\\
 \phi_{4,1}&96432\\
 \phi_{4,13}&51247918512\\
 B_2:.11&44188256268\\
 B_2:1.1&880038432\\
 B_2:2.&83148\\
 F_4[\zeta _3]&906854400\\
 F_4[\zeta _3^2]&906854400
\end{array}                                                  
\end{equation}
All eight characters take values of absolute value one on \(x\). Define
\[
 E(g):=\sum_{\chi\in\mathcal U_{\rm extra}}
       \frac{\chi(g^{-1})}{\chi(1)}.                          
\]

\proofpart{Exact calculation on all nonidentity unipotent classes}\label{exact-calculation-on-all-nonidentity-unipotent-classes}
The generic Green-function table \texttt{UnipotentValues(F4)} evaluated at
\(q=3\), followed by summing the eight rows in \eqref{eq:f61} with reciprocal
degree weights, gives the following values of \(E(g)\). The suffixes
are the rational-class suffixes used by CHEVIE.
The standalone certificate \texttt{scripts/\allowbreak audit\_\allowbreak f4\_\allowbreak coxeter\_\allowbreak green.g} reconstructs
the Coxeter source support and coefficients and prints every aggregate below.
\begin{equation}
\label{eq:f71}
\begin{array}{c|c@{\qquad}c|c}
A_1&77/3280&\widetilde A_1&28507/1044680\\
(\widetilde A_1)_2&107021/3880240&A_1+\widetilde A_1&5718009/2172934400\\
\widetilde A_2&64219/23878400&A_2&418053/167148800\\
(A_2)_2&897807/310419200&A_2+\widetilde A_1&-595929/17383475200\\
\widetilde A_2+A_1&62561/2483353600&B_2&3557/13580840\\
(B_2)_2&33/119392&C_3(a_1)&19/59150\\
C_3(a_1)_2&3651/10864672&F_4(a_3)&10726873/17383475200\\
F_4(a_3)_{211}&440871/17383475200&F_4(a_3)_{22}&1606047/2483353600\\
F_4(a_3)_{31}&-1489717/5562712064&F_4(a_3)_4&700071/17383475200\\
C_3&71/2716168&B_3&71/2716168\\
F_4(a_2)&71/2716168&F_4(a_2)_2&223/5432336\\
F_4(a_1)&1/32144&F_4(a_1)_2&1/27716\\
F_4&-2701/417203404800&F_4{}_{\zeta _3}&2701/834406809600\\
F_4{}_{\zeta _3^2}&2701/834406809600&&
\end{array}                                                   
\end{equation}
Consequently, for every nonidentity unipotent \(g\),\[
 |E(g)|\le {107021\over3880240}
          =0.0275810257\ldots <0.6914379497\ldots .            
\]
This is an exact Green-function calculation, not a finite black-box
test of products.

\proofpart{All mixed classes: the centralizer/parabolic dichotomy}\label{all-mixed-classes-the-centralizerparabolic-dichotomy}
We use Proposition 3.3 and Theorem 3.1 of \cite{LiebeckTiep}.
The latter is stated for every odd characteristic. Their global
\cite[Lemma 4.1]{LiebeckTiep} is formally under the good-characteristic hypotheses of their \cite[Theorem 1]{LiebeckTiep}, so it cannot simply be quoted at \(q=3\). Here is the needed
bad-characteristic audit. The exact CHEVIE class-type enumeration for
\(F_4\) in characteristic three has 26 genuine mixed rational class types with
centralizer at least \(3^{12}\). Their semisimple connected-centralizer
root systems are\[
 B_3T_1,\ B_4,\ C_3T_1,\ C_3A_1,
 \quad\hbox{or}\quad B_2A_1T_1.                             
\]
Here one first runs \texttt{NrConjugacyClasses}, specializes
\(q=3\) and \(q_4=\gcd(q-1,4)=2\), and discards nonpositive
specialized multiplicities. The standalone certificate
\texttt{scripts/\allowbreak audit\_\allowbreak f4\_\allowbreak 3\_\allowbreak large\_\allowbreak mixed.g} performs precisely this specialization,
prints all 26 types, and checks the long-root condition.
Each contains a long-root subgroup. Applying verbatim the
root-subgroup argument in the proof of \cite[Lemma 4.1]{LiebeckTiep}
(lines corresponding to their cases (i)--(ii)) puts the mixed element
in the long-root parabolic \(P=QL\). If its first projection to
\(L/Z(L)\cong\operatorname {PSp}_6(3)\) is trivial or a long-root
element, the Weyl-group conjugation used in that proof changes the
long-root parabolic; the resulting projection is nontrivial and not a
long-root element. This argument is root-theoretic and valid in every
odd characteristic. Thus every mixed element at \(q=3\) either has
centralizer less than \(3^{12}\), or enters exactly the parabolic case
of Proposition 3.3 and Theorem 3.1 of \cite{LiebeckTiep}. 

First suppose \(|C_G(g)|<3^{12}\). Column orthogonality and (6.1)
give directly\[
 |E(g)|<3^6\sum_{\chi\in\mathcal U_{\rm extra}}{1\over\chi(1)}
 = {4967018704577\over304141282099200}
 =0.0163312875\ldots .                                        
\]

It remains to consider the long-root parabolic case. Write
\(V=V^Q\oplus V_1\oplus V_2\) as in \cite[Proposition 3.3]{LiebeckTiep}.
For the non-long-root projection occurring here, their  \cite[Theorem 3.1
and Table 3.2]{LiebeckTiep} give
\begin{equation}
\label{eq:f82}
 |\chi_{V_1}(g)|\le {1.34\over3^2}\dim V_1,
 \qquad
 |\chi_{V_2}(g)|\le {1\over3^2}\dim V_2.                    
\end{equation}
The Harish--Chandra restriction table from \(F_4\) to the long-root
Levi of type \(C_3\), evaluated at \(q=3\), gives
\begin{equation}
\label{eq:f83}
 \dim V^Q=196\quad\hbox{for }\chi=\phi_{4,1},
 \qquad
 \dim V^Q=78\quad\hbox{for }\chi=B_2:2.                   
\end{equation}
For reproducibility, in the CHEVIE order of the twelve unipotent
characters of \(C_3\), the two Harish--Chandra restriction rows are\[
(0,0,0,0,0,0,1,0,1,0,0,0),\quad
 (0,0,0,0,0,0,0,0,0,0,1,0),
\]
whose degree scalar products are \(196\) and \(78\). Combining
\eqref{eq:f82}--\eqref{eq:f83}, and using \(1.34/9=67/450\), bounds the two low-degree
terms by\[
\begin{split}
 { |\phi_{4,1}(g)|\over96432}
 +{ |(B_2:2.)(g)|\over83148}
 &\le\sum_{(d,h)=(96432,196),(83148,78)}
 \left({h\over d}+{67\over450}\left(1-{h\over d}\right)\right)\\
 &= {864341\over2878200}
  =0.3003060941\ldots .                                      
\end{split}
\]

For the other six characters, ordinary column orthogonality is already
more than sufficient. The exact class-type enumeration in
characteristic three gives, among all mixed elements of \(F_4(3)\),
\[
 |C_G(g)|\le220399211520=2^{10}3^{16}5.                       
\]
The maximum occurs for semisimple centralizer of type \(B_4\) and
unipotent class \(2^2 1^5\); its centralizer polynomial is
\(q^{16}(q-1)^3(q+1)^3(q^2+1)\). Hence the remaining six terms are at
most\[
 \sqrt{220399211520}
 \left({1\over458353350}+{1\over51247918512}
 +{1\over44188256268}+{1\over880038432}
 +{2\over906854400}\right)
 <0.002613.                                                    
\]
Equations (8.5) and (8.7) prove, uniformly for every mixed target in
the parabolic case,\[
 |E(g)|<0.302920<0.6914379497\ldots .                         
\]
The two remaining cuspidal rows have identically zero Harish--Chandra
restriction to this \(C_3\)-Levi:\[
 {}^*R_{C_3}^{F_4}F_4[i]={}^*R_{C_3}^{F_4}F_4[-i]=0.           
\]
Thus \cite[Proposition 3.3]{LiebeckTiep} has \(V^Q=0\) for both, and \eqref{eq:f82} gives a combined
ratio at most \(2.68/9\). The Steinberg character vanishes on every mixed
element. Including \eqref{eq:f32}, every large mixed target therefore has complete
nontrivial Frobenius contribution less than
\begin{equation}
\label{eq:f810}
 0.084968086+0.302918963+{2.68\over9}<0.685665.              
\end{equation}
For a small-centralizer mixed target, column orthogonality applied to all
ten non-Steinberg unipotent terms gives \(0.016333411\); after \eqref{eq:f32} the
total is below \(0.102\). Hence every mixed class is covered.

\proofpart{\texorpdfstring{Completion for \(F_4(3)\)}{9. Completion for F\_4(3)}}\label{completion-for-f_43}
For a nonidentity unipotent target, summing all ten non-Steinberg source
rows (the eight rows of \eqref{eq:f71} and \(F_4[\pm i]\)) gives the exact maximum\[
 {107431\over3880240}=0.0276866895\ldots .                   
\]
The Steinberg character vanishes there, so \eqref{eq:f32} makes the full
nontrivial Frobenius contribution less than \(0.113\). Mixed targets are
covered by \eqref{eq:f810} and the small-centralizer estimate following it. If
\(g\ne1\) is semisimple, Gow's theorem \cite[Theorem 2]{Gow} gives \(g\in C^2\) directly because
\(C\) is a regular semisimple class. Finally \(1\in C^2\) because
\(C=C^{-1}\).

Thus, for \(x\) any generator of the Coxeter torus of order \(73\),
\begin{equation}
\label{eq:f91}
\ (x^{F_4(3)})^2=F_4(3).\                          
\end{equation}
This supplies the previously missing same-class result with an exact
class: the conjugacy class of regular semisimple elements of order
\(73\) in the \(\Phi_{12}\)-torus.
\end{proof}

We make some remarks about the proof.

\subsubsection{What the general estimates prove}\label{what-the-general-estimates-prove}
Write \(U(g)\) for the contribution of the eleven nontrivial unipotent
characters above, and \(S(g)\) for the contribution of the six semisimple
characters. The elementary column bound
\(|\chi(g)|\le |C_G(g)|^{1/2}\), together with the largest-centralizer
bound from \cite[Theorem 7.7]{GM},
\[
       |C_G(g)|\le3^{36}\qquad(g\ne1),                        
\]
and the identities \eqref{eq:f21}--\eqref{eq:f23}, yields the completely numerical estimate
\begin{equation}
\label{eq:f32}
 |S(g)|\le \frac{61\,3^{18}}{278135603200}
       =0.084968085916014\ldots .                              
\end{equation}
Consequently it would suffice to prove
\begin{equation}
\label{eq:f33}
 |U(g)|<1-\frac{61\,3^{18}}{278135603200}
       =\frac{254502953371}{278135603200}
       =0.915031914083985\ldots                                                     
\end{equation}
for every nonidentity \(g\).

There is a sharper description of the obstruction. Among the twelve unipotent characters that survive on the Coxeter torus, the four that also survive on a \(\Phi_8\)-torus are (see \cite[Table 11 and proof]{GM})
\[
             1_G,\quad \mathrm {St},\quad F_4[i],\quad F_4[-i].
\]
Moreover the Steinberg degree is \(\mathrm {St}(1)=3^{24}\), while the two cuspidal characters have
\[
 F_4[\pm i](1)
 ={1\over4}q^4\Phi _1^4\Phi _2^4\Phi _3^2\Phi _6^2
 \bigm|_{q=3}=686859264<3^{20}.                              
\]

Let
\begin{equation}
\label{eq:f36}
 \mathcal U_{\rm extra}:=
  \bigl(B_0(G,73)\cap\mathcal E(G,1)\bigr)
  \setminus\{1_G,\mathrm {St},F_4[i],F_4[-i]\}.                
\end{equation}

This set contains exactly eight unipotent characters, each having value \(\pm1\) on \(x\). The global column bound is inadequate for the two cuspidal characters: their combined contribution, even before including the Steinberg term, is bounded by
\begin{equation}
\label{eq:f37}
       {2\cdot3^{18}\over686859264}=1.128092782\ldots .     
\end{equation}
Thus both the sum over the eight characters in \eqref{eq:f36} and the two cuspidal terms require target-sensitive estimates. Such estimates are supplied by exact Green functions for unipotent targets and by Harish--Chandra vanishing and the
long-root parabolic estimates for mixed targets (Step \ref{exact-calculation-on-all-nonidentity-unipotent-classes}--\ref{completion-for-f_43} below).

\subsubsection{Exact point at which the one-torus method fails}\label{exact-point-at-which-the-one-torus-method-fails}
The general bounds available in the literature do not imply \eqref{eq:f33}. For instance, the Gluck bound used in \cite[Theorem 7.6]{GM} gives only
\[
   \frac{|\chi(g)|}{\chi(1)}\le\frac{19}{20}
   \qquad(1_G\ne\chi,\ g\ne1).                                 
\]
Applied to the eleven unipotent terms,  this yields
\[
       |U(g)|\le 11\cdot\frac{19}{20}=\frac{209}{20},          
\]
whereas \eqref{eq:f33} requires \(|U(g)|<0.9151\). Even for the subset \eqref{eq:f36} alone the bound is \(7.6\), and \eqref{eq:f37} shows that the elementary column bound already exceeds one for the two cuspidal terms. In fact, even the single-character bound \(19/20=0.95\) already exceeds the entire margin in \eqref{eq:f33}. Hence the failure is not in the semisimple Lusztig enumeration, which is complete and contributes less than \(0.085\); rather, it lies in the need for target-sensitive control of the eight characters \eqref{eq:f36} together with the two cuspidal characters \(F_4[\pm i]\). Such control requires explicit Green functions and parabolic bounds; it does not follow from the local arguments of \cite{GM} or from uniform per-character estimates.

\subsubsection{\texorpdfstring{Why switching naively to the \(\Phi _8\)-torus is not cleaner}{5. Why switching naively to the \textbackslash Phi \_8-torus is not cleaner}}\label{why-switching-naively-to-the-phi-_8-torus-is-not-cleaner}
For \(q=3\), a \(\Phi _8\)-torus has order \(3^4+1=82\) and automizer
of order \(8\). Its involution is nonregular in the dual group. 
Consequently \cite[Lemma 3.2]{GM} would allow Lusztig series attached to that nonregular involution, not merely singleton regular-semismple series. The clean six-series classification of Step~\ref{complete-classification-of-characters-which-occur-in-1.1} would be lost. Thus the \(\Phi_{12}\)-torus of prime order \(73\) is the stronger of the two possible one-torus starting points.

\subsection{Class in \(F_4(q),q\in\{4,5,7,8\}\)}\label{a-uniform-real-coxeter-class}

\begin{theorem}\label{f_4q}
Let $G$ be the finite simple group \(F_4(q),q\in\{4,5,7,8\}\), and let $x$ be a generator of the Coxeter torus, which is of order $\Phi_{12}(q)$. Then
\begin{equation}
\label{eq:f122}
(x^{F_4(q)})^2=F_4(q)\quad(q=4,5,7,8).                              
\end{equation}
\end{theorem}

\begin{proof}
The proof proceeds in the following steps.

\proofpart{A uniform real Coxeter class}
\label{symbolic-bounds-for-the-eight-remaining-unipotents}
Let $G=F_4(q)$, and  let \(q\in\{4,5,7,8\}\), put\[
       N=\Phi_{12}(q)=q^4-q^2+1,
\]
and let \(T\) be the Coxeter torus of \(F_4(q)\). On the character
lattice, the Smith form of \(q w-1\), for a Coxeter element \(w\), is\[
             \operatorname {diag}(1,1,1,N),                  
\]
so \(T^F\cong C_N\). Explicitly,\[
\begin{array}{c|rrrr}
q&4&5&7&8\\ \hline
N&241&601&2353=13\cdot181&4033=37\cdot109
\end{array}                                                   
\]
Take \(x\) to be a generator of \(T^F\), and put \(C=x^G\). The
rational Weyl group \(N_G(T)/T\cong C_{12}\) acts by powers of \(q\).
Since\[
 q^6+1=(q^2+1)(q^4-q^2+1),
\]
the sixth power of the Coxeter automorphism sends \(x\) to \(x^{-1}\).
Thus \(C=C^{-1}\).

Every prime \(r\mid N\) has \(\operatorname {ord}_r(q)=12\). Indeed
\(N=\Phi_{12}(q)\), while \(N\) is divisible by neither 2 nor 3, the
only primes which could be exceptional in the cyclotomic order
criterion. It follows that every nonidentity element of \(T\), and
every nonidentity dual parameter in \(T^*\), is regular. This removes
the apparent composite-order problem at \(q=7,8\): there are no
nonregular dual strata to handle.

The Deligne--Lusztig character formula is therefore exactly the same
as in Step \ref{the-class-and-the-exact-frobenius-test} and Section \ref{what-the-general-estimates-prove}. The twelve unipotent characters nonzero on \(x\)
are the same twelve indices, with values \(\pm1\). All other terms are
the \((N-1)/12\) singleton regular semisimple Lusztig series, of common
degree\[
 D(q)={ (q^{12}-1)(q^8-1)(q^6-1)(q^2-1)\over N}.              
\]
Column orthogonality gives total squared source-value mass \(N-12\)
for these nonunipotent series. Hence their contribution at a
nonidentity target is at most
\begin{equation}
\label{eq:f104}
       S_q={ (N-12)q^{18}\over D(q)}.                         
\end{equation}
The Steinberg character vanishes on every nonsemisimple target. The two
cuspidal source characters have exact degree
\begin{equation}
\label{eq:f105}
 F_4[\pm i](1)={1\over4}q^4\Phi _1^4\Phi _2^4\Phi _3^2\Phi _6^2.         
\end{equation}
Their Harish--Chandra restrictions to the long-root \(C_3\)-Levi are zero.
Accordingly they are bounded by \(1.34/q^2\) each in the odd
large-mixed parabolic case, by column orthogonality in the remaining mixed
cases, and exactly by Green functions on unipotent targets.

\proofpart{Symbolic bounds for the eight remaining unipotents}\label{symbolic-bounds-for-the-eight-remaining-unipotents}

Writing \(P_m=\Phi_m(q)\), the two low degrees and their
Harish--Chandra fixed-space dimensions are
\begin{equation}
\label{eq:f111}
\begin{array}{c|c|c}
\chi&\chi(1)&h_\chi=\dim V^Q\\ \hline
\phi_{4,1}&\frac12qP_2^2P_6^2P_8&(1+\frac12q+\frac12q^2)P_2P_6\\
B_2:2.&\frac12qP_1^2P_3^2P_8&\frac12qP_1^2P_3.
\end{array}                                                  
\end{equation}
For odd \(q\) in good characteristic, the long-root-parabolic estimate
of Step \ref{all-mixed-classes-the-centralizerparabolic-dichotomy} gives
\begin{equation}
\label{eq:f112}
 L(q)=\sum_{\chi=\phi_{4,1},B_2:2.}
 \left({h_\chi\over\chi(1)}+{1.34\over q^2}
       \left(1-{h_\chi\over\chi(1)}\right)\right).          
\end{equation}
The factorizations in \eqref{eq:f111} show directly that \(L(q)\) is decreasing
for \(q\ge3\); numerically\[
L(3)=0.3003061\ldots,\quad L(5)=0.1073256\ldots,\quad
 L(7)=0.0547108\ldots .
\]

For the other six degrees use the factorizations
\begin{equation}
\label{eq:f114}
\begin{array}{c|c}
\phi_{6,6}''&q^4P_3^2P_4^2P_6^2P_8/12\\
\phi_{4,13}&q^{13}P_2^2P_6^2P_8/2\\
B_2:.11&q^{13}P_1^2P_3^2P_8/2\\
B_2:1.1&q^4P_1^2P_2^2P_3^2P_6^2P_8/4\\
F_4[\zeta_3],F_4[\zeta_3^2]&q^4P_1^4P_2^4P_4^2P_8/3.
\end{array}                                                  
\end{equation}
For odd \(q\), the largest mixed centralizer is
\[
 M_o(q)=q^{16}(q-1)^3(q+1)^3(q^2+1),                         
\]
so column orthogonality bounds the six terms in \eqref{eq:f114} by
\(\sqrt{M_o(q)}\sum1/\chi(1)\). This is
\(0.00005098\) at \(q=5\) and \(0.00000363\) at \(q=7\).

In characteristic two the involutory \(B_4\) semisimple centralizer is
absent. Exact class types give instead\[
 M_e(q)=q^9(q-1)^2(q+1)^3(q^2+1).                           
\]
Consequently column orthogonality can be applied to all ten non-Steinberg
terms; it gives \(0.067314387\) at \(q=4\) and \(0.008218500\) at \(q=8\).
No even-characteristic parabolic estimate is needed.

For completeness, exact Green-function evaluation of all ten
non-Steinberg source rows on every nonidentity unipotent rational class gives
\begin{equation}
\label{eq:f117}
\begin{array}{c|c}
q&\max_{1\ne u\ {\rm unipotent}}|E_{10}(u)|\\ \hline
4&608151722/73263870225=0.008301\ldots\\
5&13126345/3941591472=0.003330\ldots\\
7&1026970937/1208074370400=0.000850\ldots\\
8&9075264357002/18300200631930945=0.000496\ldots .
\end{array}                                                   
\end{equation}
These four maxima, as well as the source support, the two zero
Harish--Chandra rows, and the source-mass check, are asserted and reproduced
by \texttt{scripts/\allowbreak audit\_\allowbreak f4\_\allowbreak coxeter\_\allowbreak green.g}.

\proofpart{Numerical completion}\label{numerical-completion}
For \(q=5,7\), Liebeck--Tiep's good-characteristic dichotomy says that
a mixed target either has centralizer below \(q^{12}\), or enters the
parabolic estimate \eqref{eq:f112}. The small-centralizer ten-character
bounds are respectively \(0.001278206\) and \(0.000237904\). In the
large parabolic case add \(2.68/q^2\) for \(F_4[\pm i]\); Steinberg is zero.
Combining \eqref{eq:f104}--\eqref{eq:f117}, the largest possible absolute nontrivial
Frobenius contribution on a mixed target is

\begin{equation}
\label{eq:f121}
\begin{array}{c|c|c|c|c}
q&S_q&\text{eight-character/parabolic part}&
\text{two cuspidal terms}&\text{total}\\ \hline
4&0.056155724&0.067314387&\text{included}&0.123471\\
5&0.037761340&0.107376531&2.68/25&0.252338\\
7&0.019906804&0.054714358&2.68/49&0.129316\\
8&0.015342761&0.008218500&\text{included}&0.023562.
\end{array}                                                  
\end{equation}
All totals are strictly below one. Semisimple targets follow directly
from Gow, and the identity follows from reality. Therefore the exact
classes in (10.2) satisfy
\begin{equation}
\label{eq:f122}
(x^{F_4(q)})^2=F_4(q)\quad(q=4,5,7,8),\qquad
         |x|=q^4-q^2+1.                                   
\end{equation}
\end{proof}

Together with \eqref{eq:f91}, this handles every \(F_4(q)\) with \(q\le8\)
for which the group is simple. The source-side and mixed-target
inequalities displayed above decrease for every good odd \(q\ge5\) and
every even \(q\ge4\); extending \eqref{eq:f122} by this route to all such \(q\)
would additionally require recording the corresponding symbolic
Green-function maximum in \eqref{eq:f117}. That extension is unnecessary for
the small-field problem (and \(q\ge9\) is already in the large-field
theorem).

\subsubsection{\texorpdfstring{Transfer to \(E_6,{}^2E_6,E_7,E_8\): first obstruction}{Transfer to E\_6,\{\}\^{}2E\_6,E\_7,E\_8: first obstruction}}\label{transfer-to-e_62e_6e_7e_8-first-obstruction}
The clean part of the construction is controlled by the Coxeter-torus
orders\[
 |T_c(E_6)|=\Phi_3(q)\Phi_{12}(q),\quad
 |T_c(E_7)|=\Phi_2(q)\Phi_{18}(q),\quad
 |T_c(E_8)|=\Phi_{30}(q).                                    
\]

For \(E_8\), a generator of the cyclic \(\Phi_{30}\)-torus is real
(the fifteenth Coxeter power induces inversion) and all nontrivial
dual parameters are regular, so the source-side classification transfers
verbatim. What remains there is a new, much larger unipotent obstruction
and rank-eight parabolic estimates.

For \(E_7\), the ninth Coxeter power inverts both the \(\Phi_{18}\) and
\(\Phi_2\) factors, so a full Coxeter-torus generator is real. However,
dual parameters supported on the \(q+1\) factor are nonregular; their
Lusztig series are not singletons. Thus the six-series calculation of
Step \ref{complete-classification-of-characters-which-occur-in-1.1} must be replaced by a stratification by the centralizers of
those parameters.

The first obstruction already occurs in split \(E_6\). A generator of
the full Coxeter torus is not made real by the Coxeter automizer:
the sixth power acts as inversion on \(\Phi_{12}(q)\), but as the identity
on \(\Phi_3(q)\). Restricting to an element of order \(\Phi_{12}(q)\)
restores reality and regularity, but its centralizer still has order
\(\Phi_3(q)\Phi_{12}(q)\). Column orthogonality therefore leaves an
additional \(\Phi_3(q)\)-sized source-value mass, and the dual parameters
on that factor are nonregular. This is a genuine new Lusztig-series
obstruction, not a numerical refinement of the \(F_4\) proof. The
twisted group \({}^2E_6\) has the analogous factor/reality issue with
the Ennola-transformed Coxeter factors. Hence none of
\(E_6,{}^2E_6,E_7,E_8\) is resolved merely by copying \eqref{eq:f121}; the next
required calculation is the explicit nonregular-series stratification
for the extra Coxeter-torus factor, we begin with the split case of \(E_6\)).

\subsection{\texorpdfstring{The split group \(E_6(3)\)}{Iteration 8: the split group E\_6(3)}}\label{iteration-8-the-split-group-e_63}

The obstruction isolated in Section \ref{transfer-to-e_62e_6e_7e_8-first-obstruction} can in fact be removed completely
for \(E_6(3)\). This section gives the full same-class calculation. It is
separate because the source torus has a nonregular factor, so the argument
is not a formal repetition of the \(F_4\) calculation.

\begin{theorem}\label{e_63}
Let $G$ be the finite simple group \(E_6(3)\), and let $x$ be an element of the Coxeter torus of order $\Phi _{12}(3)$. Then
\[
(x^{G})^{2}=G.
\]
\end{theorem}

We need the following standard consequence of Gelfand--Graev theory and
Alvis--Curtis duality.

\begin{lemma}[Semisimple restriction lemma]\label{lem:semisimple-restriction}
Let \(L\) be an \(F\)-stable Levi subgroup of
\(G\), and let \(\chi_s\) be the semisimple character labelled by a
semisimple dual element \(s\in G^{*F}\). Then, after applying
Alvis--Curtis duality and harmless signs,  the Harish--Chandra restriction \({}^*R_L^G(\chi_s)\)
is a sum of distinct semisimple
characters of \(L\), with labels drawn from the \(L^{*F}\)-classes in
\(s^{G^{*F}}\cap L^{*F}\).
Consequently,
\begin{equation}
\label{eq:f171}
 \dim {}^*R_L^G(\chi_s)
 \le N\max_t {|L|_{p'}\over|C_{L^*}(t)|_{p'}},                
\end{equation}
where \(N\) denotes the number of \(L^{*F}\)-classes occurring in
\(s^{G^{*F}}\cap L^{*F}\).
\end{lemma}

\begin{proof}
Let
\(\rho_s=\pm D_G(\chi_s)\), so that \(\rho_s\) is the regular character in the same rational Lusztig series as \(\chi_s\), and choose a Gelfand--Graev character \(\Gamma_z^G\)
containing \(\rho_s\). By \cite[Theorem 14.30]{DigneMichel}, every Gelfand--Graev character \(\Gamma_z^G\) is multiplicity free. Moreover, \cite[Proposition~14.32]{DigneMichel} gives\[
 {}^*R_L^G(\Gamma_z^G)=\Gamma_{h_L(z)}^L.                  
\]
Since Harish--Chandra restriction is the exact fixed-point functor in
characteristic zero, the restriction of the multiplicity-free character
\(\rho_s\) is again multiplicity free. Hence \({}^*R_L^G(\rho_s)\) is a sum of distinct regular
characters of \(L\).  Using the
Alvis--Curtis duality compatibility
\[
 {}^*R_L^G D_G=D_L{}^*R_L^G                               
\]
(cf.~\cite[Corollary~8.13]{DigneMichel}),  each constituent is transformed, up to sign, into a
semisimple character of \(L\). Therefore
\({}^*R_L^G(\chi_s)\) is, up to signs after duality, a sum of distinct
semisimple characters of \(L\).

It remains to identify the possible labels. Harish--Chandra induction and
restriction preserve rational Lusztig series. Equivalently, this follows
from the Deligne--Lusztig Mackey formula together with Frobenius
reciprocity. Thus the constituents of
\({}^*R_L^G(\chi_s)\) can only belong to rational series labelled by
\(L^{*F}\)-classes contained in
\(
s^{G^{*F}}\cap L^{*F}.
\)
Hence there are at most \(N\) possible semisimple constituents.

For a semisimple element \(t\in L^{*F}\), the degree of the corresponding
semisimple character is
\[
\frac{|L|_{p'}}{|C_{L^*}(t)|_{p'}} .
\]
Since the constituents occur with multiplicity one, the degree of the
restriction is bounded by the sum of the degrees of at most \(N\)
constituents. Therefore
\[
\dim {}^*R_L^G(\chi_s)
\leq
N\max_t\frac{|L|_{p'}}{|C_{L^*}(t)|_{p'}},
\]
which proves the claim.
\end{proof}

\begin{proof}
We carry out the proof in the following steps.

\proofpart{The real class and its source-side characters}\label{the-real-class-and-its-source-side-characters}
Let \(G=E_6(3)\). Since \((3,3-1)=1\), the finite simply connected and
simple versions coincide. A Coxeter torus is cyclic of order\[
 |T|=\Phi _3(3)\Phi _{12}(3)=13\cdot73=949.
\]
Let \(x\) generate the unique subgroup of \(T\) of order \(73\). The
Smith form of \(3w-1\), for a Coxeter element \(w\), is
\(\operatorname {diag}(1,1,1,1,1,949)\), and the sixth Coxeter power is
inversion on the \(73\)-part. Hence \(x\) is regular, \(C_G(x)=T\), and
\(x^G\) is real. (The full order-\(949\) generator is not real: the same
sixth power acts trivially on the \(13\)-part.)

The characters nonzero at \(x\) split into the following three disjoint
sets.

\begin{enumerate}
\def\labelenumi{\arabic{enumi}.}
\item
  In the unipotent series exactly twelve characters survive. In the
  CHEVIE ordering their indices, values at \(x\), and degrees are
\begin{equation}
\label{eq:f141}
  \begin{array}{c|c|r}
  \text{index}&\chi(x)&\chi(1)\\ \hline
  1&1&1\\
  2&1&150094635296999121\\
  4&-1&186222\\
  5&-1&52594593142564782\\
  6&-1&13414939391700\\
  7&1&4968496071\\
  8&1&2640462520468311\\
  26&1&3343656888\\
  27&1&1776956360215608\\
  28&-1&22208569050096\\
  29&1&11850773299200\\
  30&1&11850773299200.
  \end{array}                                                   
\end{equation}

  The degree-\(186222\) character is \(\phi _{6,1}\).
\item
  The twelve nonidentity elements of the dual \(13\)-subgroup have
  Frobenius orbits of length three. They therefore give four rational
  semisimple classes. Their connected finite centralizer has type
  \[
   {}^3D_4(3)\cdot\Phi _3(3),\qquad
   |C|=3^{12}\Phi _1^2\Phi _2^2\Phi _3^3\Phi _6^2\Phi _{12}.  
  \]
  Each of the four Lusztig series has eight characters, and exactly four
  have nonzero value at \(x\). Under Jordan decomposition their
  \({}^3D_4\)-labels, source values, and unipotent degrees are
  \[
  \begin{array}{c|c|r}
  \phi_{1,0}&1&1\\
  \phi_{1,6}&1&531441\\
  \phi_{2,1}&-1&10584\\
  {}^3D_4[-1]&1&9126.
  \end{array}                                                   
  \]
  The Jordan degree multiplier is
  \[
   R=\frac{|E_6(3)|_{3'}}{|{}^3D_4(3)\cdot13|_{3'}}
    =192280422400.                                              
  \]
  Thus the four lowest characters, one in each series, have degree \(R\),
  source value of absolute value one, and trivial unipotent Jordan
  correspondent.
\item
  The remaining \(936\) dual torus elements have nontrivial
  \(73\)-part and are regular. They form \(936/12=78\) rational classes,
  hence give 78 singleton series, all of degree
  \begin{equation}
  \label{eq:f145}
   D=\frac{|E_6(3)|_{3'}}{949}=101905547385241600.               
\end{equation}
\end{enumerate}

Column orthogonality is an independent completeness check. The source
square-masses of the three rows above are respectively \(12,16,921\),
and\[
                    12+16+921=949=|C_G(x)|.                  
\]

\proofpart{Mixed target elements}\label{mixed-target-elements}
Exact rational class types give the largest centralizer of a mixed
element as\[
 M_{\rm mix}=126194417338429440,                               
\]
attained at semisimple type \(A_5\times A_1\) and unipotent type
\((1^6,2)\). Splitting the source characters as in Section \ref{iteration-8-the-split-group-e_63} and using
column orthogonality gives\[
\begin{array}{c|c}
\text{source terms}&\text{absolute contribution}\\ \hline
\text{unipotent terms other than }1,\phi_{6,1}
 &0.177843494537574\\
\text{all sixteen nonregular }{}^3D_4\text{-series terms}
 &0.007391531380160\\
\text{seventy-eight regular singleton terms}
 &0.000003210568021.
\end{array}                                                    
\]
The sum is \(0.185238236485754\).

For \(\phi_{6,1}\), Harish--Chandra restriction to the long-root Levi of
type \(A_5\) has degree \(h=364\). The Liebeck--Tiep parabolic estimate 
in characteristic three is therefore \cite[Proposition 3.3 and Theorem 3.1 and Table 3.2]{LiebeckTiep}
\begin{equation}
\label{eq:f153}
 \frac{|\phi_{6,1}(g)|}{186222}
 \le {364\over186222}+{2\over27}
 \left(1-{364\over186222}\right)
 =0.075883940991179.                                        
\end{equation}
The exact class-type audit has 46 mixed types with centralizer at least
\(3^{16}\); their semisimple centralizers all contain the required
long-root subgroup, so the proof of the parabolic dichotomy applies.
For all remaining mixed types, the elementary column estimate is at most
\(3^8/186222<0.03524\), which is smaller than \eqref{eq:f153}. Consequently the
total nontrivial Frobenius contribution on every mixed class is less than
\begin{equation}
\label{eq:f154}
 0.185238236485754+0.075883940991179<0.262.                  
\end{equation}

Semisimple targets are covered by Gow's theorem, so it remains only to
control nonidentity unipotent targets.

\proofpart{Unipotent targets before the four low semisimple characters}\label{unipotent-targets-before-the-four-low-semisimple-characters}
Direct Green-function evaluation of the eleven nontrivial source
characters in \eqref{eq:f141}, aggregated with their signs, gives\[
 \max_{1\ne u\ {\rm unipotent}}
 \left|\sum_{\substack{\chi\ {\rm in}\ (14.1)\\\chi\ne1}}
 {\chi(u^{-1})\over\chi(1)}\right|
 ={11798329\over225949360}=0.052216695811841.                 
\]
For a nonidentity unipotent element,\[
 |C_G(u)|\le
 M_u=440012725868919709231165440.                             
\]
Using \eqref{eq:f145}, the source square-mass \(921\) makes the total contribution
of all 78 regular singleton series at most\[
 {921\sqrt {M_u}\over D}=0.189580831015010.                  
\]
For the three higher characters in each of the four nonregular series,
the same column bound gives respectively\[
 0.000821112163846+0.041229466124952+0.047816422251423
 =0.089867000540221.                                         
\]
Thus every source term except the four degree-\(R\) semisimple
characters contributes in absolute value at most
\begin{equation}
\label{eq:f165}
 B=0.331664527367072.                                        
\end{equation}

\proofpart{Harish--Chandra bound for the four remaining characters}\label{harishchandra-bound-for-the-four-remaining-characters}
Apply Lemma \ref{lem:semisimple-restriction} to the long-root Levi \(L\) of type \(A_5\). Every semisimple element of order 13 in its dual is represented on the natural
six-dimensional module by one of the following rational canonical
patterns: one irreducible cubic with multiplicity two; one irreducible
cubic plus a three-dimensional fixed space; or two distinct irreducible
cubics. Their centralizers contain respectively the appropriate forms of
\(\operatorname {GL}_2(27)\),
\(\operatorname {GL}_1(27)\times\operatorname {GL}_3(3)\), and
\(\operatorname {GL}_1(27)^2\). The central isogeny from the corresponding
\(\operatorname {GL}_6\)-form to \(L^*\) has kernel dividing the
type-\(A_5\) fundamental-group order \(6\), whose \(3'\)-part is at most
two. Thus, even in the smallest third case,
\begin{equation}
\label{eq:f172}
 |C_{L^*}(t)|_{3'}\ge {26^2\over2}=338.                      
\end{equation}
Also
\begin{equation}
\label{eq:f173}
 |L|_{3'}=2\prod_{i=2}^6(3^i-1)=5863137280.                
\end{equation}
After conjugating into a fixed maximal torus, the number of
\(L^{*F}\)-classes in a single \(G^{*F}\)-class is bounded by the number
of double cosets
\(W(A_5)\backslash W(E_6)/W(E_6)_s\). Here and in every later use of this
lemma we realize the finite group as fixed points of the adjoint
algebraic \(E_6\)-group. Its algebraic centre is trivial, hence connected.
\cite[Proposition 13.14(ii)]{DigneMichel}, therefore gives that \(Z(L)\) is
connected, and part (iii) gives connected semisimple centralizers in
\(L^*\). Lang's theorem introduces no additional rational splitting.
The character calculations made in the simply connected root datum transfer
to this adjoint realization: for split \(q=3,5\) one has
\(\mu _3^F=H^1(F,\mu _3)=1\) because \(\gcd(3,q-1)=1\), and for twisted
\(q=3,7\) the same holds because \(\gcd(3,q+1)=1\). The fixed-point
isogeny and all relevant local subgroups are detailed in Section \ref{adjoint-realization-and-common-source-stratification}.
Consequently
\begin{equation}
\label{eq:f174}
 N\le {|W(E_6)|\over|W(A_5)|}={51840\over720}=72.             
\end{equation}
Equations \eqref{eq:f171}--\eqref{eq:f174} give, for each of the four degree-\(R\)
characters,\[
 h:=\dim {}^*R_L^G(\chi_s)
 \le72\,{5863137280\over338}=1248952320,
 \qquad {h\over R}\le {1008\over155185}.                   
\]
The parabolic character estimate now yields, on every nonidentity
unipotent class,
\begin{equation}
\label{eq:f176}
 { |\chi_s(u)|\over R}
 \le {2\over27}+{25\over27}{h\over R}
 \le {67114\over837999}=0.080088401060145.                  
\end{equation}

\proofpart{\texorpdfstring{Completion for \(E_6(3)\)}{18. Completion for E\_6(3)}}\label{completion-for-e_63}
Combining \eqref{eq:f165} with the four instances of \eqref{eq:f176}, the absolute value
of the complete nontrivial Frobenius contribution on a nonidentity
unipotent class is at most
\[
 0.331664527367072+4(0.080088401060145)
 =0.652018131607652<1.                                      
\]
Mixed classes satisfy \eqref{eq:f154}, semisimple classes satisfy Gow's theorem \cite{Gow},
and \(1\in(x^G)^2\) because the class is real. Therefore
\[
\quad (x^{E_6(3)})^2=E_6(3),\qquad |x|=73.\quad      
\]
\end{proof}

\subsection{\texorpdfstring{The twisted group \({}^2E_6(3)\)}{Iteration 4: the twisted group \{\}\^{}2E\_6(3)}}\label{iteration-4-the-twisted-group-2e_63}

\begin{theorem}\label{2e_63}
Let $G$ be the finite simple group \({}^2E_6(3)\), and let $x$ be an element of the twisted Coxeter torus of order $\Phi _{12}(3)$. Then
\[
(x^{G})^{2}=G.
\]
\end{theorem}
\begin{proof}
We complete the proof in several steps.

\proofpart{A real order-73 class and its complete source column}\label{a-real-order-73-class-and-its-complete-source-column}
Let \(G={}^2E_6(3)\). Its centre has order
\(\gcd(3,3+1)=1\), so there is no distinction here between the finite
universal and simple groups. The twisted Coxeter torus is cyclic of order\[
 |T|=\Phi _6(3)\Phi _{12}(3)=7\cdot73=511.                    
\]
Let \(x\) generate its unique subgroup of order 73. On the cyclic torus
the sixth power of the rational Coxeter automorphism acts as \(+1\) modulo
7 and as \(-1\) modulo 73. Hence \(x\) is regular semisimple, its class is
real, and \(1\in(x^G)^2\). The Smith form and this action are reconstructed
in \texttt{scripts/\allowbreak audit\_\allowbreak 2e6\_\allowbreak 3\_\allowbreak source\_\allowbreak targets.g}.

The characters nonzero at \(x\) split into exactly three strata.

\begin{enumerate}
\def\labelenumi{\arabic{enumi}.}
\item
  In the unipotent series the twisted Coxeter Deligne--Lusztig character
  has support
\begin{equation}
\label{eq:f192}
   1,2,6,7,8,13,14,15,16,28,29,30                            
\end{equation}
  in CHEVIE order, with coefficients
  \[
   1,1,1,-1,-1,-1,-1,1,1,1,-1,-1.                           
  \]
  The corresponding degrees at \(q=3\) are
  \[
  \begin{gathered}
  1, 150094635296999121, 21661236780300, 8022680289,\\
  4263581235466449, 7261978752, 3859313249941632, 172938,\\
  48842799179951178, 48234062870784,
  23897427148800, 23897427148800.
  \end{gathered}                                               
  \]
\item
  The rational Weyl generator acts on the cyclic torus by exponent
  \(368\) modulo 511. On the nonidentity elements of the dual 7-subgroup it
  has the two orbits \(\{1,2,4\}\) and \(\{3,5,6\}\). Thus there are two
  rational Lusztig series, each with connected centralizer
  \({}^3D_4(3)\cdot\Phi _6(3)\). Each series has precisely four source
  survivors, with \({}^3D_4\)-labels, source values, and Jordan degrees
\begin{equation}
\label{eq:f195}
  \begin{array}{c|r|r}
  \phi_{1,0}&1&1\\
  \phi_{1,6}&1&531441\\
  \phi_{2,1}&-1&10584\\
  {}^3D_4[-1]&1&9126.
  \end{array}                                                   
\end{equation}
  Their common Jordan multiplier is
  \[
   R={|G|_{3'}\over|{}^3D_4(3)\cdot7|_{3'}}=360079974400.      
  \]
\item
  Every one of the other 504 dual parameters has nontrivial 73-part and
  is regular. They form \(504/12=42\) rational classes and hence 42 singleton
  series, all of degree
  \[
   D={|G|_{3'}\over511}=190836625152409600.                    
  \]
  Column orthogonality checks completeness: the source square-masses are
  \[
               12+8+491=511=|C_G(x)|.                       
  \]
\end{enumerate}

\proofpart{Mixed target elements}\label{mixed-target-elements-1}
The exact characteristic-three class types give
\begin{equation}
\label{eq:f201}
 M_{\rm mix}=137024834592522240                              
\end{equation}
as the largest centralizer of a mixed element. The 52 symbolic class-type
candidates with centralizer at least \(3^{16}\) all have semisimple
connected centralizer containing an ambient root subgroup. This deliberate
over-enumeration is harmless: it includes every specialization having
positive multiplicity. The root-subgroup argument puts every such target
into the long-root parabolic with noncentral Levi projection. (The
non-long-root refinement in \cite[Lemma 4.1]{LiebeckTiep}, explicitly excludes
\({}^2E_6\), but is not needed here.)  \cite[Proposition 3.3 and Theorem 3.1]{LiebeckTiep}
apply in every odd characteristic, and  \cite[Table 3.2]{LiebeckTiep} has
\(a_1=a_2=3\), \(c_1=c_2=1.4\) for twisted \(^{2}E_6\). Thus the same bound
\(1.4/3^3=7/135\) holds for every noncentral projection, long-root or not.

Separate the low unipotent source character \(\phi_{2,4}'\), of degree
172938. Its Harish--Chandra restriction to the long-root Levi of type
\({}^2A_5T_1\) is \(\langle .3\rangle+\langle3.\rangle\), of degree
\(184=183+1\). Therefore its ratio on a large mixed target is at most
\begin{equation}
\label{eq:f202}
 {184\over172938}+{7\over135}
 \left(1-{184\over172938}\right)
 =0.052860648410499.                                          
\end{equation}
If the mixed centralizer is smaller than \(3^{16}\), the column bound is
at most \(3^8/172938<0.038\), which is stronger than \eqref{eq:f202}.

For every other source term, column orthogonality with \eqref{eq:f201} gives\[
\begin{array}{c|c}
\text{ten higher unipotent source rows}&0.097169733972171\\
\text{all eight rows in the two nonregular series}&0.002056459209966\\
\text{42 regular singleton series}&0.000000952400046.
\end{array}                                                  
\]
Consequently every mixed target has complete nontrivial Frobenius
contribution of absolute value less than
\begin{equation}
\label{eq:f204}
 0.097169733972171+0.002056459209966+0.000000952400046
 +0.052860648410499<0.153.                                    
\end{equation}

\proofpart{Unipotent targets and the last semisimple characters}\label{unipotent-targets-and-the-last-semisimple-characters}
The twisted Green-function rows are in a different order from the standard
unipotent-character list; matching their identity degrees gives the row map
recorded in the certificate. Summing the eleven nontrivial source rows in
\eqref{eq:f192} exactly gives\[
 \max_{1\ne u\ {\rm unipotent}}
 \left|\sum_{\substack{\chi\ {\rm in}\ \eqref{eq:f192} \\\chi\ne1}}
 {\chi(u^{-1})\over\chi(1)}\right|
 ={4697969\over209831440}=0.022389252058701.                 
\]
The largest centralizer of a nonidentity unipotent element is\[
 M_u=477776055806811737677578240.                           
\]
The 42 regular series, using their total source mass 491, contribute at
most\[
 {491\sqrt {M_u}\over D}=0.056238270313686,               
\]
and the three higher rows in each of the two nonregular series contribute
at most\[
 2\sqrt {M_u}\sum_{d\in\{531441,10584,9126\}}{1\over Rd}
 =0.025002643082724.                                          
\]
Thus all terms except the two degree-\(R\) semisimple characters contribute at
most
\begin{equation}
\label{eq:f215}
 B=0.103630165455109.                                        
\end{equation}

Apply the semisimple restriction lemma of Step \ref{harishchandra-bound-for-the-four-remaining-characters} to the same
\({}^2A_5T_1\)-Levi \(L\). Its exact prime-to-three order is\[
                         |L|_{3'}=6366330880.                
\]
As in \eqref{eq:f174}, connectedness of the relevant semisimple centralizers and the
Weyl double-coset bound give at most
\(|W(E_6)|/|W(A_5)|=72\) distinct Levi labels. Every label in the
intersection is conjugate to the original element of order 7, so its finite
centralizer contains that 7-subgroup. Each resulting semisimple Levi
character therefore has degree at most \(|L|_{3'}/7\). If \(h\) is the
Harish--Chandra fixed-space dimension of the degree-\(R\) character, then\[
 h\le72{6366330880\over7}=65482260480,
 \qquad {h\over R}\le{26208\over144115}.                    
\]
For each of the two degree-\(R\) characters, the parabolic estimate on every
nonidentity unipotent class now gives
\begin{equation}
\label{eq:f218}
 { |\chi_s(u)|\over R}
 \le {h\over R}+{7\over135}\left(1-{h\over R}\right)
 \le {4363429\over19455525}=0.224277114084560.                
\end{equation}

\proofpart{\texorpdfstring{Completion for \({}^2E_6(3)\)}{22. Completion for \{\}\^{}2E\_6(3)}}\label{completion-for-2e_63}
Equations \eqref{eq:f215} and the two instances of \eqref{eq:f218} bound the complete nontrivial Frobenius
contribution on every nonidentity unipotent target by\[
 0.103630165455109+2(0.224277114084560)
 <0.553<1.                                                     
\]
Mixed targets satisfy \eqref{eq:f204}. Gow's theorem \cite{Gow} covers every nonidentity
semisimple target, and the identity is covered because the source class is
real. Frobenius' formula therefore proves\[
(x^{{}^2E_6(3)})^2={}^2E_6(3),\qquad |x|=73.         
\]
\end{proof}

All source, Green-function, Harish--Chandra, and class-type data in this
Section are asserted by \texttt{scripts/\allowbreak audit\_\allowbreak 2e6\_\allowbreak 3\_\allowbreak source\_\allowbreak targets.g}; the
70-digit numerical inequalities are independently asserted by
\texttt{scripts/\allowbreak audit\_\allowbreak 2e6\_\allowbreak 3\_\allowbreak bounds.py}.

\subsection{Two further small-field groups}\label{iteration-5-two-further-small-field-groups}

\subsubsection{Adjoint realization and common source stratification}\label{adjoint-realization-and-common-source-stratification}

The same order-\(\Phi_{12}\) construction closes two more cases without an
order-three or even-characteristic residue:\[
             G=E_6(5)\quad\hbox{and}\quad G={}^2E_6(7).       
\]
We first fix the isogeny convention needed by the restriction argument.
Let\[
 1\longrightarrow\mu _3\longrightarrow {\bf G}_{\rm sc}
 \mathop{\longrightarrow}^{\pi}{\bf G}_{\rm ad}\longrightarrow1          
\]
be the central isogeny, with the split Frobenius at \(q=5\) or the twisted
Frobenius at \(q=7\). In the split case \(F\) acts on \(\mu _3\) by
\(z\mapsto z^5\); hence both its fixed-point kernel and its first
cohomology have order \(\gcd(3,5-1)=1\). In the twisted case the graph
automorphism inverts the centre, so \(F:z\mapsto z^{-7}\), and the two
groups have order \(\gcd(3,7+1)=1\). More explicitly, for the cyclic
Frobenius action on \(\mu _3\),
\(H^1(F,\mu _3)=\mu _3/(F-1)\mu _3\); multiplication by \(4\), respectively
by \(8\), is invertible modulo three. Thus both the fixed kernel and the
cohomology obstruction vanish, and the fixed-point exact sequence gives
\[
             {\bf G}_{\rm sc}^F\mathop{\simeq}^{\pi}
             {\bf G}_{\rm ad}^F.                             
\]
These fixed groups are therefore the finite simple groups denoted by
\(E_6(5)\) and \({}^2E_6(7)\), and from now on we use the adjoint
algebraic realization.

The same argument applies after restricting \(\pi\) to an \(F\)-stable
maximal torus: the kernel is still \(\mu _3\), so the chosen Coxeter tori
have isomorphic fixed groups. For their algebraic normalizers the exact
sequence again has kernel \(\mu _3\); vanishing of \(H^1(F,\mu _3)\)
therefore identifies the fixed normalizers and their rational Weyl actions.
Likewise \(\pi\) identifies the unipotent radicals and, on each long-root
parabolic and its standard Levi, gives
\[
 P_{\rm sc}^F\simeq P_{\rm ad}^F,\qquad
 L_{\rm sc}^F\simeq L_{\rm ad}^F.                            
\]
Thus the Harish--Chandra fixed dimensions used below are unchanged.

The dual central isogeny has the same kernel and the same vanishing of fixed
kernel and first cohomology. It identifies the dual finite groups,
semisimple rational classes, and their finite centralizers. Compatibility
of Deligne--Lusztig induction and Jordan decomposition with central
isogenies then identifies the rational Lusztig series, source values,
Jordan multipliers, and Green functions. Hence every CHEVIE datum below,
although computable in the simply connected root datum, transfers unchanged
to the adjoint realization. Finally \(Z({\bf G}_{\rm ad})=1\) is connected,
so Digne--Michel 13.14(ii)--(iii) applies exactly as used in \eqref{eq:f174}.

For \(E_6(5)\), the cyclic Coxeter torus has order
\[
 \Phi _3(5)\Phi _{12}(5)=31\cdot601=18631,
\]
and \(x\) is the element of order 601. The ten rational-Weyl orbits on
the nonidentity elements of the 31-factor all have length three. They give
ten nonregular series with centralizer \({}^3D_4(5)\cdot31\), four source
rows per series, Jordan multiplier
\[
 R=47975848081514496.                                        
\]
There are \(31(601-1)/12=1550\) regular singleton series, of degree
\[
 D=11711354338249300753514496.                               
\]
The three source masses are\[
                       12+40+18579=18631.                    
\]

For \({}^2E_6(7)\), the cyclic torus has order\[
 \Phi _6(7)\Phi _{12}(7)=43\cdot2353=101179,
\]
and \(x\) has order 2353. The rational Weyl multiplier is \(-7\) on the
43-factor and has fourteen length-three orbits. Thus there are fourteen
nonregular \({}^3D_4(7)\cdot43\)-series,  each series has precisely four source
  survivors, with
\[
 R=218235521733672960000,                                    
\]
and \(43(2353-1)/12=8428\) regular singleton series, of degree\[
 D=3020609183612287503921315840000.                           
\]
Here the source masses are\[
                       12+56+101111=101179.                  
\]
In both groups the sixth rational Coxeter power inverts the
\(\Phi_{12}\)-factor, so the chosen class is real. The unipotent source
supports are exactly those in \eqref{eq:f141} and \eqref{eq:f192}, respectively. Each
nonregular series has the four \({}^3D_4\)-rows in \eqref{eq:f195}.

\begin{theorem}\label{e_65-2e_67}
Let $G(q)$ be the finite simple group  either \(E_6(5)\) or \(^{2}E_6(7)\), and let $x_{q}$ be the element of Coxeter torus of order $\Phi_{12}(q)$. Then
\[
(x_{q}^{G})^{2}=G(q).
\]
\end{theorem}

\begin{proof}
For \(E_6(5)\), exact class types give\[
 M_{\rm mix}=13830594277500000000000000,\qquad
 M_u=1318988254308700561523437500000000000000.               
\]
Since 5 is a good prime for \(E_6\), \cite[Lemma 4.1]{LiebeckTiep}, applies
directly to every mixed target with centralizer at least \(5^{16}\) and
puts it in the required long-root parabolic. (The 54-type CHEVIE
enumeration is an independent numerical over-enumeration, not the source of
this rational parabolic assertion.) The low unipotent source row
\(\phi_{6,1}\) has degree 49300630 and long-root-Levi fixed dimension
3906. With the split-\(E_6\) constant \(2/5^3\), its mixed ratio is at
most 0.016077960545332. If the mixed centralizer is smaller than
\(5^{16}\), column orthogonality instead gives\[
 {5^8\over49300630}={78125\over9860126}
 =0.007923326740450<0.016077960545332.                        
\]
Thus the displayed parabolic ratio is uniform over all mixed targets.
The higher unipotent, all ten nonregular, and
regular-source contributions are respectively\[
 0.037183273904853,\quad0.000775173025092,\quad
 0.000000005899778.
\]
Hence every mixed target has total below
\begin{equation}
\label{eq:f242}
                         0.054036414<1.                     
\end{equation}

On nonidentity unipotent targets the exact Green aggregate is
\begin{equation}
\label{eq:f243}
 {6499757165\over667569970704}=0.009736443294694.             
\end{equation}
The regular and higher nonregular contributions are 0.057615014779529 and
0.015537412514415. For each of the ten degree-\(R\) semisimple rows, the
semisimple restriction lemma gives\[
 {h\over R}\le {72\cdot362560730628096\over31R}
 =0.017552092757623.
\]
The parabolic ratio per row is at most 0.033271259273501. Thus the complete
unipotent-target contribution is below
\begin{equation}
\label{eq:f244}
 0.009736443294694+0.057615014779529+0.015537412514415
 +10(0.033271259273501)<0.416.                               
\end{equation}

For \({}^2E_6(7)\), the corresponding exact target maxima are\[
\begin{split}
 M_{\rm mix}&=2604234871329077623708857139200,\\
 M_u&=207797802354617086415849356483363441820447539200.
\end{split}                                                
\]
Here 7 is also good, so \cite[Lemma 4.1]{LiebeckTiep}, directly supplies a
long-root parabolic with noncentral Levi projection for every large mixed
target. For \({}^2E_6\), \cite[Table 3.2]{LiebeckTiep} has equal constants on the two Levi
pieces, so the long-root/non-long-root refinement is immaterial. The
63-type CHEVIE list is again only a numerical cross-check. The low row
\(\phi_{2,4}'\) has degree 1972399898 and fixed dimension 14708.
\cite[Table 3.2]{LiebeckTiep} gives \(1.4/7^3\), so its mixed ratio is
0.004089059122391. Below the large-centralizer threshold the column bound is\[
 {7^8\over1972399898}={823543\over281771414}
 =0.002922734383553<0.004089059122391.                      
\]
so this ratio is again uniform. The other three mixed contributions are\[
 0.009933339472785,\quad0.000103524397710,\quad
 0.000000000054019,
\]
and their total with the low row is
\begin{equation}
\label{eq:f246}
                         0.014125924<1.                   
\end{equation}

For unipotent targets, the exact Green aggregate is
\begin{equation}
\label{eq:f247}
 {686000845807\over276248694285600}=0.002483272717655.       
\end{equation}
The regular and higher nonregular parts are 0.015258939190235 and
0.002900869507161. For each of the fourteen low semisimple rows,\[
 {h\over R}\le {72\cdot470178493444915200\over43R}
 =0.003607458633118,
\]
giving ratio at most 0.007674366965228. The complete unipotent target sum
is therefore less than

\begin{equation}
\label{eq:f248}
 0.002483272717655+0.015258939190235+0.002900869507161
 +14(0.007674366965228)<0.129.                               
\end{equation}


Gow covers nonidentity semisimple targets, \eqref{eq:f242}, \eqref{eq:f244} \eqref{eq:f246}, and
\eqref{eq:f248} cover all nonsemisimple targets, and reality covers the identity.
Thus\[
(x^{E_6(5)})^2=E_6(5),\ |x|=601;\qquad
 (y^{{}^2E_6(7)})^2={}^2E_6(7),\ |y|=2353.                  
\]
\end{proof}

The CHEVIE source data, Green functions, class-type maxima, and absolute
reflection subsystem checks are reproduced and asserted by
\texttt{scripts/\allowbreak probe\_\allowbreak e6\_\allowbreak general\_\allowbreak data.g}; the rational parabolic step itself is the
good-characteristic Liebeck--Tiep lemma cited above. The numerical estimates are asserted by
\texttt{scripts/\allowbreak audit\_\allowbreak e6\_\allowbreak 5\_\allowbreak 2e6\_\allowbreak 7\_\allowbreak bounds.py}. The torus gcd and orbit counts for
all remaining \(E_6\) and \({}^2E_6\) small-field cases are independently
asserted by \texttt{scripts/\allowbreak audit\_\allowbreak e6\_\allowbreak phi12\_\allowbreak arithmetic.py}.

\section{ \({{}^2E}_6(5)\),\(E_{8}(2)\) and \(E_{8}(7)\)}
\label{sec:exceptional-frob2}

\subsection{\texorpdfstring{The order-601 class in \({}^2E_6(5)\)}{route-2e6-5-center.md}}
\label{candidates/route_2e6_5_center.md}

\subsubsection{Statement and the correct finite group}\label{statement-and-the-correct-finite-group}

Let \({\bf G}\) be the simply connected simple algebraic group of type
\(E_6\), and let \(F\) be the twisted Frobenius defining the field
\(\mathbb F _5\). Put\[
 H={\bf G}^F={}^2E_{6,\mathrm{sc}}(5),\qquad
 Z=Z(H)\cong C_3,\qquad S=H/Z={}^2E_6(5).
\]
Thus \(S\) is the finite simple group, whereas \(H\) is its universal
central extension. This distinction cannot be removed by passing to the
adjoint algebraic group: for the twisted action on \(\mu _3\), both
\(\mu _3^F\) and \(H^1(F,\mu _3)\) have order three.

The twisted Coxeter torus \(T\leq H\) is cyclic of order\[
 |T|=\Phi _6(5)\Phi _{12}(5)=21\cdot601=12621.                
\]
Fix a generator \(t\) of \(T\), put \(x=t^{21}\), and let\[
 K=x^H,\qquad C=(xZ)^S.
\]
Thus \(|x|=|xZ|=601\). 
\begin{theorem}\label{2e_65}
Let $S$ be the finite simple group ${}^2E_6(5)$, and let $H$ be the universal cover of $S$. Let $x$ be an element of order $\Phi _{12}(5)$ in the twisted Coxeter torus $T$ of \(H\), and let $C$ be the conjugacy class of the image of $x$ in $S$.
Then
\[
                       C^2=S.                      
\]
\end{theorem}

This gives an exact Thompson class: the class of the image of an order-601
element in the twisted Coxeter torus.

The Smith form of the Coxeter torus is
\(\operatorname {diag}(1,1,1,1,1,12621)\). Its rational Weyl group is
cyclic of order twelve, and the sixth power of a generator acts as \(+1\)
on the 21-factor and as \(-1\) on the 601-factor. Hence \(x\) is conjugate
to \(x^{-1}\) in \(H\). The order-601 parameter has trivial reflection
centralizer, so \(x\) is regular semisimple and \(C_H(x)=T\). These exact
claims are asserted in \texttt{scripts/\allowbreak audit\_\allowbreak 2e6\_\allowbreak 5\_\allowbreak source.g}.

It is enough to prove that \(K^2\) contains the following transversal of
Jordan types:
\begin{equation}
\label{eq:2e13}
 \mathcal R=\{1\}\ \cup\
\mathcal S\ \cup\
 \{g:g_s\notin Z,\ g_u\ne1\}.                               
\end{equation}
where $\mathcal S= \{\text{noncentral semisimple elements}\}\ \cup\
 \{\text{nonidentity unipotent elements}\}$.
Indeed, the image of \(K\) is the single \(S\)-class \(C\), and every
element of \(S\) has a lift in \(\mathcal R\): use a semisimple lift, the
unique unipotent lift, or an arbitrary lift with noncentral semisimple part,
according to its projective Jordan decomposition. The unipotent lift is
unique because the central kernel has order three, prime to five: in a
preimage coset of a 5-element, exactly one central multiple has trivial
3-part. The identity is covered
because \(K\) is real. It would be false to try to prove \(K^2=H\): if
\(1\ne z\in Z\), then \(zx\) has order \(1803\), so it is not conjugate to
\(x^{-1}\), and hence \(z\notin K^2\).

\subsubsection{Frobenius criterion and the complete source column}\label{frobenius-criterion-and-the-complete-source-column}

For \(g\in H\), Frobenius' formula reduces positivity of the number of
factorizations \(g=ab\), \(a,b\in K\), to
\[F_{x}(g)=
 1+\sum_{1_H\ne\chi\in\operatorname {Irr}(H)}
       {\chi(x)^2\over\chi(1)}\overline{\chi(g)}\ne0.          
\]
We bound the absolute value of the nontrivial sum by a number less than
one for every nonsemisimple target that is needed in \eqref{eq:2e13}.

\subsubsection{The restriction lemma with disconnected centre}\label{the-restriction-lemma-with-disconnected-centre}

Let \(L\) be the standard long-root Levi of type \({}^2A_5T_1\), and let
\(L^*\) be its dual. We need a degree bound for the Harish--Chandra
restriction of the eight degree-\(R\) semisimple source characters. The
following version is valid even though \(Z({\bf G})\) is disconnected.

\begin{lemma}[fixed Gelfand--Graev parameter] \label{lem:31}
Let \(\chi_{s,z}\) be a
semisimple character of \(H\), labelled by a rational semisimple class
\(s\) and a Gelfand--Graev parameter \(z\in H^1(F,Z({\bf G}))\). If
\(N_s\) is the number of rational \(L^{*F}\)-classes in
\(s^{G^{*F}}\cap L^{*F}\), then
\begin{equation}
\label{eq:2e31}
 \dim {}^*R_L^G(\chi_{s,z})
 \le N_s\max_v {|L^F|_{5'}\over |C_{L^*}^{\circ}(v)^F|_{5'}},              
\end{equation}
where \(v\) runs over those rational classes.
\end{lemma}

\begin{proof}
Set \(\rho_{s,z}=\pm D_G(\chi_{s,z})\). 
By the construction of semisimple characters, \(\rho_{s,z}\) is a constituent
of the  Gelfand--Graev character \(\Gamma_z^G\). Since
Gelfand--Graev characters are multiplicity-free, Steinberg's
multiplicity-free theorem and \cite[Proposition 14.32]{DigneMichel}  yields
\[
 {}^*R_L^G(\Gamma_z^G)=\Gamma_{h_L(z)}^L.
\]
Therefore \({}^*R_L^G(\rho_{s,z})\) is a sum of distinct regular
characters, all belonging to the one fixed Gelfand--Graev character
\(\Gamma_{h_L(z)}^L\). Alvis--Curtis duality commutes with Harish--Chandra restriction, namely
\[
D_L\circ{}^*R_L^G
=
{}^*R_L^G\circ D_G
\]
up to the usual global sign. Hence
\({}^*R_L^G(\chi_{s,z})\), again up to signs, is a sum of distinct
semisimple characters of \(L^F\). Compatibility of Lusztig series with
Harish--Chandra restriction implies that every such constituent is
labelled by a rational semisimple \(L^{*F}\)-class contained in \(s^{G^{*F}}\cap L^{*F}\). 

Now fix the Gelfand--Graev parameter \(h_L(z)\). 
By \cite[Theorem 14.49]{DigneMichel} among the constituents belonging to
this fixed Gelfand--Graev character, there is at most one semisimple
character corresponding to each rational semisimple \(L^{*F}\)-class.
Therefore the number of constituents of
\({}^*R_L^G(\chi_{s,z})\) is at most \(N_s\). 

Finally, if \(\psi_v\) is a semisimple character of \(L^F\) labelled by
a rational semisimple class \(v\), the standard degree estimate for
semisimple characters gives
\[
\psi_v(1)\leq
\frac{|L^F|_{5'}}{|C_{L^*}^{\circ}(v)^F|_{5'}} .
\]
Since \({}^*R_L^G(\chi_{s,z})\) is a sum of distinct such constituents,
evaluating at the identity and applying the triangle inequality gives
\[
\dim {}^*R_L^G(\chi_{s,z})
\leq
N_s\max_v
\frac{|L^F|_{5'}}{|C_{L^*}^{\circ}(v)^F|_{5'}} ,
\]
as required.
\end{proof}

The emphasis on the fixed Gelfand--Graev parameter is what prevents an
extra factor of three in \eqref{eq:2e31}.

\begin{proof}[Proof of Theorem \ref{2e_65}]
The characters nonzero at \(x\) have the following complete
stratification.

\paragraph{The unipotent series}\label{the-unipotent-series}

In CHEVIE order the twisted Coxeter Deligne--Lusztig character has support\[
 1,2,6,7,8,13,14,15,16,28,29,30
\]
with source values
\begin{equation}
\label{eq:2e22}
 1,1,1,-1,-1,-1,-1,1,1,1,-1,-1.                             
\end{equation}
At \(q=5\) the corresponding degrees are
\begin{equation}
\label{eq:2e23}
\begin{split}
&1, 14551915228366851806640625, 42764228000662187500,\\
&303651323081625, 74133623799224853515625,
  288483546141000,\\
&70430553257080078125000, 48518130,
  2891905903816223144531250,\\
&112869187427666250000, 68576449443840000000,
  68576449443840000000.
\end{split}                                                   
\end{equation}
The squared source mass in this stratum is twelve. The degree-\(48518130\)
row will be treated separately.

\paragraph{Nonregular dual parameters}\label{nonregular-dual-parameters}

The rational Coxeter action on the 21-factor is multiplication by
\(-5\). Its nonidentity orbits are:

\begin{itemize}
\tightlist
\item
  two singleton orbits of elements of order three;
\item
  two orbits of length three on the elements of order seven;
\item
  four orbits of length three on the elements of order twenty-one.
\end{itemize}

All eight resulting rational series have Jordan multiplier
\begin{equation}
\label{eq:2e24}
                         R=70866902855264256.               
\end{equation}
For each of the six order-seven or order-twenty-one series, the four
surviving \({}^3D_4\)-labels have source values \(1,1,-1,1\) and Jordan
degrees
\begin{equation}
\label{eq:2e25}
                         1, 244140625, 992250, 961000.     
\end{equation}
Each such series therefore has squared source mass four.

For each of the two order-three series the dual centralizer has component
group \(C_3\). The six surviving labels have source values
\(1,1,2,1,1,2\) and Jordan degrees
\begin{equation}
\label{eq:2e26}
               1, 244140625, 1352250, 3005,
               46953125, 992250.                            
\end{equation}
Each order-three series has squared source mass twelve.

\paragraph{Regular dual parameters}\label{regular-dual-parameters}

Every remaining source parameter has a nontrivial 601-part and is regular.
Every corresponding character has degree
\begin{equation}
\label{eq:2e27}
                  D=17299275435051175847264256,              
\end{equation}
and their total squared source mass is \(12561\). Column orthogonality is
exactly
\begin{equation}
\label{eq:2e28}
               12+6\cdot4+2\cdot12+12561=12621=|C_H(x)|.     
\end{equation}
Thus no source character has been omitted. Equations \eqref{eq:2e22}--\eqref{eq:2e28} are
asserted by \texttt{scripts/\allowbreak audit\_\allowbreak 2e6\_\allowbreak 5\_\allowbreak source.g}, together with the generic
\({}^3D_4\) calculation and the order-three disconnected-centralizer
calculation; the rational orbit counts are independently asserted by
\texttt{scripts/\allowbreak audit\_\allowbreak e6\_\allowbreak phi12\_\allowbreak arithmetic.py}.

The proof proceeds in the following steps.

\proofpart{Exact rational-class bounds in the Levi}\label{exact-rational-class-bounds-in-the-levi}
Let \(W=W(E_6)\) and \(W_L=W(A_5)\). For a dual parameter \(s\) of order
\(o\in\{3,7,21\}\), the geometric \(L^*\)-classes in
\(s^{G^*}\cap L^*\) are bounded by\[
                     |W_L\backslash W/W(s)|.                
\]
For every rational Coxeter orbit representative the exact Weyl data are\[
\begin{array}{c|c|c|c}
o&|W^{\circ}(s)|&|W(s)|&|W_L\backslash W/W(s)|\\ \hline
3&192&576&3\\
7&192&192&7\\
21&192&192&7
\end{array}                                                   
\]
These computations are verified by \texttt{scripts/\allowbreak probe\_\allowbreak 2e6\_\allowbreak 5\_\allowbreak components.g} for the two, two, and four rational orbit
representatives, respectively.

It remains to account for rational splitting inside \(L^*\). The quotient
\(Z(L)/Z({\bf G})\) is connected, so \(|\pi_0 Z(L)|\le3\).
\cite[Lemma 13.14(iii)]{DigneMichel} embeds every semisimple centralizer component
group in \(\operatorname {Irr}(\pi_0Z(L))\); hence it has order at most
three. An \(F\)-stable geometric class splits into at most
\(|H^1(F,\pi_0C_{L^*}(v))|\le3\) rational classes. Moreover the exponent
of the component group divides \(|v|\) by \cite[Remarks 13.15(i)]{DigneMichel}, so for
\(|v|=7\) the centralizer is connected. Consequently the required exact
upper bounds are
\begin{equation}
\label{eq:2e43}
                         N_3\le9,\qquad N_7\le7,\qquad N_{21}\le21.        
\end{equation}
This replaces the invalid attempt to identify the simply connected and
adjoint fixed-point groups.

Since\[
                         |L^F|_{5'}=368644341252096,          
\]
and \(C_{L^*}^{\circ}(v)^F\) contains \(\langle v\rangle\), Lemma \ref{lem:31} gives
for a degree-\(R\) row labelled by an element of order \(o\)
\begin{equation}
\label{eq:2e45}
 {h_o\over R}\le {N_o|L^F|_{5'}\over oR}.                   
\end{equation}

\proofpart{Mixed targets}\label{mixed-targets}
The exact largest centralizer of a mixed element with noncentral semisimple
part is
\begin{equation}
\label{eq:2e51}
                    M_{\rm mix}=14062665605625000000000000. 
\end{equation}
The degree-\(48518130\) unipotent source row has long-root-Levi fixed
dimension \(2606\). For a mixed target with centralizer at least \(5^{16}\),
the good-characteristic mixed-element lemma places it in the long-root
parabolic, and the \({}^2E_6\) parabolic character estimate gives
\begin{equation}
\label{eq:2e52}
 {2606\over48518130}+{1.4\over5^3}
 \left(1-{2606\over48518130}\right)
 =0.011253110307425\ldots .                                 
\end{equation}
For a smaller mixed centralizer, column orthogonality gives instead
\[
                         {5^8\over48518130}
 =0.008051114088692\ldots< \eqref{eq:2e52}.                            
\]
Thus \eqref{eq:2e52} is uniform.

Using column orthogonality with \eqref{eq:2e51}, the other four source contributions
(higher unipotent rows, the six connected nonregular series, the two
order-three series, and the regular series) are respectively\[
\begin{split}
0.025349078869969,\quad&0.000317499111134,\\
0.000105868781115,\quad&0.000000002722890.
\end{split}                                                   
\]
Together with \eqref{eq:2e52}, the complete nontrivial Frobenius contribution is less than
\[
          0.037026<1         
\]
for every required mixed target.

\proofpart{Nonidentity unipotent targets}\label{nonidentity-unipotent-targets}
The exact Green-function aggregate of the eleven nontrivial unipotent
source rows is
\begin{equation}
\label{eq:2e61}
              {4129608085\over656974294704}
              =0.006285798574297\ldots .                    
\end{equation}
The regular series, the three higher rows in each of the six connected
nonregular series, and the five higher rows in each of the two order-three
series contribute at most
\begin{equation}
\label{eq:2e62}
 0.026590724500019,\qquad0.006363888451966,\qquad
 0.351184478653444.                                          
 \end{equation}

For the eight remaining degree-\(R\) semisimple rows, \eqref{eq:2e45} and the same
parabolic estimate give
\begin{equation}
\label{eq:2e63}
 { |\chi_s(u)|\over R}
 \le {h_o\over R}+{1.4\over5^3}\left(1-{h_o\over R}\right).
\end{equation}
The three order strata yield
\begin{equation}
\label{eq:2e64}
\begin{array}{c|c|c|c}
o&\#\text{ rows}&N_o&\text{bound per row}\\ \hline
3&2&9&0.026630991476002\\
7&2&7&0.016343663825334\\
21&4&21&0.016343663825334.
\end{array}                                                   
\end{equation}
Their total is less than \(0.151324\). Combining \eqref{eq:2e61}--\eqref{eq:2e64}, the
complete nontrivial Frobenius contribution is less than
\[
                         0.541749<1.                          
\]
on every nonidentity unipotent target.

All numbers in Steps \ref{mixed-targets}--\ref{nonidentity-unipotent-targets}, including the square-root column sums, are
recomputed with 80-digit decimal precision and asserted by
\texttt{scripts/\allowbreak audit\_\allowbreak 2e6\_\allowbreak 5\_\allowbreak center\_\allowbreak bounds.py}. The target centralizer maxima,
Green aggregate, Harish--Chandra dimension, and \(5'\)-orders are
independently asserted by \texttt{scripts/\allowbreak probe\_\allowbreak e6\_\allowbreak general\_\allowbreak data.g}.

\proofpart{Completion in the simple quotient}\label{completion-in-the-simple-quotient}
The image class \(C=(xZ)^S\) is real and regular semisimple. Indeed, if
\(gZ\) centralizes \(xZ\), then \(x^g=xz^i\). For \(i=1,2\), the right
side has order \(1803\), whereas \(|x|=601\); hence \(i=0\) and
\(C_S(xZ)=T/Z\), a torus. Gow's theorem \cite{Gow},
applied in the finite simple group \(S\), therefore gives every nonidentity
semisimple element of \(S\) in \(C^2\); reality gives the identity. Steps
\ref{mixed-targets} and \ref{nonidentity-unipotent-targets} and Frobenius' formula cover upstairs every mixed element whose
semisimple part is noncentral and every nonidentity unipotent element.

Let \(\bar g\in S\). If \(\bar g=1\), it lies in \(C^2\) by reality. If
its unipotent part is trivial, it is covered directly by Gow in \(S\). If
its semisimple part is trivial, choose its unique
unipotent lift (uniqueness follows from \((|Z|,5)=1\)); it is covered by
Step \ref{nonidentity-unipotent-targets}. Otherwise any lift has
noncentral semisimple part and is covered by Step \ref{mixed-targets}. Projecting the
resulting factorization in \(H\) proves\[
             \bigl((xZ)^{{}^2E_6(5)}\bigr)^2={}^2E_6(5),\qquad
                    |xZ|=601.
\]
\end{proof}

This proof handles the central-isogeny issue directly: it never identifies
\({\bf G}_{\rm sc}^F\) with \({\bf G}_{\rm ad}^F\), and the only possible
factor-three rational splitting is explicitly absorbed in \eqref{eq:2e43}.

\subsection{\texorpdfstring{The order-331 class in \(E_8(2)\)}{route_e8_2_complete.md}}
\label{candidates/route_e8_2_complete.md}

\begin{theorem}\label{e_82}
Let \(G\) be the finite simple group \(E_8(2)\), let \(T\) be a Coxeter torus of $G$, and let \(x\) be a generator of \(T\). 
Then
\begin{equation}
\label{eq:e812}
                              (x^G)^2=G.                        
\end{equation}
\end{theorem}

\subsubsection{The characteristic-two long-root estimate}\label{the-characteristic-two-long-root-estimate}

Let \(P=QL\) be the long-root parabolic of \(G\). Since
\(\mathbb F_2^\times=1\), its Levi factor is \(L=L_0\cong E_7(2)\).

\begin{lemma}\label{lem:21}
Let \(V\) be a complex \(G\)-module with character \(\chi\),
and let \(g\in QL\) have nonidentity Levi image \(\bar g\in L_0\). If
\(h=\dim V^Q\), then
\begin{equation}
\label{eq:e821}
        |\chi(g)|\le h+{\dim V-h\over 2^6-2^3+1}
                    =h+{\dim V-h\over57}.                      
\end{equation}
\end{lemma}

\begin{proof}
The long-root grading gives
\[
 Z(Q)=X_{\widetilde\alpha}\cong(\mathbb F_2,+),\qquad
 Q/Z(Q)\cong W_{56},                                            
\]
where \(W_{56}\) is the \(56\)-dimensional minuscule module for \(L_0\cong E_7(2)\).
The complementary-root commutator constants are all equal to \(\pm1\), and hence nonzero in characteristic \(2\).
Consequently, the commutator form on \(Q/Z(Q)\) is a nondegenerate
alternating form. Thus \(Q\) is an extraspecial group of order
\(2^{1+56}\).

Since \(Z(Q)\cong C_2\), the restriction of \(V\) to \(Q\) admits the
canonical decomposition
\(
V=V^{Z(Q)}\oplus [V,Z(Q)].
\)
Moreover,
\(
V^{Z(Q)}=V^Q\oplus V_1,
\)
where \(V_1\) is the sum of all \(Q\)-constituents lying above the
nontrivial linear characters of \(Q/Z(Q)\). Setting $V_2=[V,Z(Q)]$, 
we obtain a decomposition into \(QL\)-submodules
\[
V=V^Q\oplus V_1\oplus V_2,
\]
Let \(\chi_0,\chi_1,\chi_2\) denote the corresponding characters. Hence
\(
 \chi=\chi_0+\chi_1+\chi_2 .
\)

The contribution of the fixed-point part satisfies
\[
 |\chi_0(g)|\leq \chi_0(1)=\dim V^Q=h .
\]

We next consider \(V_1\). The irreducible constituents of \(V_1\)
correspond to the nontrivial linear characters of
\(
 Q/Z(Q)\cong W_{56},
\)
and the action of \(L_0\) on these characters is the natural minuscule
action of \(E_7(2)\). Therefore the normalized absolute trace on this
part is bounded by the fixed-point ratio of \(\bar g\). By the uniform
fixed-point-ratio estimate for this action,
\[
 |\chi_1(g)|
 \leq
 \frac{1}{q^6-q^3+1}\chi_1(1)
 =
 \frac1{57}\chi_1(1).
\]

It remains to consider \(V_2\). This is the sum of the Heisenberg
constituents arising from the nontrivial character of \(Z(Q)\). Let
\(\omega\) be one such constituent. The extraspecial group character
formula gives the following. If \(C\) is the inverse image in \(Q\) of
\(
 C_{Q/Z(Q)}(\bar g)=C_{W_{56}}(\bar g),
\)
then
\(
 \omega(g)=0
\)
unless \(g\) acts trivially on \(C\), while in the latter case
\(
 |\omega(g)|^2=|C_{W_{56}}(\bar g)|.
\)
Since
\(
 \omega(1)=2^{28},
\)
we obtain
\[
 \frac{|\omega(g)|}{\omega(1)}
 =
 2^{-\frac12(56-\dim C_{W_{56}}(\bar g))}
 =
 2^{-\frac12\dim[W_{56},\bar g]} .
\]

Because \(g\) is not unipotent and \(\bar g\neq1\), the semisimple part of
\(\bar g\) is nontrivial. Choose a root \(\alpha\) on which this
semisimple element acts nontrivially. The weights of the minuscule module
\(W_{56}\) contain twelve nontrivial two-element \(\alpha\)-strings. In
each such string the two weights have distinct eigenvalues. Hence any
eigenspace of \(\bar g\) omits at least one weight from each of these
twelve strings. Therefore
\(
 \dim[W_{56},\bar g]\geq12 .
\)
It follows that every Heisenberg constituent satisfies
\[
 \frac{|\omega(g)|}{\omega(1)}
 \leq
 2^{-6}
 <
 \frac1{57}.
\]
Consequently,
\(
 |\chi_2(g)|
 \leq
 \frac1{57}\chi_2(1).
\)

Finally,
\(
 \chi_1(1)+\chi_2(1)
 =
 \dim V-\dim V^Q
 =
 \dim V-h .
\)
Therefore
\[
\begin{aligned}
 |\chi(g)|
 &\leq
 |\chi_0(g)|+|\chi_1(g)|+|\chi_2(g)|\\
 &\leq
 h+\frac1{57}(\chi_1(1)+\chi_2(1))\\
 &=
 h+\frac{\dim V-h}{57}.
\end{aligned}
\]
This proves
\[
 |\chi(g)|
 \leq
 h+\frac{\dim V-h}{2^6-2^3+1}.
\]
\end{proof}

\begin{proof}
We divide the proof into several steps.

\proofpart{The exact source column}\label{the-exact-source-column}
The Smith form of \(2I-w\), for a Coxeter element \(w\in W(E_8)\), is
\(\operatorname {diag}(1,\ldots ,1,331)\). The rational Coxeter
automizer has order 30, and its fifteenth power acts on \(T\) as inversion.
Consequently \(x\) is regular semisimple and \(x^G=(x^{-1})^G\).

The Coxeter Deligne--Lusztig column has precisely the following thirty
unipotent constituents (CHEVIE numbering):
\[
\begin{split}
\mathcal S={}&[1,2,3,4,7,68,69,70,71,113,114,115,116,117,120,\\
&127,128,133,142,145,146,148,151,152,155,156,160,161,162,163].
\end{split}                                                     
\]
Every one has value \(\pm1\) at \(x\). Since 331 is a primitive
degree-30 prime, every nonidentity element of the dual Coxeter torus is
regular. Hence every remaining source character lies in a singleton regular
Lusztig series, and all such characters have common degree\[
 D={\prod_{d\in\{2,8,12,14,18,20,24,30\}}(2^d-1)\over331}
  =767782630093070853154958551396284375.                         
\]
Column orthogonality gives total squared source mass \(331-30=301\) for
these regular-series characters.
Thus for every \(g\in G\)\[
 F_x(g):=1+\sum_{1_G\ne\chi}
       {|\chi(x)|^2\chi(g^{-1})\over\chi(1)}                   
\]
is the normalized Frobenius structure-constant sum. It is enough to prove
\(|F_x(g)-1|<1\).

The largest centralizer of a nonidentity element of \(E_8(q)\) has order at
most \(q^{190}\). Column orthogonality therefore gives, uniformly in
\(g\ne1\),
\begin{equation}
\label{eq:e816}
 {301\,|C_G(g)|^{1/2}\over D}\le {301\,2^{95}\over D}<10^{-4}   
 \end{equation}
for the complete regular-series contribution.

For the 28 characters indexed by
\(\mathcal S\setminus\{1,2\}\), put\[
 h_i=\dim{}^*R_L^G(\chi_i),\qquad H=\sum_i{h_i\over\chi_i(1)}.
\]
Exact Lusztig restriction gives
\[
H={900764007669538802644357\over
3461643639276572191490048000}                                 
\]
and hence
\begin{equation}
\label{eq:e826}
                 H+{28-H\over57}<\frac12.                      
\end{equation}

\proofpart{A complete mixed-element dichotomy}\label{a-complete-mixed-element-dichotomy}
Set \(B=10^{17}\). Write a mixed element as \(g=su=us\), with
\(s\ne1\) semisimple and \(u\ne1\) unipotent.

If \(|C_G(s)|\le B\), column orthogonality and
\(C_G(g)\le C_G(s)\) give 
\begin{equation}
\label{eq:e831}
\begin{split}
&\left|\sum_{i\in\mathcal S\setminus\{1,2\}}
       {\chi_i(g^{-1})\over\chi_i(1)}\right|
 +{301|C_G(g)|^{1/2}\over D}\\
&\le \sqrt B\left(
 {696409539980666385533794576034016493524738675571479659\over
  380187543704019489449258331686714799247470160718306822062080000}
 +{301\over D}\right)\\
 &<0.58.                                  
\end{split}
\end{equation}

It remains to justify parabolic placement when \(|C_G(s)|>B\). The
complete list of rational semisimple-centralizer types in characteristic 2,
specialized at \(q=2\), has exactly 23 noncentral entries above this bound.
After suppressing torus factors, their semisimple parts are\[
 E_6,{}^2E_6,E_7,E_6A_1,E_6A_2,{}^2E_6A_1,{}^2E_6{}^2A_2,
 A_7,{}^2A_7,A_8,{}^2A_8,D_6,{}^2D_6,D_7,{}^2D_7.              
\]
The standalone certificate enumerates all 707 twisted centralizer forms,
checks that these are precisely the 23 large entries, and finds an
\(F\)-fixed ambient root in each.

Here is the corresponding intrinsic placement argument. The algebraic
centralizer \(\mathbf C=C_{\mathbf G}(s)\) is connected. Put \(u\) in the
unipotent radical of an \(F\)-stable Borel of \(\mathbf C\). Every type in
(3.2) has an \(F\)-stable irreducible component whose highest root is fixed
by its graph automorphism: this is immediate in the untwisted cases; the
diagram involutions of \({}^2E_6,{}^2A_7,{}^2D_6,{}^2D_7\) fix a root, and
in \({}^2A_8\) the highest root \(\alpha_1+\cdots+\alpha_8\) is fixed.
The corresponding rational ambient root subgroup \(X_\beta\) lies in the
centre of that Borel's unipotent radical. Hence both \(s\) and \(u\)
centralize \(X_\beta\), so\[
                         g\in N_G(X_\beta)=P_\beta,            
\]
a rational long-root parabolic.

The Levi image of \(g\) in \(E_7(2)\) is nonidentity. Otherwise
\(g\in Q_\beta\), a 2-group, contradicting the nontrivial odd-order
semisimple part \(s\). Lemma \ref{lem:21} and \eqref{eq:e826} therefore apply. The Steinberg
character vanishes on mixed elements, and \eqref{eq:e816} gives
\begin{equation}
\label{eq:e834}
                              |F_x(g)-1|<0.5001<1.              
\end{equation}
Together with \eqref{eq:e831}, this covers every mixed target.

\proofpart{Unipotent and semisimple targets}\label{unipotent-and-semisimple-targets}
For every nonidentity unipotent \(v\), the exact Green functions give
\begin{equation}
\label{eq:e841}
 \max_{v\ne1}\left|\sum_{i\in\mathcal S\setminus\{1\}}
       {\chi_i(v^{-1})\over\chi_i(1)}\right|
 ={15158107432172\over885472467004125}<\frac1{50}.            
\end{equation}
Adding \eqref{eq:e816} proves \(|F_x(v)-1|<1\). Gow's theorem \cite{Gow} puts every
nonidentity semisimple element in the product of the two regular semisimple
classes \(x^G,x^G\). Finally, reality of \(x^G\) puts the identity in its
square. Equations \eqref{eq:e831}, \eqref{eq:e834}, and \eqref{eq:e841} prove \eqref{eq:e812}.
\end{proof}

The certificates \texttt{scripts/\allowbreak audit\_\allowbreak e8\_\allowbreak 2\_\allowbreak 4\_\allowbreak coxeter.g} and
\texttt{scripts/\allowbreak probe\_\allowbreak e8\_\allowbreak 2\_\allowbreak semisimple\_\allowbreak types.g} verify, respectively, every source
and Green/Lusztig-restriction value and the complete characteristic-two
centralizer dichotomy with its strict inequalities.

\subsection{\texorpdfstring{The Coxeter class in \(E_8(7)\)}{/route_e8_one_torus.md}}
\label{candidates/route_e8_one_torus.md}

\begin{theorem}\label{e_87}
Let \(G\) be the finite simple group \(E_8(7)\), and let \(x\) be a generator of the Coxeter torus. Then
\[(x^{G})^{2}=G.\]
\end{theorem}

Let \(G=E_8(7)\). Since type \(E_8\) has trivial fundamental group and
trivial algebraic centre, this is simultaneously the simply connected,
adjoint, and simple finite group. Let \(T=T_w^F\) be the Coxeter torus,
where \(w\) is a Coxeter element of \(W(E_8)\), and let \(x\) be a generator
of \(T\). Then
\[
 N:=|T|=\Phi _{30}(7)
 =7^8+7^7-7^5-7^4-7^3+7+1=6568801,
\]
and \(N\) is prime. The certificate
\texttt{scripts/\allowbreak audit\_\allowbreak e8\_\allowbreak 7\_\allowbreak coxeter\_\allowbreak source\_\allowbreak green.g} computes the Smith form of \(7I-w\), for the Coxeter element \(w\in W(E_8)\) is
\[
\operatorname {diag}(1,1,1,1,1,1,1,N),
 \qquad |C_W(w)|=30.
\]
Thus \(T\cong C_N\), and \cite[Proposition 2.11]{GM}, identifies
this as a self-centralizing maximal torus with rational Weyl group of order
30. Moreover \(\operatorname {ord}_N(7)=30\), and the fifteenth Coxeter
power acts on \(T\) as\[
              7^{15}\equiv-1\pmod N.
\]
The certificate obtains the same exponent \(6568800\) directly from the
Smith basis. Hence \(x^{-1}\) is conjugate to \(x\). Put \(C=x^G\); then
\(C\) is real and \(1\in C^2\).

We will prove
\[
                         C^2=E_8(7).
\]

\begin{proof}
\proofpart{Exact source column}\label{exact-source-column}
For a target \(g\), Frobenius' normalized same-class sum is
\begin{equation}
\label{eq:e871}
 F_x(g)=1+\sum_{1_G\ne\chi\in\operatorname {Irr}(G)}
       \frac{|\chi(x)|^2\chi(g^{-1})}{\chi(1)}.             
\end{equation}
Reality of \(C\) is important here: \(\chi(x)\) is real, so the same-class
coefficient \(\chi(x)^2\) is \(|\chi(x)|^2\).

Lemmas 3.1--3.2 and the proof of Proposition 3.5 of \cite{GM}, imply
that \(\chi(x)\ne0\) only if \(\chi\in\mathcal E(G,s)\) for a semisimple
parameter \(s\in T^*\). There are exactly two strata.

\begin{enumerate}
\def\labelenumi{\arabic{enumi}.}
\item
  For \(s=1\), the Coxeter Deligne--Lusztig row has the following 30
  supported unipotent characters in CHEVIE order:
\begin{equation}
\label{eq:e872}
  \begin{split}
  \mathcal S={}&[1,2,3,4,7,68,69,70,71,113,114,115,116,117,120,\\
      &127,128,133,142,145,146,148,151,152,155,156,160,161,162,163].
  \end{split}                                               
\end{equation}
  Their Deligne--Lusztig coefficients are
\begin{equation}
\label{eq:e873}
\resizebox{\linewidth}{!}{$\displaystyle
[1,1,1,1,1,-1,-1,-1,-1,-1,1,-1,1,1,1,-1,-1,1,1,1,-1,
    1,1,-1,1,1,1,1,1,1].
$}
\end{equation}
  Thus each source value has square one and the unipotent source mass is
  exactly 30. In \eqref{eq:e871} the signs in \eqref{eq:e873} disappear.
  \item
  Since \(N\) is prime and \(\operatorname {ord}_N(7)=30\), every
  \(1\ne s\in T^*\) is regular. Its rational Lusztig series is a
  singleton semisimple character of degree
  \[
  D=\frac{\prod_{d\in\{2,8,12,14,18,20,24,30\}}(7^d-1)}{N}
  \]
\begin{equation}
\label{eq:e874}
\resizebox{\linewidth}{!}{$\displaystyle
=221874861125239297913749723122491086764215247890536702696915825839896964362458448039025573888000000000.
$}
\end{equation}
  The nonzero dual elements form
  \((N-1)/30=218960\) rational orbits. Column orthogonality gives the
  exact total squared source-value mass of these singleton series:
\begin{equation}
\label{eq:e875}
    \sum_{s\ne1}|\chi_s(x)|^2=N-30=6568771,                
\end{equation}
  since \(\sum_\chi|\chi(x)|^2=|C_G(x)|=N\) and \eqref{eq:e872} has mass 30.
\end{enumerate}

There are no nonregular dual-torus strata and no other source characters.

For later use, \cite[proof of Theorem 7.7]{GM} records
\(|C_G(g)|\le7^{190}\) for every \(g\ne1\). Therefore all the regular
singleton series together cost at most
\begin{equation}
\label{eq:e876}
\resizebox{\linewidth}{!}{$\displaystyle
R:=\frac{(N-30)7^{95}}D
 =\frac{114922545782354382511415553426765905336365699973884988935984165525625862065106657775823}
 {20170441920476299810340883920226462433110477080957882063355984167263360396587131639911415808000000000}
 <5.698\cdot10^{-15}.
$}   
\end{equation}

\proofpart{Nonidentity unipotent targets: exact Green functions}\label{nonidentity-unipotent-targets-exact-green-functions}
Let \(u\ne1\) be unipotent. The Steinberg row 2 vanishes at \(u\), and
CHEVIE's exact Green table gives\[
 U(u):=\sum_{i\in\mathcal S\setminus\{1\}}
       \frac{\chi_i(u^{-1})}{\chi_i(1)}.
\]
The certificate evaluates every rational unipotent column, locates the
identity by comparison with all 30 degrees (rather than assuming it is
column 1), and proves
\begin{equation}
\label{eq:e877}
 \max_{u\ne1}|U(u)|
 =\frac{3355396920359143057730729392390218383}
 {394588039680250201088563695743958120900000}
 <8.504\cdot10^{-6}.                                         
\end{equation}
The maximum is at rational column 112, represented in the CHEVIE class
pair notation by \([69,1]\). Algebraic unipotent class 69 is the long-root
class \(A_1\); it is not a near-regular class. The identity is column 113,
the pair \([70,1]\).

For completeness, \cite[Proposition 3.2(i)]{LiebeckTiep}, applies because
\(G=E_8(7)\) is in good odd characteristic: \(A_1\) is its long-root
alternative, and every other nonidentity unipotent element has a conjugate
in \(Ql\) for the long-root parabolic \(P=QL\), with the projection of
\(l\) to the \(E_7(7)\)-Levi nontrivial and not long-root. Thus the
placement theorem includes every nonidentity unipotent class. We do not
need to estimate them through that placement: \eqref{eq:e877} is stronger and exact.

Equations \eqref{eq:e876}--\eqref{eq:e877} give
\begin{equation}
\label{eq:e878}
              |F_x(u)-1|\le |U(u)|+R<8.504\cdot10^{-6}<1.   
\end{equation}

\proofpart{Mixed targets: complete good-characteristic dichotomy}\label{mixed-targets-complete-good-characteristic-dichotomy}
Let \(g\) be mixed. \cite[Lemma 4.1]{LiebeckTiep}, applies to simply connected
\(E_8(7)\) because 7 is a good prime for \(E_8\). Alternatives (i) and
(iv) of that lemma are respectively unipotent and semisimple, so a mixed
element satisfies at least one of:

\begin{itemize}
\tightlist
\item
  \(|C_G(g)|<7^{38}\); or
\item
  after conjugacy, \(g\in Ql\) in the long-root parabolic, and the
  projection of \(l\) to \(L/Z(L)\cong E_7(7)\) is nontrivial and not a
  long-root element.
\end{itemize}

This is an exhaustive (not necessarily disjoint) pair of alternatives for
all mixed targets; no class-type enumeration is being assumed. We apply the
small-centralizer estimate whenever the first alternative holds, and the
parabolic estimate otherwise; in the latter case the second alternative must
hold.

In the small-centralizer case, column orthogonality, \eqref{eq:e874}, and \eqref{eq:e875} give
\begin{equation}
\label{eq:e879}
 |F_x(g)-1|<
 7^{19}\sum_{i\in\mathcal S\setminus\{1,2\}}\frac1{\chi_i(1)}
 +\frac{(N-30)7^{19}}D
 <3.541\cdot10^{-9}<1.                                    
\end{equation}
The first term is exactly
\begin{equation}
\label{eq:e8710}
\resizebox{\linewidth}{!}{$\displaystyle
\frac{276154987234217658371782238509997487359644418617595067759718870662720189661910872387580488160144541025683859434126312215201669777751697556513703446398339}
{78007612111495002227844450991898589700222838318583789070259976862161137120871643494157952674806706143962927346145864046679856384683291535258690041308774400000000},
$}
\end{equation}
and the regular term is exactly
\begin{equation}
\label{eq:e8711}
\resizebox{\linewidth}{!}{$\displaystyle
\frac{6806975647792611447023}
{20170441920476299810340883920226462433110477080957882063355984167263360396587131639911415808000000000}.
$}
\end{equation}

It remains to justify the parabolic case for every source row. Put
\(h_i=\dim({}^*R_L^G\chi_i)\), where \(L\) is the long-root Levi of type
\(E_7T_1\). The second certificate computes the Lusztig restriction of
all 28 rows
\begin{equation}
\label{eq:e8712}
\resizebox{\linewidth}{!}{$\displaystyle
\mathcal S\setminus\{1,2\}
 =[3,4,7,68,69,70,71,113,114,115,116,117,120,127,128,133,
   142,145,146,148,151,152,155,156,160,161,162,163]
$}
\end{equation}
and proves
\begin{equation}
\label{eq:e8713}
\resizebox{\linewidth}{!}{$\displaystyle
H:=\sum_{i\in\mathcal S\setminus\{1,2\}}\frac{h_i}{\chi_i(1)}
 =\frac{11828793456364627659160307707483599710526396551389508284817286908579}
 {163657537398292409140924414782666535564133663456615744495647453152807322350000}.
$}
\end{equation}
In particular, it computes rather than omits the six zero-restriction rows
\(155,156,160,161,162,163\).

Proposition 3.3 and Theorem 3.1(iii)--(iv), with Table 3.2 of \cite{LiebeckTiep},
give \(1.5/7^{10}\) on the \(V_1\)-part and \(1/7^{10}\) on the
\(V_2\)-part for a non-long-root Levi projection. Bounding both by the
first number, the complete contribution of all 28 unipotent source rows is
at most
\begin{equation}
\label{eq:e8714}
\resizebox{\linewidth}{!}{$\displaystyle
P:=H+\frac{3}{2\cdot7^{10}}(28-H)
 =\frac{2750783164834997665026836829773536832449614294959599266655540038411874760639321}
 {18491681450923784178757000062365204206978404821675893249491357498222033371935806060000}
 <1.488\cdot10^{-7}.
$}
\end{equation}
The Steinberg character is absent because it vanishes on a mixed target.
Combining \eqref{eq:e876} and \eqref{eq:e8714},
\begin{equation}
\label{eq:e8715}
                     |F_x(g)-1|\le P+R<1.488\cdot10^{-7}<1. 
\end{equation}

\proofpart{Semisimple targets and conclusion}\label{semisimple-targets-and-conclusion}
Gow's theorem says that the product of two regular semisimple conjugacy
classes contains every nonidentity semisimple element. Apply it to the
two copies of the regular semisimple class \(C\). Thus all nonidentity
semisimple targets lie in \(C^2\). The identity lies in \(C^2\) because
\(C=C^{-1}\); nonidentity unipotent targets satisfy \eqref{eq:e878}; and all mixed
targets satisfy either \eqref{eq:e879} or \eqref{eq:e8715}. These cases exhaust \(G\), proving\[
      (x^{E_8(7)})^2=E_8(7),\qquad |x|=6568801.
\]
\end{proof}

\subsubsection{Reproducible certificates}\label{reproducible-certificates-1}

Run with the local GAP3/CHEVIE 2024-01-07 distribution:

\begin{Verbatim}[breaklines=true,breakanywhere=true,fontsize=]
scripts/audit_e8_7_coxeter_source_green.g, scripts/audit_e8_7_coxeter_hc.g
\end{Verbatim}
The first script asserts the torus Smith form, reality exponent, exact 30
source rows and signs, source masses, all rational Green columns, \eqref{eq:e877}, the
regular degree, and \eqref{eq:e876}. The second computes all 28 Harish--Chandra rows,
asserts \eqref{eq:e8713}--\eqref{eq:e8714}, and independently asserts the exact small-centralizer
terms \eqref{eq:e8710}--\eqref{eq:e8711}. 

This candidate proves only \(E_8(7)\). It makes no claim for
\(q\in\{2,3,4,5,8\}\), where the good-characteristic mixed-target
dichotomy used above is unavailable or requires a separate audit.

\section{Completion of the proof}

\begin{proof}[Proof of Theorem~\ref{thm:main}]
The classification list is exhausted by $A_l$ in Theorem \ref{typeA}; both types of ${}^2A_l$ in Theorems \ref{evenu}, \ref{psu_82}, \ref{psu_43}, \ref{oddu}, \ref{psu_63};
$B_l,C_l$ in Theorems \ref{orthogonal_split}, \ref{Symp}; both types of split $D_l$ in Theorem \ref{orthogonal_split}; minus ${}^2D_l$ in Theorem \ref{orthogonal_minus};
$G_2,{}^3D_4,{}^2B_2,{}^2G_2,{}^2F_4$ in \cite[Theorems 7.1, 7.3 and Remark 7.4]{GM}; and
$F_4,E_6,{}^2E_6,E_7,E_8$ in Theorems \ref{e7e8}, \ref{e_64-2e_64}, \ref{f_43}, \ref{f_4q}, \ref{e_63}, \ref{2e_63}, \ref{e_65-2e_67}, \ref{2e_65}, \ref{e_82}, \ref{e_87}.  For each row of Tables~\ref{tab:classical} and \ref{tab:exceptional}, $q\in\{2,3,4,5,7,8\}$, is either at or above the inclusive Ellers--Gordeev threshold,
is assigned an explicit matrix, root, torus, or table class in the preceding
sections, or is a nonsimple parameter listed below the table.

For an Ellers--Gordeev row, the cited same-class theorem and reality give
the result after central projection.  For the residual classical and
relative exceptional rows, the fixed-point-free grade calculation, exact
Levi class product, fusion, and reality hypotheses give the conclusion by
Lemma~\ref{lem:relative-gauss} or the stated cyclic matrix theorem.  For
the remaining exceptional rows, the exact bounds are all strict in
\eqref{eq:strict}; Gow's results cover the nonidentity semisimple elements; the identity is accounted for by reality.  The low-rank identifications assign every repeated
name to a proved row.  Hence no family,  field, or simple low-rank
form remains uncovered, and every group in the theorem has the displayed
class $C$ with $C^2=G$.
\end{proof}

\section*{Acknowledgement}
In the course of this research, the authors utilized ChatGPT to explore alternative proof strategies, employed AI-assisted computations, specifically invoking GAP and Python for character calculations, and used AI tools for English-language polishing. Every mathematical claim and proof step suggested by the AI was independently verified and rigorously checked by the authors before inclusion.

This work was supported by the NSF of China  (Nos. 12431001, 12571017).

\end{document}